\documentclass[11pt,letterpaper,twoside,reqno,nosumlimits,final]{amsart}

\usepackage[usenames,dvipsnames]{xcolor}
\usepackage{fancyhdr}
\usepackage{amsmath,amsfonts,amsbsy,amsgen,amscd,mathrsfs,amssymb,amsthm}
\usepackage{stackengine}
\usepackage{pifont}

\usepackage{mathtools}

\usepackage{graphicx}
\graphicspath{{figures/}}

\usepackage{tikz}
\usetikzlibrary{arrows,chains,matrix,positioning,scopes}
\usepackage{tikz-3dplot}

\usepackage{xcolor} %

\definecolor{cit}{rgb}{0.91,0.39,0.16}	%

\definecolor{dark-gray}{gray}{0.3}

\definecolor{dkgray}{rgb}{.3,.3,.3}
\definecolor{medgray}{rgb}{.5,.5,.5}
\definecolor{ltgray}{rgb}{.7,.7,.7}
\definecolor{dkblue}{rgb}{0,0,.5}
\definecolor{medblue}{rgb}{0,0,.75}
\definecolor{ltblue}{rgb}{0.97,0.97,1}
\definecolor{rust}{rgb}{0.5,0.1,0.1}
\definecolor{ltyellow}{rgb}{1, 1, 0.9}

\usepackage{soul}
\sethlcolor{ltyellow}
\newcommand{\hilite}[1]{\hl{#1}}

\usepackage[normalem]{ulem}

\usepackage[
style=alphabetic, citestyle=alphabetic,
giveninits=true,
sorting=nyt, sortcites=false,
autopunct=true, autolang=hyphen,
hyperref=true,
abbreviate=false,
backref=false,
backend=bibtex
]{biblatex}
\defbibheading{bibempty}{}

\renewbibmacro{in:}{%
  \ifentrytype{article}{}{\printtext{\bibstring{in}\intitlepunct}}}

\AtEveryBibitem{%
	\clearfield{issn}%
	\clearfield{isbn}%
	\clearlist{location}%
	\ifentrytype{book}%
		{%
		\clearfield{series}%
		\clearfield{volume}%
		\clearfield{pages}}%
		{%
		\clearname{editor}}%
}

\renewbibmacro*{doi+eprint+url}{%
  \newunit\newblock
  \iftoggle{bbx:doi}
    {\printfield{doi}}
    {}%
  \iftoggle{bbx:eprint}
    {\iffieldundef{doi}{\usebibmacro{eprint}}{}}
    {}%
  \newunit\newblock
  \iftoggle{bbx:url}
    {\iffieldundef{doi}{\iffieldundef{eprint}{\usebibmacro{url}}}} %
    {}}

\AtBeginBibliography{\footnotesize}				%

\usepackage{url} %
\makeatletter
\g@addto@macro{\UrlBreaks}{\UrlOrds}
\makeatother

\usepackage{hyperref}

\hypersetup{hidelinks,colorlinks=true,breaklinks=true,urlcolor=cit,citecolor=dkblue,linkcolor=dkblue,bookmarksopen=false,hypertexnames=false}

\usepackage{bookmark}
\bookmarksetup{
open,
numbered,
addtohook={%
\ifnum\bookmarkget{level}=0 %
\bookmarksetup{bold}%
\fi
\ifnum\bookmarkget{level}=-1 %
\bookmarksetup{color=cit,bold}%
\fi
}
}

\usepackage[full]{textcomp}

\usepackage[scaled=.98,sups,osf]{XCharter}%
\usepackage[scaled=1.04,varqu,varl]{inconsolata}%
\usepackage[type1,scaled=0.98]{cabin}%
\usepackage[utopia,vvarbb,scaled=1.07]{newtxmath}
\usepackage[cal=euler]{mathalfa}
\usepackage{bm} %

\usepackage{microtype} %
\usepackage[utf8]{inputenc} %
\usepackage[T1]{fontenc} %

\usepackage[english]{babel} %
\usepackage{setspace}

\usepackage{enumitem}

\setlist{noitemsep} %

\setlist[enumerate]{font=\sffamily\bfseries\footnotesize\textcolor{dkgray},label=\arabic*.}
\setlist[itemize]{font=\textcolor{dkgray},label=\small\textcolor{dkgray}\textbullet}

\usepackage[capitalize]{cleveref}

\numberwithin{equation}{section}

\theoremstyle{theorem}

\newtheorem{theorem}{Theorem}[section]
\newtheorem{lemma}[theorem]{Lemma}

\newtheorem{proposition}[theorem]{Proposition}
\newtheorem{fact}[theorem]{Fact}

\newtheorem{conjecture}[theorem]{Conjecture}

\theoremstyle{definition}

\newtheorem{definition}[theorem]{Definition}
\newtheorem{example}[theorem]{Example}
\newtheorem{remark}[theorem]{Remark}

\newtheorem*{ideaT}{Theme}
\newtheorem*{problemT}{Problem}

\Crefname{proposition}{Proposition}{Propositions}
\Crefname{figure}{Figure}{Figures}

\RequirePackage[framemethod=default]{mdframed}

\newmdenv[skipabove=6pt,
skipbelow=6pt,
rightline=false,
leftline=true,
topline=false,
bottomline=false,
backgroundcolor=ltyellow,
linecolor=cit,
innerleftmargin=10pt,
innerrightmargin=10pt,
innertopmargin=0pt,
innerbottommargin=5pt,
leftmargin=0cm,
rightmargin=0cm,
linewidth=4pt]{iBox}

\newmdenv[skipabove=0pt,
skipbelow=0pt,
backgroundcolor=ltblue,
linecolor=dkblue,
linewidth=2pt,
rightline=false,
leftline=false,
topline=false,
bottomline=false,
innerleftmargin=7pt,
innerrightmargin=10pt,
innertopmargin=6pt,
innerbottommargin=6pt,
leftmargin=0cm,
rightmargin=0cm,
innerbottommargin=5pt]{aBox}

\newmdenv[skipabove=10pt,
skipbelow=10pt,
backgroundcolor=white,
linecolor=dkblue,
linewidth=0.5pt,
rightline=true,
leftline=true,
topline=true,
bottomline=true,
innerleftmargin=10pt,
innerrightmargin=0.5in,
innertopmargin=5pt,
innerbottommargin=5pt,
leftmargin=0cm,
rightmargin=0cm]{lfBox}

\usepackage{float}
\floatstyle{plaintop}

\numberwithin{figure}{section}
\numberwithin{table}{section}

\newfloat{recipe}{thp}{lor}
\floatname{recipe}{Recipe}
\numberwithin{recipe}{section}

\usepackage{caption}
\usepackage{subcaption}

\usepackage{booktabs,longtable,tabu} %
\usepackage{multirow} %

\newcommand{\econst}{\mathrm{e}}
\newcommand{\iunit}{\mathrm{i}}

\newcommand{\eps}{\varepsilon}

\renewcommand{\phi}{\varphi}

\newcommand{\onecirc}{\text{\ding{192}}}
\newcommand{\twocirc}{\text{\ding{193}}}

\newcommand{\vct}[1]{\bm{#1}}
\newcommand{\mtx}[1]{\bm{#1}}
\newcommand{\set}[1]{\mathsf{#1}}

\newcommand{\coll}[1]{\mathcal{#1}}

\newcommand{\term}[1]{\hilite{\textit{#1}}}

\newcommand{\N}{\mathbb{N}}

\newcommand{\R}{\mathbb{R}}

\newcommand{\C}{\mathbb{C}}
\newcommand{\F}{\mathbb{F}}

\newcommand{\M}{\mathbb{M}}
\newcommand{\Sym}{\mathbb{H}}
\newcommand{\comp}{\textsf{c}}

\newcommand{\indicator}{\mathbb{1}}

\renewcommand{\Re}{\operatorname{Re}}
\renewcommand{\Im}{\operatorname{Im}}

\newcommand{\trace}{\operatorname{Tr}}

\newcommand{\Id}{\mathbf{I}}

\newcommand{\psdle}{\preccurlyeq}

\newcommand{\abs}[1]{\vert {#1} \vert}
\newcommand{\norm}[1]{\Vert {#1} \Vert}
\newcommand{\ip}[2]{\langle {#1}, \ {#2} \rangle}

\newcommand{\abssq}[1]{\abs{#1}^2}

\newcommand{\normsq}[1]{\norm{#1}^2}
\newcommand{\fnorm}[1]{\norm{#1}_{\mathrm{F}}}
\newcommand{\fnormsq}[1]{\norm{#1}_{\mathrm{F}}^2}

\newcommand{\lnorm}[1]{\left\Vert {#1} \right\Vert}

\newcommand{\diff}{\mathrm{d}}
\newcommand{\Diff}{\mathrm{D}}
\newcommand{\idiff}{\,\diff}

\newcommand{\Expect}{\operatorname{\mathbb{E}}}
\newcommand{\Var}{\operatorname{Var}}
\newcommand{\Cov}{\operatorname{Cov}}

\newcommand{\Probe}{\mathbb{P}}
\newcommand{\Prob}[1]{\Probe\left\{ #1 \right\}}

\newcommand{\condbar}{\, \vert \,}
\newcommand{\lcondbar}{\, \big\vert \,}

\newcommand{\Varo}{\mathsf{Var}}
\newcommand{\Mo}{\mathsf{Mom}}

\newcommand{\normal}{\textsc{normal}}
\newcommand{\uniform}{\textsc{uniform}}
\newcommand{\bernoulli}{\textsc{bernoulli}}
\newcommand{\binomial}{\textsc{binomial}}
\newcommand{\poisson}{\textsc{poisson}}

\newcommand{\multinomial}{\textsc{multinomial}}

\newcommand{\goe}{\mathrm{goe}}
\newcommand{\gue}{\mathrm{gue}}

\newcommand{\mgf}{\mathrm{mgf}}

\newcommand{\supp}{\operatorname{supp}}

\DeclareFontFamily{U}{matha}{\hyphenchar\font45}
\DeclareFontShape{U}{matha}{m}{n}{
  <-6> matha5 <6-7> matha6 <7-8> matha7
  <8-9> matha8 <9-10> matha9
  <10-12> matha10 <12-> matha12
  }{}

\DeclareSymbolFont{matha}{U}{matha}{m}{n}
\DeclareMathSymbol{\abscont}{3}{matha}{"CE}

\makeatletter
\def\paragraph{\@startsection{paragraph}{4}%
  \z@\z@{-\fontdimen2\font}%
  {\normalfont\scshape}}
\makeatother

\addglobalbib{Tro24-Second-Moment.bib} %

\allowdisplaybreaks

\usepackage[foot]{amsaddr}

\title[Comparison for random psd matrices]{Comparison theorems for the minimum eigenvalue \\ of a random positive-semidefinite matrix}

\author{Joel A.~Tropp}
\address%
{Department of Computing and Mathematical Science, Caltech, Pasadena, CA, USA.}
\email{jtropp@caltech.edu, https://tropp.caltech.edu}

\date{Research: August 2024.  Manuscript: 2 January 2025.  Revised: 27 January 2025.} %


\begin{document}

\begin{abstract}
This paper establishes a new comparison principle for the minimum
eigenvalue of a sum of independent random positive-semidefinite matrices.
The principle states that the minimum eigenvalue of the matrix sum
is controlled by the minimum eigenvalue of a Gaussian random matrix
that inherits its statistics from the summands.
This methodology is powerful because of the vast arsenal of tools
for treating Gaussian random matrices.  As applications, the paper
presents short, conceptual proofs of some old and new results in high-dimensional
statistics.  It also settles a long-standing open question
in computational linear algebra about the injectivity
properties of very sparse random matrices.
\end{abstract}

\maketitle

\section{Motivation}

Random positive-semidefinite (\term{psd}) matrices appear throughout %
high-dimensional statistics and high-dimensional probability.
In particular, random psd matrices model the sample covariance of a random vector,
and they capture properties of random linear embeddings.
For a psd matrix, the minimum eigenvalue %
provides a quantitative measure of invertibility,
so it is often the crucial statistic of these random matrix models.
This paper introduces a new technique for studying the minimum eigenvalue
of a random psd matrix by establishing a comparison with the minimum eigenvalue
of a Gaussian random matrix.  It also showcases several applications of this
methodology. %

\subsection{Intuition: Positive random walks}

Consider a random walk on the real line that can only move in the positive direction.
What is the probability that the random walk remains close to its origin?
This event occurs only when \hilite{all} of the increments are small, which is
very unlikely.  We can capture this insight with a standard probability
inequality that is the starting point for our investigation.

To model the positive random walk, we introduce
a sum of \hilite{nonnegative} real random variables
that are independent and identically distributed (\term{iid}):
\[
Y = \sum_{i=1}^n W_i
\quad\text{where the $W_i$ are iid copies of $W \geq 0$.}
\]
When the increment $W$ has two moments, we can compare the
moment generating function (\term{mgf}) for the lower tail
of the positive sum $Y$ with the mgf of a Gaussian real random variable.
For all $\theta \geq 0$,
\begin{equation} \label{eqn:intro-scalar-mgf}
\Expect[ \econst^{-\theta Y} ] \leq \Expect[ \econst^{- \theta Z} ]
\quad\text{where}\quad
Z \sim \normal_{\R}\left( n \cdot \Expect[ W ],\ n \cdot \Expect[ W^2 ] \right).
\end{equation}
See \cref{sec:scalar-pf} for a proof of~\eqref{eqn:intro-scalar-mgf}.
The mgf bound leads to a classic inequality for the lower tail:
\begin{equation} \label{eqn:intro-scalar-tail}
\Prob{ Y \leq \Expect[ Z ] - t } \leq
	\econst^{-t^2 / (2 \Var[Z])}
\quad\text{for $t \in [0,1]$.}
\end{equation}
In other words, the lower tail of the positive sum $Y$ is related
to the lower tail of a matching Gaussian random variable $Z$
that inherits its statistics from the summand $W$.
See \cref{fig:comparison-1d} for an illustration.

\begin{figure}
\begin{center}
\includegraphics[height=1.3in]{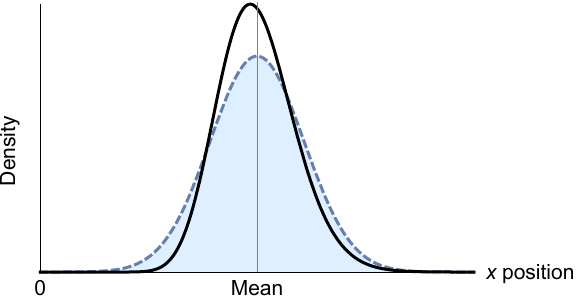} \hspace{0.1in}
\includegraphics[height=1.3in]{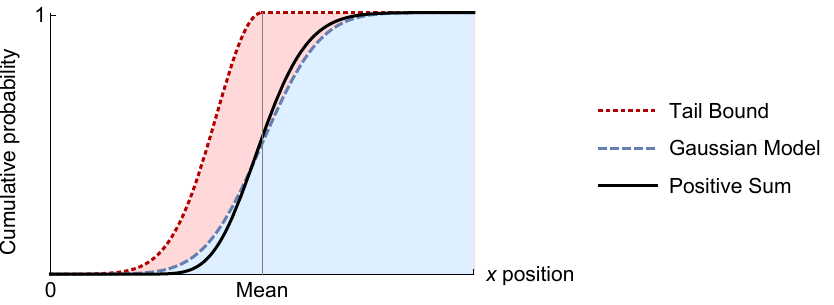}
\end{center}
\caption{\textbf{(Positive sum: Comparison).}  These plots illustrate the distribution $Y$ of an iid sum
of nonnegative real random variables {(black line)}, along with the matching Gaussian distribution $Z$ {(dashed blue)},
described by~\eqref{eqn:intro-scalar-mgf}, and the tail bound {(dotted red)}, described by~\eqref{eqn:intro-scalar-tail}.
The left-hand panel shows the densities; the right-hand panel shows the cumulative distributions.}
\label{fig:comparison-1d}
\end{figure}

\subsection{Random psd matrices}
\label{sec:intro-random-psd}

This paper demonstrates that the same phenomena
persist in the matrix setting.  Consider a sum of iid random psd matrices, either real or complex: 
\begin{equation} \label{eqn:iid-psd-intro}
\mtx{Y} = \sum_{i=1}^n \mtx{W}_i
\quad\text{where the $\mtx{W}_i$ are iid copies of a random psd matrix $\mtx{W}$.}
\end{equation}
The minimum eigenvalue $\lambda_{\min}(\mtx{Y})$ can be expressed
as the minimum of a family of iid positive sums:
\begin{equation} \label{eqn:iid-psd-sum-rayleigh}
\lambda_{\min}(\mtx{Y}) = \min\nolimits_{\norm{\vct{u}} = 1} \sum_{i=1}^n \vct{u}^* \mtx{W}_i \vct{u}.
\end{equation}
For each direction $\vct{u}$, the sum in~\eqref{eqn:iid-psd-sum-rayleigh}
is very unlikely to be zero because of~\eqref{eqn:intro-scalar-tail}.
On the other hand, the random summands $\mtx{W}_i$ must cover every direction $\vct{u}$
before the minimum eigenvalue $\lambda_{\min}(\mtx{Y})$ is strictly positive.
It is not clear which of these two opposing principles prevails.

To resolve this dilemma, we adapt the strategy behind
the scalar inequality~\eqref{eqn:intro-scalar-tail} to the matrix setting.
Construct a self-adjoint Gaussian random matrix $\mtx{Z}$
that inherits its statistics from the summands: %
\begin{equation} \label{eqn:intro-matrix-comp}
\begin{aligned}
\Expect[ \mtx{Z} ] &= n \cdot \Expect[ \mtx{W} ]; \\
\Var[ \trace[\mtx{MZ}] ] &= n \cdot \Expect[ \abssq{\trace[\mtx{MW}]} ]
\quad\text{for all self-adjoint $\mtx{M}$.}
\end{aligned}
\end{equation}
Inspired by~\eqref{eqn:intro-scalar-mgf},
we will establish a comparison between the trace mgfs of the two random matrices:
\begin{equation} \label{eqn:intro-matrix-mgf}
\Expect[ \trace \econst^{-\theta \mtx{Y}} ]
 	\leq 2 \Expect[ \trace \econst^{-\theta \mtx{Z}} ]
	\quad\text{for all $\theta \geq 0$.}
\end{equation}
\Cref{fig:comparison-2d} illustrates the comparison.
While the scalar case~\eqref{eqn:intro-scalar-mgf} is easy,
the matrix inequality~\eqref{eqn:intro-matrix-mgf} relies on a deep fact from matrix analysis
called Stahl's theorem, formerly the BMV conjecture~\cite{Sta13:Proof-BMV}.

Using standard methods from
high-dimensional probability~\cite{Tro15:Introduction-Matrix,vH16:Probability-High},
the trace inequality~\eqref{eqn:intro-matrix-mgf} leads to a
probabilistic comparison %
for the minimum eigenvalues:
\begin{equation} \label{eqn:intro-matrix-tail}
\Prob{ \lambda_{\min}(\mtx{Y}) \geq \Expect[ \lambda_{\min}(\mtx{Z}) ] - t }
	\leq 2d \cdot \econst^{-t^2 / (2 \sigma_*^2(\mtx{Z}))},
\end{equation}
where $d$ is the matrix dimension. The weak variance $\sigma_*^2(\mtx{Z}) \coloneqq
\max_{\norm{\vct{u}}=1} \Var[\vct{u}^* \mtx{Z} \vct{u}]$
controls the variance of $\lambda_{\min}(\mtx{Z})$.
The result~\eqref{eqn:intro-matrix-tail} is powerful enough to yield
\hilite{sharp, dimension-free} bounds for the minimum eigenvalue of some models.
See~\cref{thm:iid-intro} for the full statement of the comparison. %

To analyze the minimum eigenvalue $\lambda_{\min}(\mtx{Z})$ of the Gaussian matrix in~\eqref{eqn:intro-matrix-tail},
we have a bristling armamentarium of techniques at our disposal.
This methodology leads to short, conceptual proofs of
several important results from high-dimensional statistics (\cref{sec:scov}).
It also addresses a vexing open question from computational
linear algebra about very sparse random matrices (\cref{sec:subspace-inj}).

\begin{figure}[t]
\begin{center}
\includegraphics[height=2in]{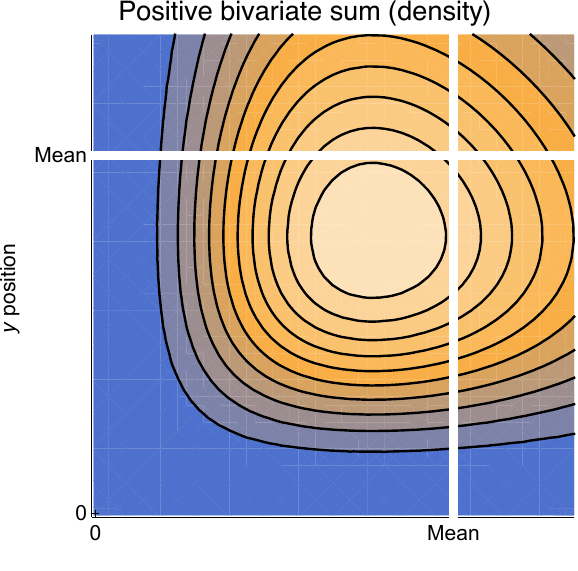} \hspace{0.25in}
\includegraphics[height=2in]{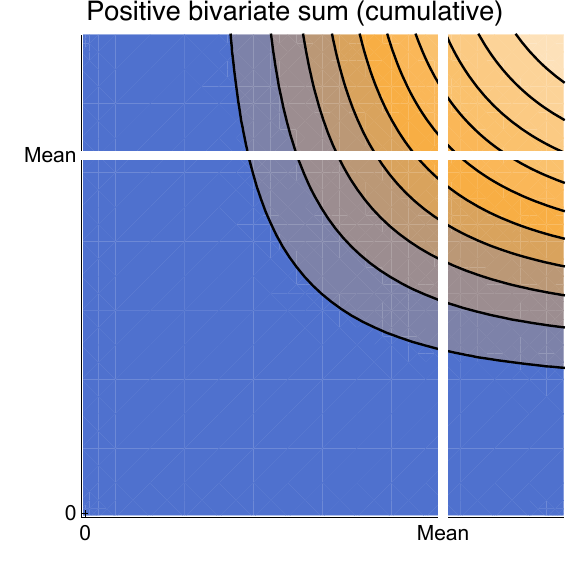} \\
\includegraphics[height=2in]{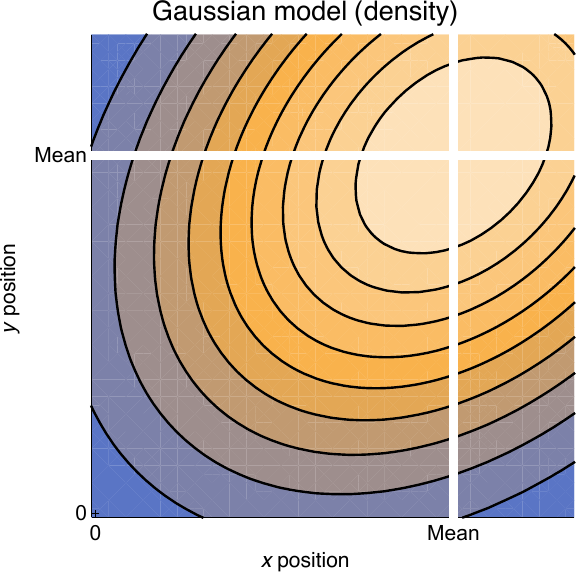} \hspace{0.25in}
\includegraphics[height=2in]{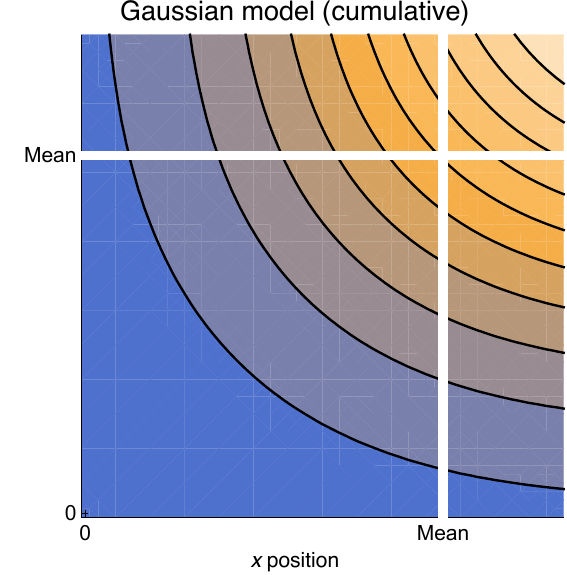}
\end{center}
\caption{\textbf{(Bivariate positive sum: Comparison).}
This figure illustrates the Gaussian comparison~\eqref{eqn:intro-matrix-comp}
for a $2 \times 2$ diagonal random matrix $\mtx{W}$ with independent entries.
\textbf{[Top]}
The bivariate distribution of a pair of independent positive sums.
\textbf{[Bottom]}
The distribution of the matching Gaussian model.
\textbf{[Left]}
Probability densities.
\textbf{[Right]}
Bivariate cumulative distribution functions.
Brighter colors correspond to higher probabilities.
The Gaussian comparison is valid southwest of the expectation (white lines).
More precisely, the comparison concerns %
the \hilite{minimum} of the two coordinates, whose distribution looks similar
to the univariate case (\cref{fig:comparison-1d}).}
\label{fig:comparison-2d}
\end{figure}

\subsection{Roadmap}

\Cref{sec:main-results} states two comparison theorems
and discusses related work.
\Cref{sec:gaussians} provides background on Gaussian random matrices
that aids in the analysis of the comparison model.
\Cref{sec:designs,sec:scov,sec:subspace-inj} apply the main results
to several examples. %
\Cref{sec:scalar-pf} details a proof of~\eqref{eqn:intro-scalar-tail}
that generalizes to matrices.
Last, \cref{sec:psd-weights,sec:iid} establish the main results,
including~\eqref{eqn:intro-matrix-mgf} and~\eqref{eqn:intro-matrix-tail}. %

\subsection{Notation}

We use standard conventions from linear algebra and probability.
Nonlinear functions bind before the trace and the expectation.
We omit brackets enclosing an argument unless required.

Our results hold in both the real ($\F = \R$) and complex ($\F = \C$) setting;
we often suppress the field.
For a natural number $d \in \N$, the set $\M_d(\F)$ is the linear
space of $d \times d$ square matrices with entries in the field $\F$.
The set $\Sym_d(\F)$ denotes the real-linear space of \term{self-adjoint} (i.e., conjugate symmetric)
matrices in $\M_d(\F)$.
As usual, $\trace[\cdot]$ returns the (unnormalized) trace of a square matrix.

The symbol ${}^\transp$ denotes the transpose of a vector or matrix,
while ${}^*$ denotes the conjugate transpose. %
The $\ell_2$ inner product is conjugate linear in the first coordinate:
$\ip{ \vct{u} }{ \vct{v} } \coloneqq \vct{u}^* \vct{v}$ for $\vct{u}, \vct{v} \in \F^d$.
Likewise, we employ the trace inner product $\ip{ \mtx{M} }{ \mtx{A} } \coloneqq \trace[ \mtx{M}^* \mtx{A} ]$
for matrices $\mtx{M}, \mtx{A}$ with the same dimensions.
The symbol $\norm{\cdot}$ denotes the $\ell_2$ norm on vectors
or the associated $\ell_2$ operator norm on matrices.
The Frobenius norm $\fnorm{\cdot}$ also appears regularly.

In heuristic discussions, we may employ the orders $\lesssim$ and $\gtrsim$,
which suppress universal constants.
Both $\asymp$ and $\mathcal{O}(\cdot)$ mean equivalent
within constant factors.  The symbol $\ll$ means
``much less than''.

\section{Main results and related work}
\label{sec:main-results}

This section states our two main comparison theorems.
The first concerns the sum of psd matrices with random weights.
The second theorem concerns a sum of iid random psd matrices
described in \cref{sec:intro-random-psd}.
Afterward, we present a simple first example,
and we discuss some related work.

\subsection{Gaussian comparison: Randomly weighted sum of psd matrices}
\label{sec:sampling-intro}

The first result treats a random matrix model where we randomly
weight the terms in a sum of fixed psd matrices.

\begin{theorem}[Comparison: Randomly weighted psd matrices] \label{thm:sampling-intro}
Fix a system of \hilite{psd} matrices $(\mtx{A}_1, \dots, \mtx{A}_n)$, real or complex, %
with common dimension $d$.
Consider an \hilite{independent} family $(W_1, \dots, W_n)$ of \hilite{nonnegative}
real random variables with two finite moments: $W_i \geq 0$ and $\Expect[ W_i^2 ] < +\infty$.
Define the random matrices
\begin{equation} \label{eqn:sampling-comp}
\begin{aligned}
\mtx{Y} &= \sum_{i=1}^n W_i \mtx{A}_i; \\
\mtx{Z} &= %
	\sum_{i=1}^n X_i \mtx{A}_i
&&\text{where $X_i \sim \normal_{\R}(\Expect[ W_i ], \ \Expect[ W_i^2 ])$ are \hilite{independent}.}
\end{aligned}
\end{equation}
Then there is a stochastic comparison between the minimum eigenvalues of $\mtx{Y}$ and $\mtx{Z}$.
In expectation,
\[
\Expect \lambda_{\min}(\mtx{Y})
	\geq %
	\Expect \lambda_{\min}( \mtx{Z} ) - \sqrt{2 \smash{\sigma_{*}^2}(\mtx{Z}) \log d}. %
\]
In addition, for all $t \geq 0$, the lower tail satisfies
\[
\Prob{ \lambda_{\min}(\mtx{Y}) \leq \Expect \lambda_{\min}( \mtx{Z} ) - t }
	\leq d \cdot \econst^{-t^2 / (2 \sigma_*^2(\mtx{Z}))}.
\]
For this model, the weak variance %
takes the form
\[
\sigma_*^2(\mtx{Z}) = \max\nolimits_{\norm{\vct{u}}=1} \sum_{i=1}^n \Expect[ W_i^2 ] \cdot (\vct{u}^* \mtx{A}_i \vct{u})^2.
\]
\end{theorem}

The proof of \cref{thm:sampling-intro} appears in \cref{sec:psd-weights}.
The short argument employs tools
from Stein's method~\cite{Ros11:Fundamentals-Steins,CGS11:Normal-Approximation};
it also relies on Stahl's theorem~\cite{Sta13:Proof-BMV}.
The weak variance $\sigma_*^2(\mtx{Z})$ arises from
Gaussian concentration~\cite[Thm.~5.5]{BLM13:Concentration-Inequalities}
for the minimum eigenvalue; see~\cref{fact:gauss-lip}.

A valuable feature of \cref{thm:sampling-intro} is that
the psd coefficient matrices $\mtx{A}_i$ are arbitrary,
and the random weights $W_i$ can have different distributions.
Of particular interest is the case where $W_i \sim \bernoulli(p_i)$,
which models a random sample from a family of fixed psd matrices.
Compare the construction of the Gaussian comparison model~\eqref{eqn:sampling-comp}
with the scalar case outlined in~\eqref{eqn:intro-scalar-mgf}.

\Cref{sec:designs} applies \Cref{thm:sampling-intro} to the geometric problem of
sampling from a complex projective design.  This result complements
Rudelson's theorem~\cite{Rud99:Random-Vectors} on sampling from
an isotropic system.

\subsection{Second moments and Gaussians}
\label{sec:intro-mom}

To state our next result compactly, we introduce notation that describes
the second moments of a random matrix model.

\begin{definition}[Random matrix: Second moments, self-adjoint case]
\label{def:second-moments}
Let $\mtx{X} \in\Sym_d$ be a random \hilite{self-adjoint} matrix.
Define the \term{second moment function} and
the \term{variance function} of the random matrix:
\begin{equation*}
\begin{aligned}
\Mo[\mtx{X}](\mtx{M}) &\coloneqq \Expect{} \abssq{\ip{\mtx{M}}{\mtx{X}}} = \Expect[ (\trace[ \mtx{MX} ])^2 ], \\
\Varo[\mtx{X}](\mtx{M}) &\coloneqq \Var[ \ip{\mtx{M}}{\mtx{X}} ] = \Var[ \trace[\mtx{MX}] ]
\end{aligned}
\qquad\text{for \hilite{self-adjoint} $\mtx{M} \in \Sym_d$.}
\end{equation*}
These functions pack up the second-order statistics of the linear marginals of the random matrix.
\end{definition}

Recall that a (real or complex) random matrix is \term{Gaussian} if
and only if the real and imaginary parts of the entries
compose a jointly Gaussian family of real random variables;
e.g., see~\cite[Sec.~2.1.2]{BvH24:Universality-Sharp}.
A self-adjoint Gaussian matrix $\mtx{Z} \in \Sym_d$ is uniquely determined by
its expectation $\Expect[ \mtx{Z} ]$ and variance function $\Varo[\mtx{Z}]$.
Given a self-adjoint matrix $\mtx{\Delta} \in \Sym_d$
and a positive quadratic form $\mathsf{V} : \Sym_d \to \R_+$,
we write $\normal(\mtx{\Delta}, \mathsf{V})$ for the unique
Gaussian distribution on self-adjoint matrices with
expectation $\mtx{\Delta}$ and variance function $\mathsf{V}$.
For more background on Gaussian random matrices, see \cref{sec:gaussians}.

\subsection{Gaussian comparison: Sum of iid psd random matrices}

Our second result provides a Gaussian comparison for a sum
of iid random psd matrices, the model described in the introduction.

\begin{theorem}[Comparison: Sum of iid psd matrices] \label{thm:iid-intro}
Let $\mtx{W}$ be a random \hilite{psd} matrix, real or complex, with dimension $d$
and with two finite moments: $\Expect \norm{\mtx{W}}^2 < + \infty$.
Consider any \hilite{self-adjoint Gaussian} matrix $\mtx{X}$ with dimension $d$
whose first- and second-order statistics satisfy
\begin{equation} \label{eqn:iid-intro-match}
\Expect[ \mtx{X} ] = \Expect[ \mtx{W} ]
\quad\text{and}\quad
\Varo[ \mtx{X} ] \geq \Mo[ \mtx{W} ].
\end{equation}
For a natural number $n \in \N$, define the random matrices
\begin{equation} \label{eqn:iid-comp}
\begin{aligned}
\mtx{Y} &= %
\sum_{i=1}^n \mtx{W}_i
&&\text{where $\mtx{W}_i \sim \mtx{W}$ iid;} \\
\mtx{Z} &= %
\sum_{i=1}^n \mtx{X}_i
&&\text{where $\mtx{Z}_i \sim \mtx{X}$ iid.}%
\end{aligned}
\end{equation}
Then there is a stochastic comparison between the minimum eigenvalues of $\mtx{Y}$ and $\mtx{Z}$.
In expectation,
\begin{align} \label{eqn:intro-thm-iid-expect}
\Expect \lambda_{\min}(\mtx{Y})
	&\geq %
	\Expect \lambda_{\min}( \mtx{Z} ) - \sqrt{2 \smash{\sigma_*^2}(\mtx{Z}) \log(2d)}. %
\end{align}
In addition, for all $t \geq 0$, the lower tail satisfies
\begin{align} \label{eqn:intro-thm-iid-tail}
\Prob{ \lambda_{\min}(\mtx{Y}) \leq \Expect \lambda_{\min}( \mtx{Z} ) - t }
	&\leq 2d \cdot \econst^{-t^2 / (2 \sigma_*^2(\mtx{Z}))}. %
\end{align}
The weak variance %
is defined as
\begin{equation} \label{eqn:intro-weak-var}
\sigma_*^2(\mtx{Z}) \coloneqq \max\nolimits_{\norm{\vct{u}}=1} \Var[ \vct{u}^* \mtx{Z} \vct{u} ].
\end{equation}
\end{theorem}

\cref{thm:iid-intro}
is a (nontrivial) corollary of \cref{thm:sampling-intro}.
For the proof, the strategy is to draw a sample from the distribution
of the summand $\mtx{W}$ and to approximate the iid sum
$\mtx{Y}$ by extracting a subsample.  See~\cref{sec:iid} for the grisly details.

Let us expand on the difference between the two comparison theorems.
While \cref{thm:sampling-intro} forces the summands
to take the simple form $W_i \mtx{A}_i$, %
each summand has its own distribution.
In contrast, \cref{thm:iid-intro} comprehends a general
distribution $\mtx{W}$, but it requires the summands
to be identical copies of $\mtx{W}$.  We permit inequality
in the comparison~\eqref{eqn:iid-intro-match}
to facilitate the application of the result.
The Gaussian comparison model~\eqref{eqn:iid-comp}
for the sum $\mtx{Y}$ can also be written in the form
\begin{equation} \label{eqn:iid-comp-simple}
\mtx{Z} = n \cdot (\Expect \mtx{X}) + \sqrt{n} \cdot (\mtx{X} - \Expect \mtx{X})
	\sim \normal(n \cdot \Expect[\mtx{X}],\ n \cdot \Varo[\mtx{X}] ).
\end{equation}
The statement~\eqref{eqn:iid-comp-simple} follows
from the stability properties of the Gaussian.
Compare with~\eqref{eqn:intro-scalar-mgf}.

\Cref{thm:iid-intro} supports many applications.  As a first example,
\cref{sec:wishart} provides a bound for the minimum eigenvalue of
a Wishart matrix.  \Cref{sec:scov} develops some old and new results
on the minimum eigenvalue of a sample covariance matrix.
\Cref{sec:subspace-inj} studies random linear embeddings,
and it resolves a recalcitrant problem on the injectivity
properties of very sparse random matrices.

\begin{conjecture}[Non-iid sums]
It is natural to conjecture a comparison between the random matrices
\begin{equation*}
\begin{aligned}
\mtx{Y} &= \sum_{i=1}^n \mtx{W}_i
	&& \text{where $\mtx{W}_i$ are psd and independent;} \\
\mtx{Z} &= \sum_{i=1}^n \mtx{X}_i
	&& \text{where $\mtx{X}_i \sim \normal(\Expect[ \mtx{W}_i],\ \Mo[\mtx{W}_i])$ are independent.}
\end{aligned}
\end{equation*}
In this setting, we were only able to establish weak variants of~\eqref{eqn:intro-thm-iid-expect}
and~\eqref{eqn:intro-thm-iid-tail}, but we believe that similar statements should remain valid.
\end{conjecture}

\subsection{Example: Wishart matrix}
\label{sec:wishart}

As a first example, we treat the minimum eigenvalue %
of a standard real Wishart matrix~\cite[Sec.~3.2]{Mui82:Aspects-Multivariate}.
Introduce the random, rank-one psd matrix
\[
\mtx{W} = \vct{gg}^\transp \in \Sym_d(\R)
\quad\text{where $\vct{g} \sim \normal_{\R}(\vct{0}, \Id_d)$.}
\]
Draw independent copies $\mtx{W}_1, \dots, \mtx{W}_n$ of
the random matrix $\mtx{W}$, and form the sum:
\begin{equation} \label{eqn:wishart}
\mtx{Y} = \sum_{i=1}^n \mtx{W}_i
\sim \textsc{wishart}_{\R}(\Id_d, n). %
\end{equation}
After a calculation of the first and second moments of $\mtx{W}$, %
\cref{thm:iid-intro} furnishes a comparison between the Wishart matrix $\mtx{Y}$
and the Gaussian matrix
\begin{equation} \label{eqn:wishart-comp}
\mtx{Z} = n \cdot \Id_d + \gamma \sqrt{n} \cdot \Id_d + \sqrt{n} \cdot \mtx{G}_{\goe} \in \Sym_d(\R),
\end{equation}
where $\gamma \sim \normal_{\R}(0,1)$ and $\mtx{G}_{\goe}$ is drawn independently
from the (unnormalized) Gaussian orthogonal ensemble (GOE); %
see~\cref{sec:goe} for details. %
Exploiting standard facts about the GOE matrix~\cite[Exer.~6.48]{AS17:Alice-Bob}, we find that the minimum eigenvalue and the weak variance satisfy
\[
\Expect \lambda_{\min}(\mtx{Z}) \geq n - 2 \sqrt{dn}
\qquad\text{and}\qquad
\sigma_{*}^2(\mtx{Z}) \leq 3n.
\]
Therefore, \cref{thm:iid-intro} yields the explicit, nonasymptotic bound
\begin{equation} \label{eqn:wishart-nonasymp}
\Expect \lambda_{\min}(\mtx{Y}) \geq n - 2\sqrt{dn} - \sqrt{6 n \log(2d)}.
\end{equation}
The bound is nontrivial when $n \geq \big(2 \sqrt{d} + \sqrt{6 \log(2d)} \big)^2$.
In particular, it suffices that $n \gtrsim d$.

How tight is the inequality~\eqref{eqn:wishart-nonasymp}?  Rescale by the number $n$ of samples,
and introduce the aspect ratio $\varrho \coloneq d/n \in (0,1]$. %
\hilite{When $d, n \to \infty$ with the ratio $\varrho$ fixed}, we determine that
\[
\Expect \lambda_{\min}(n^{-1} \mtx{Y}) \geq 1 - 2 \sqrt{ \varrho } - \sqrt{6 \log (2d) / n}
	\to 1 - 2 \sqrt{ \varrho }.
\]
In this regime, the sharp asymptotic~\cite{BY93:Limit-Smallest} is
\[
\lambda_{\min}(n^{-1} \mtx{Y}) \to 1 - 2 \sqrt{\varrho} + \varrho
\quad\text{almost surely}.
\]
When the aspect ratio $\varrho$ is small (that is, $n \gg d$),
the bound~\eqref{eqn:wishart-nonasymp} %
is correct to first order, including the numerical constant.
On the other hand, the bound is only active inside the regime where $n \geq 4d$,
so it does not speak to the more challenging case
where $n \approx d$. %

\subsection{Nonexample: Wishart matrix}

Suppose that we apply \cref{thm:iid-intro} directly to the random matrix
$\mtx{Y} \sim \textsc{wishart}_{\R}(\Id_d, n)$ %
without a decomposition into rank-one terms.
(That is, we set $\mtx{W} \sim \textsc{wishart}_{\R}(\Id_d, n)$
and add only a single copy.)
After a calculation, we obtain the comparison model: %
\[
\mtx{Z} = n \cdot \Id_d + \gamma n \cdot \Id_d + \sqrt{n} \cdot \mtx{G}_{\goe}.
\]
The minimum eigenvalue satisfies the same bound as before,
but the weak variance is much larger:
\[
\Expect \lambda_{\min}(\mtx{Z}) \geq 1 - 2 \sqrt{dn}
\quad\text{and}\quad
\sigma_{*}^2(\mtx{Z}) \leq n^2 + 2n.
\]
\Cref{thm:iid-intro} results in the comparison
\[
\Expect \lambda_{\min}(\mtx{Y}) \geq n - 2 \sqrt{dn} - n \sqrt{2 (1 + 2/n) \log (2d)}.
\]
This inequality is always vacuous.

From this exercise, we discover that \cref{thm:iid-intro} furnishes different conclusions,
depending on how we decompose a random matrix as a sum of iid  psd terms.
Heuristically, we want to break the random matrix into the smallest pieces we can
to extract the most leverage from the theorem.

\subsection{Equivariance}

An attractive feature of \cref{thm:sampling-intro,thm:iid-intro}
is that the results are equivariant under linear transformations
(\cref{prop:mom-equi}).  In particular, if $\mtx{Z}$ is a Gaussian
comparison model for the random psd sum $\mtx{Y}$, then $\mtx{K}^* \mtx{Z} \mtx{K}$
is a Gaussian comparison model for $\mtx{K}^* \mtx{Y} \mtx{K}$.
In this statement, $\mtx{K}$ is any conformable matrix, not necessarily square.

This observation facilitates the computation of comparison models.
For instance, it is often convenient to transform the random matrix
$\mtx{Y}$ so that its expectation $\Expect \mtx{Y} = \Id$.
The equivariance property also plays a central role in the analysis
of randomized subspace injections (\cref{sec:subspace-inj}).

\subsection{Gaussian comparison versus matrix concentration}

When is the Gaussian comparison method effective?
Consider a self-adjoint Gaussian matrix $\mtx{Z} \in \Sym_d$.
Its minimum eigenvalue satisfies
\begin{equation} \label{eqn:lmin-lb}
\Expect \lambda_{\min}(\mtx{Z})
	\geq \lambda_{\min}(\Expect \mtx{Z}) - \Expect \norm{ \mtx{Z} - \Expect \mtx{Z} }.
\end{equation}
A sufficient condition for
the comparison~\eqref{eqn:intro-thm-iid-expect} %
between $\lambda_{\min}(\mtx{Y})$ and $\lambda_{\min}(\mtx{Z})$ to be informative
is that the right-hand side of~\eqref{eqn:lmin-lb} is positive.  To check the latter condition,
we need to understand the scale for the expected norm.

To that end, define the \term{matrix variance} statistic~\cite[Eqn.~(2.2.4)]{Tro15:Introduction-Matrix}:
\[
\sigma^2(\mtx{Z}) \coloneqq \norm{ \Expect{} (\mtx{Z} - \Expect \mtx{Z})^2 }.
\]
The matrix variance controls the expected norm of a self-adjoint, centered
Gaussian matrix: %
\begin{equation} \label{eqn:nck-intro}
\sqrt{(2/\pi) \, \sigma^2(\mtx{Z})} 
	\leq \Expect \norm{ \mtx{Z} - \Expect \mtx{Z} }
	\leq \sqrt{2 \sigma^2(\mtx{Z}) \log(2d)}.
\end{equation}
This statement~\eqref{eqn:nck-intro} is a variant of the
matrix Khinchin inequality (\cref{fact:nck}).
Both bounds in~\eqref{eqn:nck-intro} are saturated.  We must undertake a more sensitive analysis
to determine whether or not the expected norm includes the dimensional factor, $\log(2d)$.
The matrix variance compares with the weak variance:
\begin{equation} \label{eqn:intro-var-wvar}
\sigma_*^2(\mtx{Z}) \leq \sigma^2(\mtx{Z}) \leq d \cdot \sigma_*^2(\mtx{Z}).
\end{equation}
Both bounds in~\eqref{eqn:intro-var-wvar} are attainable.  In contrast to $\lambda_{\min}(\mtx{Z})$,
the statistics $\sigma^2(\mtx{Z})$ and $\sigma_*^2(\mtx{Z})$ are easy to compute,
as they only depend on the second moments of the random matrix.

We can obtain a coarse version of \cref{thm:iid-sum} by incorporating
the estimates~\eqref{eqn:lmin-lb}, \eqref{eqn:nck-intro}, and~\eqref{eqn:intro-var-wvar}.
For instance, the expectation bound~\eqref{eqn:intro-thm-iid-expect}
implies that
\begin{equation} \label{eqn:mtx-bern-lb}
\Expect \lambda_{\min}(\mtx{Y}) \geq \lambda_{\min}(\Expect \mtx{Z})
	- 2\sqrt{2 \sigma^2(\mtx{Z}) \log(2d)}. %
\end{equation}
In fact, an improvement of the estimate~\eqref{eqn:mtx-bern-lb} follows from simpler arguments
based on the scalar mgf inequality~\eqref{eqn:intro-scalar-mgf} and
matrix concentration tools~\cite{Tro15:Introduction-Matrix}; see \cref{app:epz} for details. %

The benefits of the Gaussian comparison method now come into sharper focus.
\Cref{thm:sampling-intro,thm:iid-intro} are most effective in case
\[
\lambda_{\min}(\Expect \mtx{Z})
	\gg \Expect \norm{ \mtx{Z} - \Expect \mtx{Z}}
	\approx \sigma(\mtx{Z})
	\gg \sigma_*(\mtx{Z}).
\]
But we need a finer scalpel than the matrix Khinchin inequality~\eqref{eqn:nck-intro}
to assess whether the expected norm includes the dimensional factor.

\begin{example}[Wishart: Matrix concentration]
Suppose we apply the matrix concentration bound~\eqref{eqn:mtx-bern-lb}
to the comparison model~\eqref{eqn:wishart-comp} for the Wishart matrix
$\mtx{Y} \sim \textsc{wishart}_{\R}(\Id_d, n)$.  The matrix variance statistic
$\sigma^2(\mtx{Z}) = n(d+2)$, and we arrive at the bound
\[
\Expect \lambda_{\min}(\mtx{Y}) \geq n - 2 \sqrt{2n (d+2) \log(2d)}.
\]
This bound, while nontrivial, does not capture the correct dimensional
dependence that is visible in~\eqref{eqn:wishart-nonasymp}.
We have squandered the valuable distributional information
provided by \cref{thm:iid-intro}.
\end{example}

\begin{remark}[Intrinsic freeness]
Recent research~\cite{Tro18:Second-Order-Matrix,BBvH23:Matrix-Concentration} %
has demonstrated that we can sometimes compare the eigenvalue distribution
of a Gaussian matrix with a free probability model, using simple summary statistics.
These intrinsic freeness results can be valuable for handling the Gaussian comparison model,
but they are unhelpful for the applications in this paper.
\end{remark}

\subsection{Related work}

We are not aware of previous results that are similar in detail
with the Gaussian comparisons from~\cref{thm:sampling-intro,thm:iid-intro}.
Nevertheless, the literature is scattered with breadcrumbs that
form a broken trail toward these new results.

\subsubsection{Universality laws for random matrices}

Compared with our work, the results that are closest in spirit
appear in a recent paper of Brailovskaya \& van Handel~\cite{BvH24:Universality-Sharp}.
Their paper contains quantitative universality theorems for sums of random matrices,
in the spirit of the Berry--Esseen theorem.  As we will explain,
their results are incomparable with the ones in this paper.

Brailovskaya \& van Handel consider an independent sum of
\hilite{self-adjoint} random matrices $\mtx{W}_i$,
\hilite{not necessarily psd or identically distributed}.
They compare the sum with a Gaussian model that shares the
same first- and second-order statistics,
as in the multivariate central limit theorem:
\[
\mtx{Y} = \sum_{i=1}^n \mtx{W}_i
\quad\text{and}\quad
\mtx{Z}' \sim \normal( \Expect[\mtx{Y}],\ \Varo[\mtx{Y}] ).
\]
In contrast, our approach requires the summands $\mtx{W}_i$ to be iid random psd matrices,
and it compares the sum with a slightly different Gaussian model.

Under appropriate conditions on the summands, Brailovskaya \& van Handel argue
that the eigenvalue distribution of $\mtx{Y}$ and the
eigenvalue distribution of the Gaussian model $\mtx{Z}'$ are similar.
Their results address both the spectral density and the spectral support.
As one may imagine, it appears to require stricter assumptions on
the statistics of the random matrices to ensure this strong affinity.

For instance, to control the minimum eigenvalue of the random sum,
Brailovskaya \& van Handel assume that the uniform bound statistic %
\[
R \coloneqq \Expect \max\nolimits_i \norm{\mtx{W}_i}^2 \ll \sigma(\mtx{Y}) (\log d)^{-3}.
\]
This condition ensures that each one of the summands makes a limited contribution to the sum.
The restriction is particularly important when the random matrices have few moments
or the summands have inhomogeneous distributions.
Under this surmise,
they prove~\cite[Thm.~2.8]{BvH24:Universality-Sharp}
that the expected distance between the minimum eigenvalues satisfies
\[
\Expect \abs{ \lambda_{\min}(\mtx{Y}) - \lambda_{\min}(\mtx{Z}') }
	\lesssim \sigma(\mtx{Z}')^{5/6} R^{1/6} \log d %
	+ \smash{\sigma_*(\mtx{Z}')} (\log d)^{1/2}.
\]
As in the present paper, the second term arises from Gaussian concentration.
To understand the first term, recall from~\eqref{eqn:nck-intro} that the scale
for the minimum eigenvalue of $\Expect \lambda_{\min}(\mtx{Z}' - \Expect \mtx{Z}')$
is the statistic $\sigma(\mtx{Z}')$.  In some examples,
the factor $R^{1/6} \log d$ submerges the first term below this level.

The results of Brailovskaya \& van Handel are powerful and wide ranging.
For some of the applications we consider in this paper, however,
their approach produces several parasitic logarithmic factors
that make it impossible to reach the optimal bounds.
These factors are particularly significant
in applications to computational mathematics (\cref{sec:subspace-inj}),
where constants and logarithms matter.

As with our Gaussian comparison theorems, the proof strategy in the paper of
Brailovskaya \& van Handel is based on techniques from Stein's method.
While there are small points of similarity (e.g., the use of interpolation),
our technical apparatus follows an independent design.

\subsubsection{An independent sum of random psd matrices}
\label{sec:related-psd-sum}

There is also a body of work that provides specialized bounds
for the minimum eigenvalue of an independent sum of random psd matrices.
Several of these papers are inspired by the same observation~\eqref{eqn:intro-scalar-tail}
that an independent sum of nonnegative real random variables has a Gaussian lower tail,
and they pursue this insight in creative and multifarious ways.
We focus on the earliest and most distinctive contributions.

This literature treats several different random matrix models.
To facilitate comparisons, we summarize the implications
for a special case involving sample covariance matrices.
Consider a random vector $\vct{w} \in \R^d$ that has \hilite{four} finite moments.
For normalization, assume that the vector is \term{isotropic}:
$\Expect[ \vct{ww}^\transp ] = \Id_d$.
Form the sample covariance matrix based on $n$ samples:
\[
\mtx{Y} = \frac{1}{n} \sum_{i=1}^n \vct{w}_i \vct{w}_i^\transp
\quad\text{where $\vct{w}_i \sim \vct{w}$ iid.}
\]
For a parameter $\eps \in (0, 1)$, how many samples $n = n(\eps)$
are sufficient to ensure that
$\lambda_{\min}(\mtx{Y}) \geq 1 - \eps$ with high probability?
The weak assumptions on moments make this question very challenging.

Srivastava \& Vershynin~\cite{SV13:Covariance-Estimation}
formulated this problem and made the first contribution.
They considered a random vector $\vct{w} \in \R^d$ that satisfies the uniform fourth moment condition
\begin{equation} \label{eqn:sv-moments}
\max\nolimits_{\norm{\vct{u}} = 1} \Expect{} \abs{ \ip{ \vct{u} }{ \vct{w} } }^{4} \leq L.
\end{equation}
They establish a sample complexity bound:
\[
n \geq 400 L \eps^{-3} d
\quad\text{implies}\quad
\Expect \lambda_{\min}(\mtx{Y}) \geq 1 - \eps.
\]
Their paper is important because the sample complexity $n$ has the correct (linear)
dependence on the dimension $d$, although it exhibits a suboptimal dependence on $\eps$.
Their proof employs a Stieltjes transform argument inspired
by work in spectral graph theory~\cite{BSS14:Twice-Ramanujan}.

Koltchinskii \& Mendelson~\cite{KM15:Bounding-Smallest} developed a family
of related results under a weak form of the moment condition~\eqref{eqn:sv-moments}.
For some $\eta > 2$, assume that
\[
\Prob{ \abs{\ip{ \vct{u} }{ \vct{w} }} > t } \leq L t^{-(2 + \eta)}
\quad\text{when $\norm{ \vct{u} } = 1$ and $t \geq 1$.}
\]
Their work yields the correct dependence on the parameter $\eps$.  Indeed,
\[
	n \geq C \eps^{-2} d
	\quad\text{implies}\quad
\Prob{ \lambda_{\min}(\mtx{Y}) < 1 - \eps }
	\leq C \log(\econst n / d) \cdot \econst^{- d / C }.
\]
The constant $C = C(\eta, L)$.
Their proof involves truncation, small ball probabilities, and a Vapnik--Chervonenkis dimension argument.

Oliveira~\cite{Oli16:Lower-Tail} proposed a third approach.
Under the assumption~\eqref{eqn:sv-moments}, he established that
\begin{equation} \label{eqn:scov-oliveira}
	n \geq 81 L \eps^{-2}{} (d + 2 \log(2/\delta))
	\quad\text{implies that}\quad
\Prob{ \lambda_{\min}(\mtx{Y}) < 1 - \eps }
	\leq \delta.
\end{equation}
The bound~\eqref{eqn:scov-oliveira} has the correct dependence on all of the parameters.
Oliveira's argument relies on the
PAC-Bayesian method~\cite{LST02:PAC-Bayes,McA03:Simplified-PAC-Bayesian,AC11:Robust-Linear},
a technique that smooths the distribution and employs the variational
properties of the entropy to obtain bounds.
In \cref{sec:scov}, we will exploit the Gaussian comparison method
to give a short proof of Oliveira's bound~\eqref{eqn:scov-oliveira}.

The papers discussed in this section depend heavily on
moment assumptions, such as~\eqref{eqn:sv-moments},
which cannot capture the detailed distribution of a random vector.
To highlight the benefits of the Gaussian comparison method,
we will obtain new bounds for %
the sample covariance of a very sparse random vector (\cref{thm:sparse-cov}).

\section{Gaussian random matrices}
\label{sec:gaussians}

To take advantage of the Gaussian comparison method,
we must be able to construct the comparison model
and determine its spectral properties.
This section outlines some facts about
Gaussian random matrices that will play a role
in the applications and the proofs of the main theorems.
For generality, we work in the complex setting,
which includes the real setting as a special case.

\subsection{The Cartesian decomposition}

To simplify some formulas, let us introduce the Cartesian decomposition of a square matrix.
The \term{real part} and \term{imaginary part} of a square matrix
$\mtx{M} \in \M_d(\C)$ with complex entries are the self-adjoint matrices
\begin{align*}
\Re \mtx{M} \coloneqq \frac{1}{2} (\mtx{M} + \mtx{M}^*) \in \Sym_d(\C)
\quad\text{and}\quad
\Im \mtx{M} \coloneqq \frac{1}{2\iunit} (\mtx{M} - \mtx{M}^*) \in \Sym_d(\C).
\end{align*}
The \term{Cartesian decomposition} states that $\mtx{M} = (\Re \mtx{M}) + \iunit (\Im \mtx{M})$,
and the two terms are orthogonal with respect to the trace inner product.

\subsection{First and second moments of random matrices}

Let $\mtx{X} \in \Sym_d(\C)$ be a random \hilite{self-adjoint} matrix with complex entries.
The second moment function (\cref{def:second-moments})
contains information about the second moment of each entry of the random matrix:
\begin{align*}
\Mo[\mtx{X}]( \Re \mathbf{E}_{ij} )
	= \Expect \abssq{\Re X_{ij}}
	\quad\text{and}\quad
\Mo[\mtx{X}]( \Im \mathbf{E}_{ij} )
	= \Expect \abssq{\Im X_{ij}}.
\end{align*}
As usual, $\mathbf{E}_{ij} \in \M_d(\C)$ is the $(i, j)$ element
of the standard basis for matrices.
We can extract mixed second moments via polarization.  For example,
\[
\Mo[\mtx{X}]\big( (\Re \mathbf{E}_{ij}) + (\Re \mathbf{E}_{k\ell}) \big)
	\ -\ \Mo[\mtx{X}]\big( (\Re \mathbf{E}_{ij}) - (\Re \mathbf{E}_{k\ell}) \big)
	= 4 \Expect\big[ (\Re X_{ij})(\Re X_{k\ell}) \big].
\]
In the last two displays, the range of the indices $i,j,k,\ell = 1, \dots, d$.

The variance function (\cref{def:second-moments}) collects the second moments
of the \hilite{centered} random matrix: %
\begin{equation*} %
\Varo[\mtx{X}] = \Mo[ \mtx{X} - \Expect \mtx{X} ] = \Mo[\mtx{X} ] - \Mo[ \Expect \mtx{X} ].
\end{equation*}
For random self-adjoint matrices $\mtx{X}, \mtx{Y} \in \Sym_d(\C)$,
the variance function is additive in the sense that
\begin{equation} \label{eqn:cov-add}
\Varo[ \mtx{X} + \mtx{Y} ] = \Varo[ \mtx{X} ] + \Varo[\mtx{Y} ]
\quad\text{when $\mtx{X}, \mtx{Y}$ are \hilite{independent}.}
\end{equation}
For contrast, the second moment function $\Mo$ does not satisfy
the additivity rule~\eqref{eqn:cov-add}.

The first and second moments of a random matrix are equivariant under linear transformations.
This fact allows us to transfer the comparison theorems between models.
We remark that there are many types of linear transformations that preserve the cone
of psd matrices~\cite[Sec.~2.1]{Bha07:Positive-Definite}.

\begin{proposition}[Moments: Equivariance] \label{prop:mom-equi}
Consider random self-adjoint matrices $\mtx{W}, \mtx{X} \in \Sym_{d}(\C)$ with
\[
\Expect[ \mtx{X} ] = \Expect[ \mtx{W} ]
\quad\text{and}\quad
\Varo[\mtx{X}] \geq \Mo[\mtx{W}].
\]
Then, for each real-linear map $\coll{L} : \Sym_{d}(\C) \to \Sym_{d'}(\C)$ on self-adjoint matrices,
\[
\Expect[ \coll{L}(\mtx{X}) ] = \Expect[ \coll{L}(\mtx{W}) ]
\quad\text{and}\quad
\Varo[ \coll{L}(\mtx{X}) ] \geq \Mo[ \coll{L}(\mtx{W}) ].
\]
\end{proposition}

\begin{proof}
By linearity,
\[
\Expect[ \coll{L}(\mtx{X}) ] = \coll{L}( \Expect \mtx{X} )
	= \coll{L}( \Expect \mtx{W} ) = \Expect[ \coll{L}(\mtx{X}) ].
\]
By duality, for each $\mtx{M} \in \Sym_{d}(\C)$,
\begin{align*}
\Varo[ \coll{L}(\mtx{X}) ](\mtx{M}) &= \Var[ \ip{ \mtx{M} }{ \coll{L}(\mtx{X}) } ]
	= \Var[ \ip{ \coll{L}^*(\mtx{M}) }{ \mtx{X} } ] \\
	&= \Varo[ \mtx{X} ](\coll{L}^*(\mtx{M})
	\geq \Mo[ \mtx{W} ](\coll{L}^*(\mtx{M})) \\
	&= \Expect \abssq{ \ip{ \coll{L}^*(\mtx{M}) }{ \mtx{W} } }
	= \Expect \abssq{ \ip{ \mtx{M} }{ \coll{L}(\mtx{W}) } }
	= \Mo[ \coll{L}(\mtx{W}) ]( \mtx{M} ).
\end{align*}
As usual, $\coll{L}^*$ is the adjoint of the linear map with respect to the trace inner product.
\end{proof}

\subsection{Matrix Gaussian series}

\Cref{sec:intro-mom} %
provides an abstract definition of a %
Gaussian random matrix.
This section introduces
a concrete model that is often useful for calculations.
A (self-adjoint) \term{matrix Gaussian series} is a random matrix model of the form
\begin{equation} \label{eqn:mtx-gauss}
\mtx{Z} = \mtx{\Delta} + \sum_{i=1}^n \gamma_i \mtx{H}_i
\quad\text{where $\gamma_i \sim \normal_{\R}(0,1)$ iid.}
\end{equation}
The coefficients $(\mtx{\Delta}, \mtx{H}_1, \dots, \mtx{H}_n)$
are deterministic matrices in $\Sym_d(\C)$.  Let us emphasize
that the Gaussian random variables $\gamma_i$ %
are \hilite{real-valued}.
Every self-adjoint Gaussian random matrix can be written in the form~\eqref{eqn:mtx-gauss}
in many different ways.

We can easily compute the mean and variance function of the
random matrix~\eqref{eqn:mtx-gauss}:
\begin{equation} \label{eqn:mtx-gauss-var}
\Expect \mtx{Z} = \mtx{\Delta}
\quad\text{and}\quad
\Varo[\mtx{Z}](\mtx{M}) = \sum_{i=1}^n \abssq{ \ip{ \mtx{M} }{\mtx{H}_i} }
\quad\text{for $\mtx{M} \in \Sym_d(\C)$.}
\end{equation}
The expression for the variance function follows from the additivity rule~\eqref{eqn:cov-add}.

\subsection{Gaussian monotonicity}

Gaussian random matrices enjoy a strong monotonicity property.
As the variance function increases, expectations of
convex functions of the random matrix also increase.

\begin{proposition}[Gaussians: Monotonicity] \label{prop:gauss-monotone}
Consider two self-adjoint Gaussian matrices whose distributions
$\mtx{Z} \sim \normal(\mtx{\Delta}, \mathsf{V})$ and
$\mtx{Z}' \sim \normal(\mtx{\Delta}, \mathsf{V}')$ take values in $\Sym_d(\C)$
and share the same mean.
Suppose the variance functions
$\mathsf{V}$ and $\mathsf{V}'$ satisfy the pointwise inequality
\begin{equation} \label{eqn:gauss-psd}
\mathsf{V}(\mtx{M}) \leq \mathsf{V}'(\mtx{M})
	\quad\text{for all $\mtx{M} \in \Sym_d(\C)$.}
\end{equation}
For each convex function $f : \Sym_d(\C) \to \R$,
\begin{equation} \label{eqn:gauss-convex}
\Expect f(\mtx{Z}) \leq \Expect f(\mtx{Z}').
\end{equation}
\end{proposition}

\begin{proof} %
Since $\mathsf{V}' - \mathsf{V}$ is a positive quadratic form on $\Sym_d(\C)$,
we can construct a Gaussian matrix $\mtx{X} \sim \normal(\mtx{0}, \mathsf{V}' - \mathsf{V})$
that is independent from $\mtx{Z}$.
By linearity of expectation, $\Expect[ \mtx{Z} + \mtx{X} ] = \mtx{\Delta}$.
By additivity of the variance function~\eqref{eqn:cov-add},
\[
\Varo[ \mtx{Z} + \mtx{X} ] = \Varo[ \mtx{Z} ] + \Varo[ \mtx{X} ] = \mathsf{V} + (\mathsf{V}' - \mathsf{V}) = \mathsf{V}'.
\]
Therefore, the sum $\mtx{Z} + \mtx{X} \sim \normal(\mtx{\Delta}, \mathsf{V}')$
follows the same distribution as $\mtx{Z}'$.
To obtain the convexity inequality~\eqref{eqn:gauss-convex}, note that
\[
\Expect f(\mtx{Z}') = \Expect f(\mtx{Z} + \mtx{X})
	\geq \Expect f(\mtx{Z}).
\]
We have invoked Jensen's inequality, conditional on $\mtx{Z}$.
\end{proof}

\subsection{Matrix variance}

We can capture information about spectral features
of a Gaussian matrix using scalar summary statistics.
The \term{matrix variance} %
of a self-adjoint Gaussian matrix $\mtx{Z} \in \Sym_d(\C)$
is
\begin{equation} \label{eqn:matrix-var}
\sigma^2(\mtx{Z}) \coloneqq \lnorm{ \Expect{} (\mtx{Z} - \Expect \mtx{Z})^2 }
	= \lnorm{ \sum_{i=1}^n \mtx{H}_i^2 }.
\end{equation}
The second identity in~\eqref{eqn:matrix-var} is valid
for an \hilite{arbitrary} representation~\eqref{eqn:mtx-gauss}
of $\mtx{Z}$ as a Gaussian series.
When $\mtx{Z}, \mtx{Z}'$ are \hilite{independent}, %
we have the subadditivity rule $\sigma^2(\mtx{Z} + \mtx{Z}') \leq \sigma^2(\mtx{Z}) + \sigma^2(\mtx{Z}')$.

The matrix Khinchin inequality~\cite[Thm.~4.6.1]{Tro15:Introduction-Matrix}
states that the matrix variance controls the expectation of the extreme eigenvalues of the
\hilite{centered} Gaussian matrix:

\begin{fact}[Matrix Khinchin] \label{fact:nck}
Consider a \hilite{self-adjoint Gaussian} matrix $\mtx{Z}$ with dimension $d$.
Then
\begin{equation} \label{eqn:mki}
- \Expect \lambda_{\min}(\mtx{Z} - \Expect \mtx{Z})
= \Expect \lambda_{\max}(\mtx{Z} - \Expect \mtx{Z})
\leq \sqrt{2 \sigma^2(\mtx{Z}) \log d}.
\end{equation}
In general, the logarithmic factor in~\eqref{eqn:mki} is required.
The lower bound stated in~\eqref{eqn:nck-intro} follows from a sharp moment
comparison~\cite[Cor.~3]{LO99:Gaussian-Measures} and Jensen's inequality.
\end{fact}

\subsection{Weak variance}

For a self-adjoint Gaussian matrix $\mtx{Z} \in \Sym_d(\C)$,
the \term{weak variance} statistic is defined as
\begin{equation} \label{eqn:weak-var}
\sigma_*^2(\mtx{Z}) \coloneqq \max\nolimits_{\norm{\vct{u}} = 1} \Var[ \vct{u}^* \mtx{Z} \vct{u} ]
	= \max\nolimits_{\norm{\vct{u}} = 1} \sum_{i=1}^n (\vct{u}^* \mtx{H}_i \vct{u})^2.
\end{equation}
The second identity holds for an \hilite{arbitrary} representation~\eqref{eqn:mtx-gauss}
of $\mtx{Z}$ as a Gaussian series.
As stated in~\eqref{eqn:intro-var-wvar},
the weak variance is controlled by the matrix variance,
and both bounds are attainable.

The weak variance can easily be expressed in terms of the variance function:
\[
\sigma_*^2(\mtx{Z}) = \max\nolimits_{\norm{\vct{u}}=1} \Varo[ \mtx{Z} ](\vct{uu}^*)
	= \max\nolimits_{\norm{\mtx{M}}_1 = 1} \Varo[ \mtx{Z} ](\mtx{M}).
\]
We have written $\norm{\cdot}_1$ for the Schatten $1$-norm.
As a consequence, for self-adjoint Gaussian matrices $\mtx{Z}, \mtx{Z}' \in \Sym_d(\C)$,
the weak variance is monotone with respect to the variance function:
\begin{equation} \label{eqn:wvar-monotone}
\Varo[\mtx{Z}] \leq \Varo[\mtx{Z}']
\quad\text{implies}\quad
\sigma_*^2(\mtx{Z}) \leq \sigma_*^2(\mtx{Z}').
\end{equation}
When $\mtx{Z}, \mtx{Z}'$ are \hilite{independent}, we also have the subadditivity rule
$\sigma_*^2(\mtx{Z} + \mtx{Z}') \leq \sigma_*^2(\mtx{Z}) + \sigma_*^2(\mtx{Z}')$.

The weak variance arises when studying concentration properties
of Gaussian random matrices.

\begin{fact}[Gaussian concentration] \label{fact:gauss-lip}
Let $h : \Sym_d(\C) \to \R$ be a function that is $1$-Lipschitz with respect to the
$\ell_2$ operator norm:
\[
\abs{ h(\mtx{A}) - h(\mtx{B}) } \leq \norm{ \mtx{A} - \mtx{B} }
\quad\text{for all $\mtx{A}, \mtx{B} \in \Sym_d(\C)$.}
\]
For a self-adjoint Gaussian matrix $\mtx{Z} \in \Sym_d(\C)$,
the mgf %
of $h(\mtx{Z})$ satisfies
\begin{equation} \label{eqn:gauss-lip-mgf}
\Expect \econst^{ \theta (h(\vct{Z}) - \Expect h(\mtx{Z}))}
	\leq \econst^{\theta^2 \sigma_*^2(\mtx{Z}) / 2}
	\quad\text{for all $\theta \in \R$.}
\end{equation}
We can instantiate~\eqref{eqn:gauss-lip-mgf} with $h = \lambda_{\max}$ or $h = \lambda_{\min}$
because of Weyl's inequality~\cite[Cor.~III.2.6]{Bha97:Matrix-Analysis}.
\end{fact}
\begin{proof}[Proof sketch]
To confirm~\eqref{eqn:gauss-lip-mgf}, define the composite function
\[
f : (\gamma_1, \dots, \gamma_n) \mapsto h\left( \mtx{\Delta} + \sum_{i=1}^n \gamma_i \mtx{H}_i \right) \eqqcolon h(\mtx{Z}).
\]
We quickly calculate the Lipschitz constant of $f$ with respect to the $\ell_2$ norm:
\begin{align*}
\abs{ f(a_1, \dots, a_n) - f(b_1, \dots, b_n) }
	&\leq \lnorm{ \sum_{i=1}^n (a_i - b_i) \mtx{H}_i } \\
	&= \max\nolimits_{\norm{\vct{u}}=1} \sum_{i=1}^n (a_i - b_i) (\vct{u}^* \mtx{H}_i \vct{u})
	\leq \sigma_*(\mtx{Z}) \cdot \norm{ \vct{a} - \vct{b} }.
\end{align*}
The last inequality is Cauchy--Schwarz.  In other words,
the statistic $\sigma_*(\mtx{Z})$ is the Lipschitz constant of $h(\mtx{Z})$.
Gaussian Lipschitz concentration~\cite[Thm.~5.5]{BLM13:Concentration-Inequalities}
yields the mgf bound~\eqref{eqn:gauss-lip-mgf}.
\end{proof}

\subsection{Basic examples}
\label{sec:gauss-examples}

Let us collect some simple Gaussian matrices that
we may combine to build comparison models.
Indeed, 
we can represent a Gaussian matrix as a sum of independent Gaussian
matrices by breaking the variance function
into pieces and identifying a
Gaussian model for each piece.
This strategy is justified by the 
additivity~\eqref{eqn:cov-add} of
the variance.
In this section,
the random variables $\gamma, \gamma_i, \gamma_{jk}, \smash{\gamma_{jk}'}$
are iid $\normal_{\R}(0,1)$.
For $v \geq 0$, the distribution $\normal_{\C}(0,v)$ generates a complex Gaussian random
variable whose real and imaginary parts are iid $\normal_{\R}(0,v/2)$
random variables.

\subsubsection{Scalar Gaussian matrix}

As a warmup, consider the scalar matrix
$\mtx{S} \coloneqq \gamma \Id \in \Sym_d(\C)$.
It is easy to see that $\Expect \lambda_{\min}(\mtx{S}) = \Expect \lambda_{\max}(\mtx{S}) = 0$.
The variance function of the scalar matrix acts as
\[
\Varo[\mtx{S}](\mtx{M}) = (\trace \mtx{M})^2
\quad\text{for $\mtx{M} \in \Sym_d(\C)$.}
\]
Last, note that the matrix variance and weak variance coincide:
$\sigma^2(\mtx{S}) = \sigma^2_{*}(\mtx{S}) = 1$.

\subsubsection{Diagonal Gaussian matrix}

Next, consider the diagonal matrix:
\[
\mtx{D} \coloneqq \sum_{i=1}^d \gamma_i \mathbf{E}_{ii} \in \Sym_d(\C). %
\]
The extreme eigenvalues satisfy
\[
- \Expect \lambda_{\min}(\mtx{D}) = \Expect \lambda_{\max}(\mtx{D}) = \Expect \max\{ \gamma_i : i =1, \dots, d\}
	\leq \sqrt{2 \log d}.
\]
This expectation bound is asymptotically sharp.
The variance function of the diagonal matrix is
\[
\Varo[\mtx{D}](\mtx{M}) = \sum_{i=1}^d \abssq{m_{ii}}
\quad\text{for $\mtx{M} \in \Sym_d(\C)$.}
\]
The matrix variance and weak variance again coincide:
$\sigma^2(\mtx{D}) = \sigma_*^2(\mtx{D}) = 1$.

\subsubsection{Gaussian orthogonal ensemble}
\label{sec:goe}

Now, we turn to a more sophisticated example.
A random matrix from the (unnormalized) \term{Gaussian orthogonal ensemble (GOE)}
takes the form
\[
\mtx{G}_{\goe} \coloneqq \sqrt{2} \sum_{j,k = 1}^d \gamma_{jk} (\Re \mathbf{E}_{jk})
	\in \Sym_d(\R).
\]
Note that this matrix is real and symmetric.
The diagonal entries are iid $\normal_{\R}(0,2)$ random variables,
while the entries in the upper triangle are iid $\normal_{\R}(0,1)$ random variables.
The key property of the GOE distribution is orthogonal invariance:
\begin{equation} \label{eqn:goe-invar}
\mtx{Q}^\transp \mtx{G}_{\goe} \mtx{Q} \sim \mtx{G}_{\goe}
\quad\text{for each orthogonal $\mtx{Q} \in \M_d(\R)$.}
\end{equation}
The extreme eigenvalues of the GOE matrix satisfy an elegant bound~\cite[Exer.~6.48]{AS17:Alice-Bob}:
\begin{equation} \label{eqn:gue-eig}
- \Expect \lambda_{\min}( \mtx{G}_{\goe} ) = \Expect \lambda_{\max}( \mtx{G}_{\goe} ) \leq 2 \sqrt{d}.
\end{equation}
To recognize the GOE matrix, observe that its variance function is
\[
\Varo[\mtx{G}_{\goe}](\mtx{M}) = 2 \sum_{j,k = 1}^d \abssq{m_{jk}} = 2 \fnormsq{\mtx{M}}
\quad\text{for $\mtx{M} \in \Sym_d(\R)$.}
\]
The matrix variance and weak variance differ almost
as much as possible.  Indeed,
$\sigma^2(\mtx{G}_{\goe}) = d + 1$ and
$\sigma_*^2(\mtx{G}_{\goe}) = 2$.
GOE matrices arise in several of our comparisons. (For example, see \cref{sec:wishart}.)

\subsubsection{Gaussian unitary ensemble}

The \term{Gaussian unitary ensemble (GUE)} is the complex cousin of the GOE.
A random matrix from this family takes the form
\[
\mtx{G}_{\gue} \coloneqq \sum_{j,k = 1}^d \big[ \gamma_{jk} (\Re \mathbf{E}_{jk})
	+ \gamma_{jk}' (\Im \mathbf{E}_{jk}) \big]
	\in \Sym_d(\C).
\]
The diagonal entries are iid real $\normal_{\R}(0,1)$ random variables,
while the entries in the upper triangle are iid complex $\normal_{\C}(0,1)$ random variables.
The GUE distribution is unitarily invariant:
\[
\mtx{U}^* \mtx{G}_{\gue} \mtx{U} \sim \mtx{G}_{\gue}
\quad\text{for each unitary $\mtx{U} \in \M_d(\C)$.}
\]
The extreme eigenvalues of the GUE matrix satisfy the following bound~\cite[Exer.~6.38]{AS17:Alice-Bob}.
\[
- \Expect \lambda_{\min}( \mtx{G}_{\gue} ) = \Expect \lambda_{\max}( \mtx{G}_{\gue} ) \leq 2 \sqrt{d}.
\]
The variance function of the GUE matrix acts as
\[
\Varo[\mtx{G}_{\gue}](\mtx{M}) = \sum_{j,k=1}^d \abssq{m_{jk}}
	= \fnormsq{\mtx{M}}
\quad\text{for $\mtx{M} \in \Sym_d(\C)$.}
\]
The matrix variance and weak variance differ as much as possible:
$\sigma^2(\mtx{G}_{\gue}) = d$ and
$\sigma_*^2(\mtx{G}_{\gue}) = 1$.
Among all random matrices, the GUE matrix is perhaps the most fundamental.

\section{Application: Sampling from a design}
\label{sec:designs}

We begin with a geometric example.  If we randomly subsample a set of vectors
that spans a (complex) linear space, does the reduced set still span the space?
In considering this question, we must enforce some regularity properties to
avoid situations where most of the vectors fall in a proper
subspace.  For illustration, we impose a rather strong condition, which ensures
that the vectors are distributed evenly across the unit sphere.

\subsection{Projective designs}

Consider a finite system $\coll{U} \coloneqq (\vct{u}_1, \dots, \vct{u}_n)$ of unit-norm vectors in $\C^d$.
The system is called a \term{complex projective $t$-design} when it yields a quadrature
rule for homogeneous degree-$2t$ polynomials on the complex unit sphere $\mathbb{S}^{d-1}(\C)$.
More precisely, for each vector $\vct{a} \in \C^d$,
\[
\frac{1}{n} \sum_{i=1}^n \abs{ \ip{\vct{a}}{\vct{u}_i} }^{2t} = \Expect{}\abs{ \ip{\vct{a}}{\vct{v}} }^{2t}
\quad\text{where $\vct{v} \sim \uniform(\mathbb{S}^{d-1})$.} %
\]
In particular, a \term{projective 1-design} is the same as an \term{isotropic system}
(aka a \term{unit-norm tight frame}), so it spans the whole space:
\begin{equation} \label{eqn:untf}
\frac{1}{n} \sum_{i=1}^n \vct{u}_i \vct{u}_i^* = \frac{1}{d} \cdot \Id_d.
\end{equation}
We are interested in a \term{projective 2-design}, which is characterized by the condition
\begin{equation} \label{eqn:4design}
\frac{1}{n} \sum_{i=1}^n \abssq{ \ip{\mtx{M}}{\vct{u}_i \vct{u}_i^*} }
	= \frac{1}{d(d+1)} \big[ \fnormsq{\mtx{M}} + (\trace \mtx{M})^2 \big]
	\quad\text{for all $\mtx{M} \in \Sym_d(\C)$.}
\end{equation}
A projective $2$-design is always a projective $1$-design.
Examples of projective 2-designs include systems of equiangular vectors,
mutually unbiased bases, and other highly symmetric configurations.
See~\cite[Chap.~8]{Wal18:Introduction-Finite} or~\cite[App.~B.1]{GKKT20:Fast-State-Tomography}
for more examples and discussion.

\subsection{Sampling from a projective design}

A classic question in nonasymptotic random matrix theory
asks when a random subset of an isotropic system %
remains a spanning set.  The fundamental result
for this problem is due to Rudelson~\cite{Rud99:Random-Vectors};
see \cref{sec:rudelson} for a proof sketch.

\begin{fact}[Sampling: Projective 1-design] \label{fact:sampling-1-design}
Consider a complex projective \hilite{1-design} $\coll{U} \coloneqq (\vct{u}_1, \dots, \vct{u}_n)$
consisting of $n$ unit-norm vectors in $\smash{\C^d}$.
For a parameter $1 \leq s \leq n$, construct a random subsystem $\coll{U}'$
with an average of $s$ vectors by uniform sampling:
\begin{equation} \label{eqn:subsample-vecs}
\coll{U}' = (\vct{u}_i : \xi_i = 1)
\quad\text{where $\xi_i \sim \bernoulli(s/n)$ iid.}
\end{equation}
When $s \geq d \log(d/\delta)$,
the random system $\coll{U}'$ spans $\C^d$ with probability at least $1 - \delta$.
The logarithmic factor in the sampling complexity is necessary.
\end{fact}

As a complement to \cref{fact:sampling-1-design}, %
we study the problem of
sampling from a complex projective \hilite{2-design}.
We will %
establish that a random subset
remains a spanning set when the average number of vectors
exceeds the ambient dimension by a \hilite{constant} factor.
The proof appears in~\cref{sec:sampling-2-design}.

\begin{theorem}[Sampling: Projective 2-design] \label{thm:sampling-2-design}
With the notation of \cref{fact:sampling-1-design}, assume that
the system $\coll{U}$ is a complex projective \hilite{2-design}.
When the average number $s$ of sampled vectors satisfies
\[
s \geq 4 \big[ \sqrt{d} + \sqrt{\log(d/\delta)} \big]^2, %
\]
the random subset $\coll{U}'$ spans $\C^d$ with probability at least $1 - \delta$.
\end{theorem}

To analyze the subsampled system $\coll{U}'$, defined in \eqref{eqn:subsample-vecs},
we introduce the random psd matrix
\begin{equation} \label{eqn:Y-subsample}
\mtx{Y} \coloneqq \sum_{\vct{u} \in \coll{U}'} \vct{uu}^*
	= \sum_{i=1}^n \xi_i \vct{u}_i \vct{u}_i^*
	\quad\text{where $\xi_i \sim \bernoulli(s/n)$ iid.}
\end{equation}
Observe that the system $\coll{U}'$ spans $\C^d$ if and only if $\lambda_{\min}(\mtx{Y}) > 0$.
We can obtain the minimum eigenvalue bound from a quick application of the Gaussian comparison method (\cref{thm:sampling-intro}).
In contrast,
the existing techniques for minimum eigenvalues~\cite{SV13:Covariance-Estimation,KM15:Bounding-Smallest,Oli16:Lower-Tail} %
all fail to provide the correct bound for $\lambda_{\min}(\mtx{Y})$
because the summands in~\eqref{eqn:Y-subsample} are both spiky and non-identically distributed.
Spectral universality tools, such as~\cite[Cor.~2.7]{BvH24:Universality-Sharp},
also produce the wrong answer when applied to the random matrix $\mtx{Y}$.

\subsubsection{Proof of \cref{fact:sampling-1-design}}
\label{sec:rudelson}

Assume that $\coll{U}$ is a complex projective \hilite{1-design}.
For the proof, we express the average number of sampled vectors $s = \beta d$
in terms of a parameter in the range $1 \leq \beta \leq n/d$.

Consider the random matrix $\mtx{Y}$, defined in~\eqref{eqn:Y-subsample}.
By the isotropy property~\eqref{eqn:untf},
\begin{equation*} \label{eqn:EY-subsample}
\Expect \mtx{Y} = \frac{\beta d}{n} \sum_{i=1}^n \vct{u}_i \vct{u}_i^* = \beta \cdot \Id_d.
\end{equation*}
The $\vct{u}_i$ are unit vectors, so the spectral norm of each summand in~\eqref{eqn:Y-subsample}
is bounded by one.
An application of the matrix Chernoff inequality~\cite[Thm.~5.1.1]{Tro15:Introduction-Matrix}
yields
\[
\Prob{ \lambda_{\min}(\mtx{Y}) = 0 }
	\leq d \cdot \econst^{-\beta}.
\]
If the oversampling parameter $\beta \geq \log(d/\delta)$, then the probability that $\mtx{Y}$
is singular is at most $\delta$.  %

Last, to confirm that the sampling complexity bound is unimprovable without further assumptions,
consider the 1-design obtained by repeating the standard basis for $\R^d$: %
\[
\coll{U} = (\underbrace{\mathbf{e}_1, \dots, \mathbf{e}_1}_R, \underbrace{\mathbf{e}_2, \dots, \mathbf{e}_2}_R, \dots, \underbrace{\mathbf{e}_d, \dots, \mathbf{e}_d}_R).
\]
The parameter $R$ is a large natural number.
To reliably select each distinct basis vector at least once, it takes $s \geq d \log d$ samples
because of the coupon collector phenomenon~\cite[Sec.~2.4.1]{MU17:Probability-Computing-2ed}.

\subsubsection{Proof of \cref{thm:sampling-2-design}}
\label{sec:sampling-2-design}

Assume that $\coll{U}$ is a complex projective \hilite{2-design},
and parameterize the average number of samples as $s = \beta d$.

Consider the random matrix $\mtx{Y}$ defined in~\eqref{eqn:Y-subsample}.
\Cref{thm:sampling-intro} immediately yields a comparison with the Gaussian matrix
\begin{equation*}
\begin{aligned}
\mtx{Z} &= \sum_{i=1}^n X_i \vct{u}_i \vct{u}_i^*
&&\quad\text{where $X_i \sim \normal_{\R}(\beta d / n, \ \beta d / n)$ iid;} \\
	&\sim \beta \cdot \Id_d + \sqrt{\frac{\beta d}{n}} \sum_{i=1}^n \gamma_i \vct{u}_i \vct{u}_i^*
	&&\quad\text{where $\gamma_i \sim \normal_{\R}(0,1)$ iid.}
\end{aligned}
\end{equation*}
Indeed, $\Expect \mtx{Z} = \beta \cdot \Id_d$
because a 2-design satisfies the isotropy property~\eqref{eqn:untf}.
The 2-design property~\eqref{eqn:4design} also allows us to compute the variance function.
For self-adjoint $\mtx{M} \in \Sym_d(\C)$,
\begin{align*}
\Varo[\mtx{Z}](\mtx{M}) = \frac{\beta d}{n} \sum_{i=1}^n \abssq{ \ip{ \mtx{M} }{ \vct{u}_i \vct{u}_i^* } } %
	= \frac{\beta}{d+1} \left[ \fnormsq{\mtx{M}} + (\trace \mtx{M})^2 \right].
\end{align*}
The variance function determines the distribution of the random part of $\mtx{Z}$.
In view of the discussion in \cref{sec:gauss-examples}, we identify
a GUE matrix and a scalar Gaussian matrix:
\[
\mtx{Z} = \beta \cdot \Id_d + \sqrt{\frac{\beta}{d+1}}( \mtx{G}_{\gue}  + \gamma \Id_d)
\quad\text{where $\mtx{G}_{\gue}$ and $\gamma \sim \normal_{\R}(0,1)$ are independent.}
\]
Using the inequality~\eqref{eqn:gue-eig} for the GUE eigenvalues, %
\[
\Expect \lambda_{\min}(\mtx{Z}) = \beta + \sqrt{\frac{\beta}{d+1}} \cdot \Expect \big[ \lambda_{\min}(\mtx{G}_{\gue}) + \gamma \big]
	\geq \beta - 2 \sqrt{\frac{\beta d}{d+1}}
	> \beta - 2\sqrt{\beta}.
\]
Meanwhile, the weak variance satisfies
\[
\sigma_*^2(\mtx{Z}) \leq \frac{\beta}{d+1} [ \sigma_*^2(\mtx{G}_{\gue}) + \sigma_{*}^2(\gamma \Id) ] < \frac{2\beta}{d}.
\]
Instantiating \cref{thm:sampling-intro}, we arrive at the probability inequality %
\[
\Prob{ \lambda_{\min}(\mtx{Y}) \leq \beta - 2\sqrt{\beta} - t }
	\leq d \cdot \econst^{-t^2 d/(4\beta)}.
\]
Choose $t = \beta - 2\sqrt{\beta}$, set the failure probability equal to $\delta$,
and solve for the oversampling parameter $\beta$ to complete the argument.
\qed

\section{Application: Sample covariance matrices}
\label{sec:scov}

Our next application arises from high-dimensional statistics.
Suppose that we want to detect which linear marginals of a random
vector have strictly positive variance~\cite{SV13:Covariance-Estimation}.
One procedure is to draw iid copies of the random vector and to form the sample
covariance matrix.  The sample covariance matrix provides
estimates for the variance of each marginal.  How many
samples are sufficient to ensure that none of these
estimates is too small?

Using the Gaussian comparison theorem (\cref{thm:iid-intro}),
we can reproduce and extend several major results on sample covariance
matrices from the recent literature.
When the second moments and fourth moments
of the random vector are comparable, then
the sampling complexity is proportional
to the dimension of the random vector~\cite{Oli16:Lower-Tail}. %
We can also obtain useful information about
the sample covariance matrix of a very sparse
random vector~\cite{DZ24:Extreme-Singular}. %

\subsection{The sample covariance matrix}

Consider a \hilite{real} random vector $\vct{w} \in \R^d$ with
dimension $d$ that has \hilite{four} finite moments: $\Expect \norm{\vct{w}}^4 < + \infty$.
Assume that the random vector is \hilite{centered}, and introduce
the (population) covariance matrix:
\[
\Expect[ \vct{w} ] = \vct{0}
\quad\text{and} \quad
 \mtx{K} \coloneqq \Expect[ \vct{ww}^\transp ].
\]
For simplicity, we also assume that the covariance matrix $\mtx{K}$
has \hilite{full rank $d$}.  We can compute the variance
of a linear marginal of the random vector in terms of the covariance matrix:
\begin{equation} \label{eqn:marg-var}
\Var[ \ip{ \vct{a} }{ \vct{w} } ] = \vct{a}^\transp \mtx{K} \vct{a}
\quad\text{for each $\vct{a} \in \R^d$.}
\end{equation}
Our goal is to obtain \hilite{lower bounds} for the variance in every direction,
which we call the \term{variance detection} problem.

Suppose that we sample $n$ iid copies of the random vector $\vct{w}$. %
The \term{sample covariance matrix} of this data is the random psd matrix
\begin{equation} \label{eqn:scov-def}
\widehat{\mtx{K}}_n \coloneqq \frac{1}{n} \sum_{i=1}^n \vct{w}_i \vct{w}_i^\transp
\quad\text{where $\vct{w}_i \sim \vct{w}$ iid.}
\end{equation}
The sample covariance is an unbiased estimator for the true covariance:
$\Expect[ \widehat{\mtx{K}}_n ] = \mtx{K}$ for each $n \in \N$.
Moreover, the sample covariance matrix %
provides estimates for the variance
of each linear marginal~\eqref{eqn:marg-var} of the distribution.
For a parameter $\eps \in (0,1)$, with high probability,
we aspire that
\begin{equation} 
\vct{a}^\transp \widehat{\mtx{K}}_n \vct{a} \geq (1 - \eps) \cdot \vct{a}^\transp \mtx{K} \vct{a}
\quad\text{for \hilite{all} $\vct{a} \in \R^d$.}
\label{eqn:scov-marginal-lb}
\end{equation}
How many samples $n = n(\vct{w}, \eps)$ are sufficient to ensure
that the property~\eqref{eqn:scov-marginal-lb} is likely to prevail?

The answer to this question depends on the distribution of the random vector $\vct{w}$.
Since $\mtx{K}$ has rank $d$, it is \hilite{necessary} that $n \geq d$
because the rank of the sample covariance matrix $\smash{\widehat{\mtx{K}}_n}$
does not exceed the number $n$ of samples.  For a worst-case vector
that satisfies $\norm{\vct{w}} \lesssim \smash{\sqrt{d}}$,
it is necessary to draw $n \asymp d \log d$ samples~\cite{Rud99:Random-Vectors}.

A major technical challenge is to find large classes of random
vectors where the sample complexity of~\eqref{eqn:scov-marginal-lb}
is proportional to the dimension: $n \asymp d$.
This problem has already inspired a vast literature that draws
on a diverse set of technical ideas.
The Gaussian comparison inequality (\cref{thm:iid-intro}) offers a new methodology
for addressing this problem.

\begin{remark}[Sample covariance: Complex setting]
Our approach adapts to the complex setting with minimal changes.
We work in the real setting, as it is more typical in
the statistics literature.
\end{remark}

\subsection{Sample covariance: Four moment theorem}

Our first result reconstructs a prominent theorem
from high-dimensional statistics, in a form
due to Oliveira~\cite{Oli16:Lower-Tail}.
When the second and fourth moments of a random
vector are comparable, then the sample complexity
of the variance detection problem is proportional
to the dimension.
The proof appears below in~\cref{sec:scov-mom-result-pf}.

\begin{theorem}[Sample covariance matrices: Four moment theorem] \label{thm:scov-mom-result}
Suppose that $\vct{w} \in \R^d$ is a \hilite{centered} random vector
with covariance matrix $\mtx{K}$ and with four finite moments. %
For a constant $\beta \geq 1$, assume that
\begin{equation} \label{eqn:scov-mom-cond}
\Expect{} \abs{ \ip{\vct{a}}{\vct{w}}}^4
	\leq \beta^2 \cdot \big(\Expect{} \abssq{\ip{\vct{a}}{\vct{w}}} \big)^2
	\quad\text{for each $\vct{a} \in \R^d$.}
\end{equation}
For parameters $\eps, \delta \in (0,1)$, suppose that the number $n$ of samples satisfies
\begin{equation} \label{eqn:scov-mom-samples}
n \geq \frac{24 \beta^2 \cdot( d  \vee \log(2d/\delta))}{\eps^2}.
\end{equation}
Then, with probability at least $1 -\delta$, the sample covariance matrix $\smash{\widehat{\mtx{K}}_n}$
defined in~\eqref{eqn:scov-def} satisfies
\begin{equation} \label{eqn:scov-mom-result}
\vct{a}^\transp \widehat{\mtx{K}}_n \vct{a}
	> (1 - \eps) \cdot \vct{a}^\transp \mtx{K} \vct{a}
	\quad\text{for all $\vct{a} \in \R^d$.}
\end{equation}
\end{theorem}

The sample complexity~\eqref{eqn:scov-mom-samples} contains a ``startup cost''
of $n \gtrsim \beta^2 \log(d/\delta)$ samples to ensure that the failure probability is
controlled.  Usually, the sample complexity is governed by the other term:
$n \gtrsim \beta^2 d$.  The statistic $\beta$ relates the second and fourth moments
of marginals of the random vector.

The sample complexity $n \asymp d$ whenever the comparison factor $\beta$
is independent of the dimension $d$.
Many types of random vectors meet this criterion, including random vectors
with independent bounded entries, log-concave random vectors, and tensor products
of independent random vectors.
We refer to the literature~\cite{SV13:Covariance-Estimation,KM15:Bounding-Smallest,Oli16:Lower-Tail,Zhi24:Dimension-Free-Bounds}
for more examples and discussion.

\subsubsection{Prior work}

Srivastava \& Vershynin~\cite[Thm.~1.5]{SV13:Covariance-Estimation} established
the first result in the spirit of~\Cref{thm:scov-mom-result}, with
suboptimal dependence on the parameter $\eps$.
Koltchinskii \& Mendelson~\cite[Thm.~1.3]{KM15:Bounding-Smallest} obtained
improvements to the dependence on $\eps$ with slightly stricter
moment assumptions.
Oliveira~\cite[Thm.~1.1]{Oli16:Lower-Tail} obtained a version of the result
stated here, which gives the correct dependence on all the parameters.
See \cref{sec:related-psd-sum} for more discussion.

Several papers~\cite{SV13:Covariance-Estimation,KM15:Bounding-Smallest,Tik18:Sample-Covariance}
present nonasymptotic bounds on the minimum eigenvalue of the sample covariance
assuming control of the $2 + \eta$ moments, where $\eta > 0$.
These conditions are weaker than~\eqref{eqn:scov-mom-cond}.
Nevertheless, we will see that the Gaussian comparison method
can tease out structure from the random vector that is invisible
to a uniform bound on moments (\cref{sec:sparse-cov}).

\subsubsection{Proof of \cref{thm:scov-mom-result}}
\label{sec:scov-mom-result-pf}

The key idea %
is to exploit the superior
concentration of the Gaussian comparison model.
Indeed, we can invoke the most elementary
net argument to bound the norm of the Gaussian matrix.
The result we need is encapsulated
in the following lemma, which is established below in
\cref{sec:gauss-bdd-marginals-pf}.

\begin{lemma}[Gaussian matrix: Bounded variance] \label{lem:gauss-bdd-marginals}
Let $\mtx{Z} \in \Sym_d(\R)$ be symmetric and Gaussian %
with %
\begin{equation} \label{eqn:gauss-bdd-var-hyp}
\Var[ \vct{a}^\transp \mtx{Z} \vct{a} ] \leq \beta^2 \cdot \norm{\vct{a}}^2
\quad\text{for each $\vct{a} \in \R^d$.}
\end{equation}
Then %
\begin{align*}
\Expect \norm{ \mtx{Z} - \Expect \mtx{Z} } \leq \sqrt{12 \beta^2 d}
\qquad\text{and}\qquad
\sigma_*^2(\mtx{Z}) \leq \beta^2.
\end{align*}
\end{lemma}

Let us continue with the proof of \cref{thm:scov-mom-result}.
Observe that the moment condition~\eqref{eqn:scov-mom-cond}
and the conclusion~\eqref{eqn:scov-mom-result} are both
invariant under invertible linear transformations of the random vector.
Therefore, we can make the change of variables
$\vct{w} \mapsto \mtx{K}^{-1/2} \vct{w}$ to
ensure that the random vector $\vct{w}$ is \hilite{isotropic}:
$\Expect[ \vct{ww}^\transp ] = \Id_d$.

Introduce the Gaussian matrix
$\mtx{X} \in \Sym_d(\R)$ with statistics
\[
\Expect[ \mtx{X} ] = \Expect[ \vct{ww}^\transp ] = \Id_d
\qquad\text{and}\qquad
\Varo[ \mtx{X} ] = \Mo[ \vct{ww}^\transp ].
\]
Note that the variances of the quadratic forms satisfy %
\[
\Var[ \vct{a}^\transp \mtx{X} \vct{a} ]
	= \Varo[ \mtx{X} ](\vct{aa}^\transp)
	= \Mo[\vct{ww}^\transp](\vct{aa}^\transp)
	= \Expect \abs{ \ip{\vct{a}}{\vct{w}} }^4
	\leq \beta^2 \cdot \normsq{\vct{a}}.
\]
The last inequality is the hypothesis~\eqref{eqn:scov-mom-cond}
of the theorem.

According to~\cref{thm:iid-intro},
the Gaussian comparison model $\mtx{Z} \in \Sym_d(\R)$
for the sample covariance matrix $\smash{\widehat{\mtx{K}}_n}$ is
\[
\mtx{Z} = \frac{1}{n} \sum_{i=1}^n \mtx{X}_i
	\sim (\Expect \mtx{X}) + \frac{1}{\sqrt{n}} (\mtx{X} - \Expect \mtx{X})
	\quad\text{where $\mtx{X}_i \sim \mtx{X}$ iid.} %
\]
Using Lemma~\ref{lem:gauss-bdd-marginals},
\[
\Expect \lambda_{\min}(\mtx{Z}) \geq \lambda_{\min}( \Expect \mtx{X} ) - n^{-1/2} \Expect \norm{ \mtx{X} - \Expect \mtx{X} } \geq 1 - \sqrt{12 \beta^2 d / n}.
\]
Meanwhile, the weak variance $\sigma_*^2(\mtx{Z}) \leq \beta^2 / n$.  Activating \cref{thm:iid-intro},
we discover that
\[
\Prob{ \lambda_{\min}(\widehat{\mtx{K}}_n) \geq 1 - \sqrt{12 \beta^2 d/n} - \eps/\sqrt{12} }
	\leq 2d \cdot \econst^{- n \eps^2 / (24 \beta^2)}.
\]
Introduce the stated bound~\eqref{eqn:scov-mom-samples} for the sample complexity $n$ and simplify
the expression.  The Rayleigh variational formula for the minimum eigenvalue yields
the conclusion~\eqref{eqn:scov-mom-result}.
\qed

\subsubsection{Proof of \cref{lem:gauss-bdd-marginals}}
\label{sec:gauss-bdd-marginals-pf}

Without loss, we may assume that $\mtx{Z}$ is centered.
The statement about the weak variance is immediate.
Indeed, by its definition~\eqref{eqn:weak-var} and the hypothesis~\eqref{eqn:gauss-bdd-var-hyp} of the lemma,
\[
\sigma_*^2(\mtx{Z}) = \max\nolimits_{\norm{\vct{u}}=1} \Var[ \vct{u}^\transp \mtx{Z} \vct{u} ]
	\leq \beta^2.
\]
For the norm bound, recall that the unit sphere in $\R^d$ admits an $\eps$-net $\coll{N}$
with at most $(1 + 2/\eps)^d$ points~\cite[Cor.~4.2.13]{Ver18:High-Dimensional-Probability}.
By the standard discretization argument~\cite[Sec.~4.4]{Ver18:High-Dimensional-Probability},
\[
\Expect \norm{ \mtx{Z} } \leq \Expect \max\nolimits_{\vct{u} \in \coll{N}} \abs{\vct{u}^\transp \mtx{Z} \vct{u}}
	+ 2\eps \cdot \Expect \norm{ \mtx{Z} }.
\]
Since $\Var[ \vct{u}^\transp \mtx{Z} \vct{u} ] \leq \beta^2$ uniformly,
the maximal inequality~\cite[Lem.~5.1]{vH16:Probability-High}
for Gaussian random variables asserts that
\[
\Expect \max\nolimits_{\vct{u} \in \coll{N}} \abs{\vct{u}^\transp \mtx{Z} \vct{u}}
	\leq \sqrt{ 2 \beta^2 \log( 2 \abs{\coll{N}} ) }
	\leq \sqrt{ 2 \beta^2 (\log 2 + d \log (1 + 2/\eps)) }.
\]
The factor two inside the logarithm takes care of the absolute value
in the maximum.  Combine the last two displays, and solve for the expected norm to obtain
\[
\Expect \norm{ \mtx{Z} } \leq\frac{1}{1-2\eps} \sqrt{ 2 \beta^2 (\log 2 + d \log (1 + 2/\eps)) }.
\]
Select $\eps = 1/12$, and note that the numerical constant is bounded by $\sqrt{12}$ for all $d \geq 1$.
\qed

\subsection{Example: Sparse covariance matrices}
\label{sec:sparse-cov}

For a more challenging example that goes beyond the scope of \cref{thm:scov-mom-result},
we consider variance detection for a \hilite{sparse} random vector.  For simplicity,
we treat the case of iid entries, but we only require four finite moments.

Fix the dimension $d$, and let $\zeta \in (0, d]$ be a sparsity parameter.
Introduce a real random variable $\psi$ that is standardized and has bounded
fourth moment:
\[
\Expect[ \psi ] = 0
\quad\text{and}\quad
\Var[\psi] = 1
\quad\text{and}\quad
\Expect \abs{\psi}^4 = C.
\]
Construct a sparse random vector $\vct{w} \in \R^d$ with iid entries:
\[
\vct{w} = \sqrt{\frac{d}{\smash{\zeta}}} \begin{bmatrix} \xi_1 \psi_1 \\ \vdots \\ \xi_d \psi_d \end{bmatrix}
\quad\text{where}\quad
\begin{aligned}
& \text{$\xi_i \sim \bernoulli(\zeta / d)$ iid}; \\
& \text{$\psi_i \sim \psi$ iid}. \\
\end{aligned}
\]
It is straightforward to check that $\vct{w}$ is centered and isotropic, and $\vct{w}$
has $\zeta$ nonzero entries on average.
How does the sample complexity of the variance detection problem depend
on the sparsity?

\begin{theorem}[Sample covariance: Sparse vectors] \label{thm:sparse-cov}
Instate the prevailing notation, and fix parameters $\eps, \delta \in (0,1)$.
Suppose that the number $n$ of samples satisfies
\begin{equation} \label{eqn:sparse-cov-samples}
n \geq \frac{25 d \cdot (1 \vee (2C \zeta^{-1} \log(2d/\delta)))}{\eps^2}.
\end{equation}
With probability at least $1 - \delta$, the sample
covariance matrix of the sparse random vector $\vct{w}$
has minimum eigenvalue $\lambda_{\min}(\smash{\widehat{\mtx{K}}_n}) > 1 - \eps$.
\end{theorem}

\Cref{thm:sparse-cov} handles two different behavioral regimes.
\begin{itemize}
\item	When $\zeta \geq 2C \log d$, the sample complexity $n \asymp d$.

\item	When $\zeta \leq 2C \log d$, the sample complexity $n \asymp C \zeta^{-1} d \log d$.
\end{itemize}

\noindent
The linear sample complexity in the first regime is quite difficult
to achieve, especially when the distribution $\psi$ has few moments.
In the ultra-sparse case ($\zeta \leq 1$), similar results follow
from classic matrix concentration inequalities, such as~\cref{fact:matrix-epz}.

We were unable to locate any prior work that matches \cref{thm:sparse-cov},
although there are several closely related results.
In particular, assume that the distribution $\psi$ is \hilite{bounded}.
In this case, the paper~\cite[Thm.~3.5]{CHZ22:Nonasymptotic-Concentration}
achieves the same complexity estimate~\eqref{eqn:sparse-cov-samples}
for the {expectation} of the minimum eigenvalue
of the sample covariance matrix. %
In the regime where the sparsity $\zeta \gtrsim \log d$ and $d/n$ is constant,
Dumitriu \& Zhu~\cite[Thm.~2.9]{DZ24:Extreme-Singular} achieve stronger estimates for
the expectation of the minimum eigenvalue that are sufficient to attain
the Bai--Yin limit, which is out of reach for our method.
For a thorough literature review, refer to the paper~\cite{DZ24:Extreme-Singular}.

\subsubsection{Proof of \cref{thm:sparse-cov}}

By elementary considerations, we can calculate the second moment function
of the random matrix $\vct{ww}^\transp$.  Indeed, for symmetric $\mtx{M} \in \Sym_d(\R)$,
\begin{align*}
\Mo[ \vct{ww}^\transp ] (\mtx{M}) = \Expect\big[ (\vct{w}^\transp \mtx{M} \vct{w})^2 \big]
	&= \sum_{i \neq j} (2 m_{ij}^2 + m_{ii} m_{jj}) + \frac{Cd}{\zeta} \sum_{i=1}^d m_{ii}^2 \\
	&\leq 2 \fnormsq{ \mtx{M} } + (\trace \mtx{M})^2 + \frac{Cd}{\zeta} \sum_{i=1}^d m_{ii}^2.
\end{align*}
A natural comparison model for $\vct{ww}^\transp$ is the
Gaussian random matrix $\mtx{X} \in \Sym_d(\R)$ with statistics
\begin{align*}
\Expect[ \mtx{X} ] &= \Expect[ \vct{ww}^\transp ] = \Id_d; \\
\Varo[\mtx{X}](\mtx{M}) &= 2 \fnormsq{ \mtx{M} } + (\trace \mtx{M})^2 + \frac{Cd}{\zeta} \sum_{i=1}^d m_{ii}^2
	\geq \Mo[ \vct{ww}^\transp ](\mtx{M}).
\end{align*}
Referring back to \cref{sec:gauss-examples}, we quickly determine the
Gaussian matrix with these statistics:
\[
\mtx{X} = \Id_d + \mtx{G}_{\goe} + \gamma \Id_d + \sqrt{Cd/\zeta} \, \mtx{D}.
\]
The matrix $\mtx{D}$ is a diagonal matrix of real standard Gaussian variables,
independent from the GOE matrix $\mtx{G}_{\goe}$ and the scalar
Gaussian $\gamma \sim \normal_{\R}(0,1)$.

Therefore, the Gaussian comparison $\mtx{Z} \in \Sym_d(\R)$
for the sample covariance matrix $\smash{\widehat{\mtx{K}}_n}$ is
\[
\mtx{Z} \sim (\Expect \mtx{X}) + n^{-1/2} (\mtx{X} - \Expect \mtx{X})
	= \Id_d + n^{-1/2} \mtx{G}_{\goe}
	+ n^{-1/2} \gamma \Id_d
	+ \sqrt{Cd /(\zeta n)}\, \mtx{D}.
\]
We quickly calculate the spectral statistics:
\begin{align*}
\Expect \lambda_{\min}(\mtx{Z}) \geq 1 - 2 \sqrt{\frac{d}{\smash{n}}} - \sqrt{\frac{2Cd \log d}{\smash{\zeta} n}}
\qquad\text{and}\qquad %
\sigma_*^2(\mtx{Z}) \leq \frac{3 + Cd/\zeta}{n}.
\end{align*}
Invoking \cref{thm:iid-intro},
\[
\Prob{ \lambda_{\min}(\widehat{\mtx{K}}_n) \leq 1 - 2 \sqrt{\frac{d}{n}} - \sqrt{\frac{2Cd \log d}{\smash{\zeta n}}} - \frac{2 \eps}{5} }
	\leq 2d \cdot \exp\left( \frac{- 2 n \eps^2 \zeta}{ 25 d (C + 3\zeta/d) } \right).
\]
Introduce the sample complexity parameter from~\eqref{eqn:sparse-cov-samples}, and simplify to reach the stated result.
\qed

\section{Application: Randomized subspace injections}
\label{sec:subspace-inj}

In numerical linear algebra,
we can devise \hilite{randomized} algorithms
for large-scale matrix computations; for example, see~\cite{MT20:Randomized-Numerical}.
The design and analysis of these algorithms often
relies on methods from random matrix theory.
In fact, the tools in the present paper were developed to treat
challenging mathematical problems from this field.

A \term{randomized subspace injection} is a random linear map whose action preserves
the dimension of a fixed subspace~\cite{Sar06:Improved-Approximation,OT18:Universality-Laws}. %
These injections serve as an important
building block for randomized linear algebra
algorithms~\cite{Woo14:Sketching-Tool,MT20:Randomized-Numerical,MDM+23:Randomized-Numerical}.
In this section, we employ the Gaussian comparison theorem (\cref{thm:iid-intro})
to develop a new analysis of subspace injections based on very sparse random matrices~\cite{CW13:Low-Rank-Approximation,NN13:OSNAP-Faster}.  %
The result largely settles an open question of Nelson \& Nguyen~\cite{NN13:OSNAP-Faster,NN14:Lower-Bounds} %
that has generated a substantial literature;
\Cref{sec:subspace-prior} summarizes the prior work.

\subsection{Subspace injections}

For a given subspace, a subspace injection is a linear map that does not annihilate
any vector in that subspace~\cite[Sec.~7.1]{OT18:Universality-Laws}.

\begin{definition}[Subspace injection] \label{def:subspace-inj}
Fix a matrix $\mtx{Q} \in \R^{n \times d}$ with $d$ orthonormal columns.
For a fixed \term{contraction factor} $\alpha > 0$,
suppose that $\mtx{\Phi} \in \R^{k \times n}$ is a matrix that satisfies
\begin{equation} \label{eqn:subspace-inj}
\norm{ \mtx{\Phi} \mtx{Q} \vct{u} }^2
	\geq \alpha \cdot \norm{ \vct{u} }^2
	\quad\text{for all $\vct{u} \in \R^d$.} %
\end{equation}
When~\eqref{eqn:subspace-inj} holds, we say that $\mtx{\Phi}$
is an \term{$\alpha$-subspace injection} for the $d$-dimensional
range of $\mtx{Q}$ with \term{embedding dimension} $k$.
If~\eqref{eqn:subspace-inj} holds only when $\alpha = 0$,
then $\mtx{\Phi}$ is not a subspace injection for $\mtx{Q}$.
\end{definition}

The subspace injection property~\eqref{eqn:subspace-inj} has an equivalent
spectral formulation:
\begin{equation} \label{eqn:subspace-inj-lmin}
\lambda_{\min}((\mtx{\Phi} \mtx{Q})^\transp (\mtx{\Phi} \mtx{Q})) \geq \alpha > 0.
\end{equation}
Of course, a necessary condition for
the subspace injection property~\eqref{eqn:subspace-inj} or~\eqref{eqn:subspace-inj-lmin}
is that the embedding dimension exceeds the subspace dimension: $k \geq d$.
We want to design subspace injections where the embedding dimension
$k \asymp d$.

\begin{remark}[Subspace embedding~\cite{Sar06:Improved-Approximation}] \label{rem:subspace-embedding}
For an orthonormal matrix $\mtx{Q} \in \R^{n \times d}$,
suppose that $\mtx{\Phi} \in \R^{k \times n}$ satisfies the two-sided bound
\begin{equation} \label{eqn:subspace-exp}
\alpha \cdot \norm{ \vct{u} }^2 \leq \norm{ \mtx{\Phi Q} \vct{u} }^2 \leq \beta \cdot \norm{ \vct{u} }^2
\quad\text{for all $\vct{u} \in \R^d$.}
\end{equation}
Then we say that $\mtx{\Phi}$ is an \term{$(\alpha, \beta)$-subspace embedding}
for the range of $\mtx{Q}$.  While it is critical that the contraction factor $\alpha > 0$,
weak control on the ratio $\beta / \alpha$ suffices for most applications.
\end{remark}

\subsection{Randomized subspace injections}

In many applications to computational linear algebra, we must construct a subspace
injection without detailed knowledge of the fixed subspace $\mtx{Q}$.
We can achieve this goal by drawing the matrix $\mtx{\Phi}$ \hilite{at random}.

Consider an \hilite{isotropic} random vector $\vct{\phi} \in \R^n$;
that is, $\Expect[ \vct{\phi} \vct{\phi}^\transp ] = \Id_n$.
The distribution of the random vector $\vct{\phi}$ is an algorithmic design choice.
For an embedding dimension $k$, construct the random matrix
\begin{equation} \label{eqn:iid-subspace-inj}
\mtx{\Phi} = \frac{1}{\sqrt{k}}
	\begin{bmatrix} \text{} & \vct{\phi}_1^\transp & \text{} \\
	&\vdots& \\
	\text{} & \vct{\phi}_k^\transp & \text{} \end{bmatrix}
	\in \R^{k \times n}
	\quad\text{where $\vct{\phi}_i \sim \vct{\phi}$ iid.}
\end{equation}
Since the rows are isotropic, the random matrix is an isometry on average:
\begin{equation} \label{eqn:subspace-isom}
\Expect \norm{ \mtx{\Phi} \vct{u} }^2 = \norm{ \vct{u} }^2
\quad\text{for each $\vct{u} \in \R^n$.}
\end{equation}
The normalization~\eqref{eqn:subspace-isom} %
is chosen so that $\mtx{\Phi}$ typically has a contraction factor
$\alpha \leq 1$.

For a fixed subspace $\mtx{Q} \in \R^{d \times n}$ and embedding dimension $k$,
we want to demonstrate that the random matrix $\mtx{\Phi}$ is a subspace injection
for the range of $\mtx{Q}$ \hilite{with high probability}.  Quantitatively,
we need to understand how the contraction factor $\alpha$ depends on the embedding dimension $k$.
In view of~\eqref{eqn:subspace-inj-lmin}, it suffices to establish
a lower bound on the minimum eigenvalue of the random psd matrix
\begin{equation} \label{eqn:subspace-inj-psd}
\mtx{Y} \coloneqq (\mtx{\Phi} \mtx{Q})^\transp (\mtx{\Phi} \mtx{Q})
	= \frac{1}{k} \sum_{i=1}^k \mtx{Q}^\transp (\vct{\phi}_i \vct{\phi}_i^\transp) \mtx{Q}.
\end{equation}
By the construction~\eqref{eqn:iid-subspace-inj},
the matrix $\mtx{Y}$ is a sum of iid random psd matrices,
so we can activate our Gaussian comparison tools (\cref{thm:iid-intro}).

\subsection{Sparse dimension reduction maps}

In computational applications, it is desirable to employ \hilite{very sparse} random matrices
as subspace injections.
While these maps are widely used~\cite{MT20:Randomized-Numerical,MDM+23:Randomized-Numerical}, %
no existing analysis justifies the typical parameter choices.
We outline prior work in \cref{sec:subspace-prior}.

Choose an embedding dimension $k \geq d$, %
and fix the sparsity parameter $\zeta \in (0, k]$.
Construct the random vector
\begin{equation} \label{eqn:sparse-inj-vec}
\vct{\phi}^\transp = \sqrt{\frac{k}{\smash{\zeta}}} \cdot \begin{bmatrix}  \xi_1 \psi_1 & \hdots & \xi_n \psi_n \end{bmatrix} \in \R^n
\quad\text{where}\quad
\begin{aligned}
&\text{$\xi_i \sim \bernoulli(\zeta / k)$ iid;} \\
&\text{$\psi_i \sim \uniform\{\pm 1\}$ iid.}
\end{aligned}
\end{equation}
It is easy to verify that $\vct{\phi}$ is centered and
isotropic.
Form the random matrix $\mtx{\Phi}$, as in~\eqref{eqn:iid-subspace-inj}.
Observe that $\mtx{\Phi}$ has an average of $\zeta$ nonzero entries per column,
for a total of $\zeta n$ nonzero entries on average.
We inquire when this sparse random matrix $\mtx{\Phi}$ serves as a subspace injection
for a fixed subspace.  To achieve an embedding dimension $k \asymp d$,
what is the minimal sparsity $\zeta$?

The result depends on a geometric property of the subspace.
Define the \term{coherence} $\mu(\mtx{Q})$
of an orthonormal matrix \smash{$\mtx{Q} \in \R^{n \times d}$} via
\begin{equation} \label{eqn:coherence}
\mu(\mtx{Q}) \coloneqq \max\{ \normsq{ \mathbf{e}_i^\transp \mtx{Q} } : i = 1, \dots, n \}.
\end{equation}
The coherence describes the alignment of the range of $\mtx{Q}$ with the
standard coordinate basis $(\mathbf{e}_1, \dots, \mathbf{e}_n)$.
It satisfies the inequalities $d/n \leq \mu(\mtx{Q}) \leq 1$.

\begin{theorem}[Subspace injection: Sparse random matrix] \label{thm:sparse-injection}
Fix an orthonormal matrix $\mtx{Q} \in \R^{n \times d}$ with coherence $\mu(\mtx{Q})$,
defined in~\eqref{eqn:coherence}.
For parameters $\eps, \delta \in (0,1)$, choose  %
\begin{equation} \label{eqn:subspace-inj-params}
k \geq \frac{16 (d \vee 6 \log(2d/\delta))}{\eps^2}
\qquad\text{and}\qquad
\zeta \geq \frac{32 \mu(\mtx{Q}) \log(2d / \delta)}{\eps^2}.
\end{equation}
Construct the sparse random matrix $\mtx{\Phi} \in \R^{k \times n}$
using~\eqref{eqn:iid-subspace-inj} and~\eqref{eqn:sparse-inj-vec}.
Then $\mtx{\Phi}$ is a $(1-\eps)$-subspace injection for $\mtx{Q}$
with probability at least $1 - \delta$.
\end{theorem}

\noindent
The proof of \cref{thm:sparse-injection} appears below in \cref{sec:sparse-inj-pf}.
It follows from a quick application of Gaussian comparison, in the
same manner as the result for sparse sample covariance matrices (\cref{thm:sparse-cov}).

\cref{thm:sparse-injection} has several attractive features.
First, we have achieved the optimal embedding dimension $k \asymp d$.
Second, the sparsity level $\zeta \lesssim \log d$ for every subspace,
nearly matching known lower bounds for sparse subspace embeddings;
see~\eqref{eqn:nelson-nguyen} below.
Third, the result shows that we can even reduce the sparsity for subspaces
with small coherence.  The ultimate limit, for a minimally coherent subspace,
is $\zeta \gtrsim (d \log d) / n$, which allows for a random subspace
injection with just $\mathcal{O}(d \log d)$ nonzero entries.
Further improvements to the sparsity are only possible for special subspaces
(e.g., when the rows of $\mtx{Q}$ compose a complex projective 2-design).

For a sparse random matrix %
with iid entries,
the $\eps^{-2}$ dependence in the embedding
dimension $k$ and the sparsity $\zeta$ seems to be
necessary~\cite[Rem.~3.6]{CDD24:Optimal-Oblivious}.
For most applications, the parameter $\eps$ is a constant (say, $\eps = 1/2$),
so the poor scaling in $\eps$ is not a significant concern.
In contrast, the subspace dimension $d$ can be quite
large, so it is essential to obtain the minimal dependence on $d$.

\begin{remark}[Subspace embedding: Sparse random matrix] \label{rem:sparse-embed}
To verify the subspace embedding property (\cref{rem:subspace-embedding})
of the iid sparse random matrix $\mtx{\Phi}$ in \cref{thm:sparse-injection},
we can easily obtain adequate upper bounds for the dilation factor $\beta$.
For example, with $k \asymp d$ and $\zeta \asymp \mu(\mtx{Q}) \log d$, %
\[
\Expect \beta
	= \Expect \normsq{ \mtx{\Phi} \mtx{Q} }
	\lesssim \log d.
\]
This statement follows quickly when we apply the matrix Bernstein
inequality~\cite[Thm.~6.1.1]{Tro15:Introduction-Matrix}
to the decomposition
\[
\mtx{\Phi} \mtx{Q} = \sqrt{\frac{1}{\smash{\zeta}}} \sum_{i=1}^k \sum_{j=1}^n \xi_{ij} \psi_{ij} \mathbf{E}_{ij} \mtx{Q}
\quad\text{where}\quad
\begin{aligned}
&\text{$\xi_{ij} \sim \bernoulli(\zeta / k)$ iid;} \\
&\text{$\psi_{ij} \sim \uniform\{\pm 1\}$ iid.}
\end{aligned}
\]
For these parameter choices, the remaining question
is whether we can reach the bound $\Expect \beta \leq \mathrm{Const}$.
\end{remark}

\subsubsection{Prior work and discussion}
\label{sec:subspace-prior}

Sarl{\'o}s~\cite{Sar06:Improved-Approximation} introduced the definition 
of a subspace \hilite{embedding} (\cref{rem:subspace-embedding}).
Clarkson \& Woodruff~\cite{CW13:Low-Rank-Approximation} proposed
the first construction of a \hilite{sparse} subspace embedding.
Soon after, Nelson \& Nguyen~\cite{NN13:OSNAP-Faster} identified more effective constructions,
including the iid entry model in \cref{thm:sparse-injection} and a related model,
called a \term{fixed-sparsity subspace embedding}, that has \hilite{exactly}
$\zeta$ nonzero entries per \hilite{column}.
There are several variants of these sparse subspace embeddings,
which offer slightly different advantages and disadvantages.
For brevity, we summarize results without listing
the details of the embedding constructions.

What are the opportunities for and limitations on sparse subspace embeddings?
For any fixed-sparsity subspace {embedding} with conditioning ratio $\beta / \alpha \eqqcolon 1 + \eps$,
Nelson \& Nguyen~\cite[Thm.~13]{NN14:Lower-Bounds} established \hilite{lower}
bounds for the parameters:
\begin{equation} \label{eqn:nelson-nguyen}
k \gtrsim d / \eps^{2}
\quad\text{and}\quad
\zeta \gtrsim (\log d) / (\eps \log \log d).
\end{equation}
Bourgain et al.~\cite[Thm.~5]{BDN15:Toward-Unified} obtained \hilite{upper} bounds
for the parameters of sparse subspace embeddings that identified the role of the
coherence statistic $\mu(\mtx{Q})$.
Using matrix concentration tools, Cohen~\cite[Thm.~4.2]{Coh16:Nearly-Tight-Oblivious}
obtained upper bounds for fixed-sparsity subspace embeddings:
\[
k \lesssim (d \log d) / \eps^2
\quad\text{and}\quad
\zeta \lesssim (\log d) / \eps.
\]
For a long time, Cohen's analysis was the best available, and
it has remained a vexing problem to obtain an upper bound that
matches the minimal dependence~\eqref{eqn:nelson-nguyen}.

In applications, constants and logarithms are important.
Based on computer experiments, Tropp et al.~\cite{TYUC19:Streaming-Low-Rank}
recommended the following parameter settings for fixed-sparsity subspace embeddings:
\[
k = 2 d
\quad\text{and}\quad
\zeta = 8.
\]
These parameters work well, but they lack theoretical justification. %
Regardless, fixed-sparsity subspace embeddings are widely used in
computational linear algebra~\cite{MT20:Randomized-Numerical,MDM+23:Randomized-Numerical}.

The last few years have witnessed a burst of new theoretical activity.
Cartis et al.~\cite{CFS21:Hashing-Embeddings} established upper bounds
that improve over the Bourgain et al.~result for subspaces with sufficiently
small coherence.
Chennakod et al.~\cite[Thm.~1.2]{CDDR24:Optimal-Embedding} removed the $\log d$
factor from the embedding dimension $k$ at the cost of higher sparsity $\zeta$
by invoking results of Brailovskaya \& van Handel~\cite{BvH24:Universality-Sharp}.
In the last few months, Chennakod et al.~\cite[Thm.~3.4]{CDD24:Optimal-Oblivious}
made further improvements to their arguments, reaching the following guarantee for fixed-sparsity
subspace embeddings:
\[
k \asymp d / \eps^2
\quad\text{and}\quad
\zeta \lesssim \log^2(d/\eps) / \eps + \log^3(d / \eps).
\]
The latter result is the state of the art.
While the embedding dimension $k$ has the optimal form,
the excess logarithmic factors in the sparsity $\zeta$
remain a serious limitation. %

As outlined, the prior work has focused on sparse subspace \hilite{embeddings}
(\cref{rem:subspace-embedding}), rather than sparse subspace \hilite{injections}
(\cref{def:subspace-inj}).
Nevertheless, the injection property is by far the more important feature,
both in theory and in practice; see~\cite[Sec.~7.1]{OT18:Universality-Laws} or~\cite[Sec.~8]{MT20:Randomized-Numerical}.
Our result (\cref{thm:sparse-injection}) is the first to prove that
sparse random matrices serve as subspace injections with (essentially)
the minimal dependence~\eqref{eqn:nelson-nguyen} on the subspace dimension $d$ in both the embedding dimension $k$
and the sparsity $\zeta$.  The plain role of the subspace coherence $\mu(\mtx{Q})$ is an added bonus.

\begin{remark}[Fixed-sparsity subspace injection]
We have treated the simplest model for a sparse subspace injection,
where the random matrix $\mtx{\Phi}$ has \hilite{iid entries}.
The Gaussian comparison theorem also allows us to study the injection
properties of a certain class of fixed-sparsity random matrices~\cite[Def.~3.2]{CDD24:Optimal-Oblivious}.
To obtain a $(1-\eps)$-subspace injection, it is sufficient that
\[
k \asymp d / \eps^2
\qquad\text{and}\quad
\zeta \lesssim (\log d) / \eps^2.
\]
As compared with \cref{thm:sparse-injection}, the subspace coherence $\mu(\mtx{Q})$
does not appear in this result.
As compared with Chennakod et al.~\cite{CDD24:Optimal-Oblivious},
we have reduced the dependence on $\log d$ significantly, at the
cost of a worse dependence on $\eps$.
\end{remark}

\subsubsection{Proof of \cref{thm:sparse-injection}}
\label{sec:sparse-inj-pf}

To apply the Gaussian comparison theorem, we need to construct
an appropriate comparison model.  We can accomplish this goal
in a few short steps, building on the calculations for the
sparse covariance matrix (\cref{thm:sparse-cov}).

As before, the second moment function of the random matrix
$\mtx{W} = \vct{\phi} \vct{\phi}^\transp \in \Sym_n(\R)$ takes the form
\begin{align*}
\Mo[ \mtx{W} ](\mtx{M})
	= \sum_{i \neq j}  (2 m_{ij}^2 + m_{ii} m_{jj}) + \frac{k}{\zeta} \sum_{i=1}^n m_{ii}^2
	\leq 2 \fnormsq{\mtx{M}} + (\trace \mtx{M})^2 + \frac{k}{\zeta} \sum_{i=1}^n m_{ii}^2.
\end{align*}
For the random matrix $\mtx{W}$, we construct a Gaussian comparison model $\mtx{X} \in \Sym_d(\R)$
with the statistics
\begin{align*}
\Expect[ \mtx{X} ] &= \Expect[ \mtx{W} ] = \Id_n; \\
\Varo[\mtx{X}](\mtx{M})
	&= 2 \fnormsq{\mtx{M}} + (\trace \mtx{M})^2 + \frac{k}{\zeta} \sum_{i=1}^n m_{ii}^2
	\geq \Mo[ \mtx{W} ](\mtx{M}).
\end{align*}
Therefore,
$\mtx{X} \sim \Id_n + \smash{\mtx{G}_{\goe}^{(n)}} + \gamma \Id_n + \sqrt{k/\zeta}\, \mtx{D}_n$.
As usual, the GOE matrix $\smash{\mtx{G}_{\goe}^{(n)}} \in \Sym_n(\R)$,
the standard normal variable $\gamma \sim \normal_{\R}(0,1)$,
and the standard normal diagonal matrix $\mtx{D}_n \in \Sym_n(\R)$ are independent.

Since first and second moments are equivariant (\cref{prop:mom-equi}),
the Gaussian random matrix $\mtx{Q}^\transp \mtx{X} \mtx{Q}$ serves as
a comparison model for $\mtx{Q}^\transp \mtx{W} \mtx{Q}$.  That is,
\[
\Expect[ \mtx{Q}^\transp \mtx{X} \mtx{Q} ] = \Expect[ \mtx{Q}^\transp \mtx{W} \mtx{Q} ]
\qquad\text{and}\qquad
\Varo[ \mtx{Q}^\transp \mtx{X} \mtx{Q} ] \geq \Mo[ \mtx{Q}^\transp \mtx{W} \mtx{Q} ].
\]
The matrix $\mtx{Q} \in \R^{n \times d}$ has orthonormal columns, so
\[
\mtx{Q}^\transp \mtx{X} \mtx{Q} \sim \Id_d + \mtx{G}_{\goe}^{(d)}
	+ \gamma \Id_d + \sqrt{k/\zeta} \, \mtx{Q}^\transp \mtx{D}_n \mtx{Q}. %
\]
We have exploited the fact~\eqref{eqn:goe-invar} that the GOE distribution is
invariant under rotations to confirm that the matrix $\smash{\mtx{G}_{\goe}^{(d)}} \in \Sym_d(\R)$
also follows the GOE distribution.

Next, we construct a comparison model $\mtx{Z}$ for the random matrix $\mtx{Y}$,
appearing in~\eqref{eqn:subspace-inj-psd}, which captures the subspace injection property.
Since $\mtx{Y} = k^{-1} \sum_{i=1}^k \mtx{Q}^\transp \mtx{W}_i \mtx{Q}$
for iid $\mtx{W}_i \sim \mtx{W}$, we quickly determine that
\begin{align*}
\mtx{Z} &\sim \mtx{Q}^\transp \big[ (\Expect \mtx{X}) + k^{-1/2} (\mtx{X} - \Expect \mtx{X}) \big] \mtx{Q} \\
	&\sim \Id_d + k^{-1/2} \mtx{G}_{\goe}^{(d)} + \gamma k^{-1/2} \Id_d
	+ \zeta^{-1/2} \mtx{Q}^\transp \mtx{D}_n \mtx{Q}. %
\end{align*}
To finish the proof, we need to compute spectral statistics.

To that end, express the compression of the diagonal matrix as a Gaussian series:
\[
\mtx{Q}^\transp \mtx{D}_n \mtx{Q} = \sum_{i=1}^n \gamma_i \vct{q}_i \vct{q}_i^\transp \in \Sym_d(\R)
\quad\text{where $\gamma_i \sim \normal_{\R}(0,1)$ iid.}
\]
The vector $\vct{q}_i^\transp \in \R^d$ is the $i$th \hilite{row} of $\mtx{Q}$.
The matrix variance satisfies
\[
\sigma^2( \mtx{Q}^\transp \mtx{D}_n \mtx{Q} )
	= \lnorm{ \sum_{i=1}^n (\vct{q}_i \vct{q}_i^\transp)^2 }
	= \lnorm{ \sum_{i=1}^n \normsq{ \vct{q}_i } \vct{q}_i \vct{q}_i^\transp}
	\leq \mu(\mtx{Q}) \cdot \norm{ \mtx{Q}^\transp \mtx{Q} }
	= \mu(\mtx{Q}).
\]
where $\mu(\mtx{Q})$ is the coherence~\eqref{eqn:coherence} of the subspace.
The matrix Khinchin inequality (\cref{fact:nck}) provides
\[
\Expect \lambda_{\max}(\mtx{Q}^\transp \mtx{D}_n \mtx{Q})
	\leq \sqrt{ 2 \mu(\mtx{Q}) \log d }.
\]
Owing to the comparison~\eqref{eqn:intro-var-wvar}, the weak variance $\sigma_*^2(\mtx{Q}^\transp \mtx{D}_n \mtx{Q}) \leq \sigma^2(\mtx{Q}^\transp \mtx{D}_n \mtx{Q}) \leq \mu(\mtx{Q})$.

Last, bound the expected minimum eigenvalue and weak variance of the comparison model:
\begin{align*}
\Expect \lambda_{\min}(\mtx{Z}) &\geq 1 - 2\sqrt{d/k} - \sqrt{2 \zeta^{-1} \mu(\mtx{Q}) \log d}; \\
\sigma_*^2( \mtx{Z} ) &\leq 3/k + \mu(\mtx{Q}) / \zeta.
\end{align*}
Activating \cref{thm:iid-intro},
\[
\Prob{ \lambda_{\min}(\mtx{Y}) \geq 1 - 2\sqrt{d/k} - \sqrt{2 \zeta^{-1} \mu(\mtx{Q}) \log d} - \eps/4 }
	\leq 2d \cdot \exp\left( \frac{-\eps^2 / 32}{3/k + \mu(\mtx{Q}) / \zeta} \right).
\]
Introduce the parameter choices~\eqref{eqn:subspace-inj-params}
and simplify to reach the stated result.
\qed

\section{Gaussian comparison: Positive scalar sums}
\label{sec:scalar-pf}

The technical development commences in this section.
As a warmup, we develop a proof of the lower tail bound
for an independent sum of positive scalar random variables.

\begin{theorem}[Positive sum: Gaussian lower tail] \label{thm:positive-sum}
Consider an \hilite{independent} family $(W_1, \dots, W_n)$ of
\hilite{nonnegative}, square-integrable, real random variables:
$W_i \geq 0$ and $\Expect W_i^2 < + \infty$.
Introduce the sum of the random variables, along with the sum of second moments:
\[
X \coloneqq \sum_{i=1}^n W_i
\quad\text{and}\quad
L_2 \coloneqq \sum_{i=1}^n \Expect W_i^2.
\]
Then the lower tail of the sum $X$ satisfies
\[
\Prob{ X \leq \Expect X - t } \leq \econst^{-t^2 / (2L_2)}
\quad\text{for all $t \geq 0$.}
\]
\end{theorem}

\begin{proof}
Introduce the mgf of the \hilite{lower tail} of the \hilite{centered} sum:
\begin{equation} \label{eqn:scalar-mgf}
\mgf_X(\theta) \coloneqq \Expect \econst^{ - \theta (X - \Expect X)}
\quad\text{for $\theta \geq 0$.}
\end{equation}
\Cref{prop:scalar-mgf}, below, states that $\log \mgf_X(\theta) \leq \theta^2 L_2/2$ for all $\theta \geq 0$.
The Laplace transform method~\cite[Sec.~2.2]{BLM13:Concentration-Inequalities}
yields the tail bound
\[
\Prob{ X - \Expect X \leq -t } \leq \inf\nolimits_{\theta > 0} \econst^{-\theta t + \log \mgf_X(\theta)}
	\leq \inf\nolimits_{\theta > 0} \econst^{-\theta t + \theta^2 L_2 / 2}
	= \econst^{-t^2 / (2L_2)}.
\]
The infimum is attained when $\theta = t / L_2$.
\end{proof}

The proof of~\cref{thm:positive-sum} depends on a bound for
$\mgf_X$, defined in~\eqref{eqn:scalar-mgf}.
Although this bound is classic and rather elementary~\cite[Exer.~2.9]{BLM13:Concentration-Inequalities},
we will develop an alternative treatment that has more potential for generalization.
The results in this section will reappear in the proof of the matrix
comparison inequalities.

\subsection{Completely monotone functions}

Our approach takes advantage of a special feature of the
decaying exponential that is encapsulated in the next definition.

\begin{definition}[Completely monotone function] \label{def:cm}
A function $f : \set{I} \to \R$ on the interval $\set{I} \subseteq \R$
is \term{completely monotone to order $K$} when its first $K$ derivatives
exist and alternate sign:
\begin{equation} \label{eqn:cm}
(-1)^k f^{(k)}(w) \geq 0
\quad\text{for all $w \in \set{I}$ and each $k = 0, 1, 2, \dots, K$.}
\end{equation}
If the interval $\set{I}$ includes an endpoint, the derivatives at
the endpoint are interpreted as one-sided derivatives.  The function 
$f$ is \term{completely monotone} when~\eqref{eqn:cm} holds for each $K \in \N$.
\end{definition}

The fundamental example of a completely monotone function is a \hilite{decaying} exponential.
For fixed \hilite{$\theta \geq 0$} and variable $w \in \R$, the function
\[
w \mapsto \econst^{-\theta w}
\quad\text{is completely monotone for $w \in \R$.}
\]
In fact, a function $f : \R_{+} \to \R$ on the nonnegative real line
is completely monotone \hilite{if and only if} it is the Laplace
transform of a finite, positive Borel measure $\mu$:
\begin{equation} \label{eqn:bernstein}
f(w) = \int_{[0,\infty)} \econst^{-\theta w} \, \mu(\diff{\theta}).
\end{equation}
The representation~\eqref{eqn:bernstein} is called Bernstein's theorem~\cite[Thm.~IV.12a]{Wid41:Laplace-Transform}.

Completely monotone functions support a beautiful theory, elaborated
in the classic book of Widder~\cite[Chap.~IV]{Wid41:Laplace-Transform}.
The monograph of Schilling et al.~\cite{SSV12:Bernstein-Functions-2ed}
is a more recent reference.  This background is not required for our
purposes.

\subsection{Tools from Stein's method}

The main steps in the analysis are adapted from the literature
on Charles Stein's method, a collection of tools for establishing
distributional approximations and concentration inequalities.
This section outlines some ideas from Stein's method.
See the survey of Ross~\cite{Ros11:Fundamentals-Steins}
or the book of Chen et al.~\cite{CGS11:Normal-Approximation}
for more information.

To describe the variability in a distribution,
we can employ exchangeable pairs of random variables.
A pair $(W, Y)$ of real random variables is \term{exchangeable} when
$(W, Y) \sim (Y, W)$.  Equivalently,
\begin{equation} \label{eqn:exch}
\Expect F(W, Y) = \Expect F(Y, W).
\end{equation}
for every bivariate function $F : \R \times \R \to \R$ where
the expectation exists.  A pair of iid random variables
is the most basic example of an exchangeable pair.

The next ingredient is an elegant covariance identity.
Let $g, h : \set{I} \to \R$ be functions on the interval
$\set{I} \subseteq \R$.  For an \hilite{iid pair} $(W,Y)$
of random variables taking values in $\set{I}$,
\begin{equation} \label{eqn:cov-exch}
\Cov(g(W), h(W)) %
	= \frac{1}{2} \Expect\big[ (g(W) - g(Y))( h(W) - h(Y) ) \big].
\end{equation}
This formula can be verified by direct calculation.
It is valid whenever the expectations are finite.

To bound differences of function values, as in~\eqref{eqn:cov-exch},
we employ a formula of Hermite.  For a continuously
differentiable function $h : \set{I} \to \R$,
\begin{equation} \label{eqn:hermite}
\frac{h(w) - h(y)}{w-y} = \int_0^1 h'(\tau w + (1 - \tau) y) \idiff{\tau}
\quad\text{for all $w, y \in \set{I}$.}
\end{equation}
Identity~\eqref{eqn:hermite} follows from the fundamental theorem of calculus.
Moreover, \hilite{if $h'$ is convex on $\set{I}$}, then
\begin{equation} \label{eqn:hermite-cvx}
\frac{h(w) - h(y)}{w - y}
	\leq \frac{h'(w) + h'(y)}{2}
	\quad\text{for all $w, y \in \set{I}$.}
\end{equation}
When $w = y$ in~\eqref{eqn:hermite} or~\eqref{eqn:hermite-cvx},
we interpret the left-hand side as $h'(w)$.

The most important element in our proof of \cref{thm:positive-sum} 
is an association inequality
for functions with opposite sense~\cite[Thm.~2.14]{BLM13:Concentration-Inequalities}.
Assume that $g : \set{I} \to \R$ is \hilite{increasing}, while $h : \set{I} \to \R$ is \hilite{decreasing}.
For any random variable $W$ taking values in $\set{I}$,
\begin{equation} \label{eqn:neg-assoc}
\Expect[ g(W) h(W) ] \leq \Expect[ g(W) ] \cdot \Expect[ h(W) ].
\end{equation}
Equivalently, the covariance of $g(W)$ and $h(W)$ is \hilite{negative}.
To establish the result~\eqref{eqn:neg-assoc}, note that
\[
(g(s) - g(t))(h(s) - h(t)) \leq 0
\quad\text{for all $s, t \in \set{I}$.}
\]
Combine this formula with the covariance identity~\eqref{eqn:cov-exch}.

\subsection{Covariance bounds for completely monotone functions}

This section shows how the tools from Stein's method lead to clean
bounds for covariances involving a completely monotone function.
In particular, this argument applies to the mgf of the lower tail.

\begin{lemma}[Covariance bound: Completely monotone function] \label{lem:cov-cm}
Let $f : \R_+ \to \R$ be a function on the nonnegative real line
that is \hilite{completely monotone to order four}.
For each \hilite{nonnegative} real random variable $W$, %
\[
\Cov( W, f'(W) ) = \Expect\big[ (W - \Expect W) f'(W) \big]
	\leq \Expect[ W^2 ] \cdot \Expect[ f''(W) ].
\]
The bound is valid when the expectations are finite.
\end{lemma}

\begin{proof}
According to \cref{def:cm}, the derivatives of a
completely monotone function $f$ alternate sign.
We only rely on the conditions that $f''$ is positive, decreasing, and convex,
so complete monotonicity to order four is sufficient.

Let $Y$ be an independent copy of $W$.  By the exchangeable
pairs formula~\eqref{eqn:cov-exch},
\begin{align*}
\Cov(W, f'(W)) &= \frac{1}{2} \Expect\big[ (W - Y)( f'(W) - f'(Y) ) \big] \\
	&\leq \frac{1}{4} \Expect \big[ (W - Y)^2 ( f''(W) + f''(Y) ) \big]
	= \frac{1}{2} \Expect \big[ (W - Y)^2 f''(W) \big].
\end{align*}
The inequality is the bound~\eqref{eqn:hermite-cvx} for the
divided difference of the function $f'$, whose derivative $f''$ is convex.
In the last step, we have used the exchangeability~\eqref{eqn:exch}
of $(W, Y)$ to simplify the expression.

To continue, let us separate the two factors in the expectation.
Note that $(w - y)^2 \leq w^2 + y^2$ for $w, y \geq 0$.
Since $f''$ is positive,
\[
\Cov(W, f'(W)) \leq \frac{1}{2} \Expect\big[ (W^2 + Y^2) f''(W) \big]
	\leq \frac{1}{2} \Expect[ W^2 f''(W) ] + \frac{1}{2} \Expect[ W^2 ] \cdot \Expect[ f''(W) ].
\]
The second step relies on the independence and identical distribution of $(W,Y)$.
On the nonnegative real line $\R_+$, the function $w \mapsto w^2$ is increasing,
while $f''$ is decreasing.  Thus, the association inequality~\eqref{eqn:neg-assoc}
furnishes a bound for the first expectation on the right-hand side:
\[
\Expect[ W^2 f''(W) ] \leq \Expect[ W^2 ] \cdot \Expect[ f''(W) ].
\]
Combine the last two displays to arrive at the advertised result.
\end{proof}

Lemma~\ref{lem:cov-cm} depends crucially on the assumption that $f'''$ is negative,
which allows us to invoke the association inequality~\eqref{eqn:neg-assoc}
to decouple $W^2$ from $f''(W)$.  The entropy method
employs association inequalities in a similar fashion;
for instance, see~\cite[Thm.~6.27]{BLM13:Concentration-Inequalities}.
The overall structure of the proof is similar with classic
arguments from Stein's method, exemplified in~\cite[Thm.~1.5(ii)]{Cha07:Steins-Method}.

\subsection{Positive sum: mgf bound}

With~\cref{lem:cov-cm} at hand, we quickly obtain a bound for
the mgf of a sum of independent, nonnegative random variables.

\begin{proposition}[Positive sum: mgf bound] \label{prop:scalar-mgf}
Instate the hypotheses of \cref{thm:positive-sum}.
The mgf~\eqref{eqn:scalar-mgf} satisfies the bound
\[
\log \mgf_X(\theta) \leq %
\theta^2 L_2/2
\quad\text{for all $\theta \geq 0$.}
\]
\end{proposition}

We can rewrite the outcome of \cref{prop:scalar-mgf} in a more suggestive fashion:
\begin{equation} \label{eqn:scalar-mgf-normal}
\Expect \econst^{-\theta X}
	\leq \Expect \econst^{-\theta Z}
	\quad\text{where $Z \sim \normal(\Expect X, L_2)$.}
\end{equation}
In other words, we have compared the lower tail of the
sum $X$ with the lower tail of an appropriate normal random variable $Z$
whose statistics derive from the summands in $X$.  In \cref{sec:psd-weights},
we will see that the inequality~\eqref{eqn:scalar-mgf-normal}
generalizes to matrices.

\begin{proof}
Recall the definition~\eqref{eqn:scalar-mgf} of $\mgf_X(\theta)$.
Since $\log \mgf_{X}(0) = 0$, the result holds for $\theta = 0$,
and we may assume that $\theta > 0$.
The derivative of the mgf takes the form
\begin{align*}
\mgf'_X(\theta) &= - \Expect\big[ (X - \Expect X) \econst^{-\theta(X - \Expect X)} \big] \\
	&= - \sum_{i=1}^n \Expect\big[ (W_i - \Expect W_i) \econst^{-\theta(X - \Expect X)} \big]
	\eqqcolon - \sum_{i=1}^n \Expect \big[ (W_i - \Expect W_i) \econst^{\theta B_i - \theta W_i} \big].
\end{align*}
The second relation expands the sum $X - \Expect X = \sum_{i=1}^n (W_i - \Expect W_i)$.
Note that the random variable $B_i \coloneqq W_i - (X - \Expect X)$ is statistically independent
from $W_i$.

Conditional on $B_i$, apply~\cref{lem:cov-cm} to the completely monotone function
$f_i(w) \coloneqq \econst^{\theta B_i - \theta w}$.  This step results in the bound
\begin{align*}
\mgf'_X(\theta) &\leq \theta \sum_{i=1}^n \Expect[ W_i^2 ] \cdot \Expect[ \econst^{\theta B_i - \theta W_i} ] \\
	&= \theta \sum_{i=1}^n \Expect[ W_i^2 ] \cdot \Expect[ \econst^{-\theta(X - \Expect X)} ]
	= \theta L_2 \cdot \mgf_X(\theta).
\end{align*}
Solve the differential inequality to complete the proof.
\end{proof}

\section{Gaussian comparison: Randomly weighted sums of psd matrices}
\label{sec:psd-weights}

In this section, we turn to the proof of \cref{sec:sampling-intro},
which provides a bound for the minimum eigenvalue of a randomly
weighted sum of fixed psd matrices.
For technical reasons, we develop the result using slightly different
notation and assumptions.

Fix a system $(\mtx{A}_1, \dots, \mtx{A}_n)$ of \hilite{psd} matrices,
with common dimension $d$. %
These matrices need not be distinct from each other.
Consider an \hilite{independent} family $(W_1, \dots, W_n)$ of \hilite{square-integrable, nonnegative}
real random variables: $W_i \geq 0$ and $\Expect W_i^2 < + \infty$.
Form the random psd matrix
\begin{equation} \label{eqn:psd-weight}
\mtx{X} \coloneqq \sum_{i=1}^n W_i \mtx{A}_i.
\end{equation}
We will compare the random psd matrix $\mtx{X}$ with an appropriate Gaussian model.
Define the deterministic self-adjoint matrices
\begin{equation} \label{eqn:psd-weight-system}
\mtx{H}_i \coloneqq (\Expect W_i^2 )^{1/2} \mtx{A}_i
\quad\text{for $i = 1, \dots, n$.}
\end{equation}
Construct the \hilite{centered} Gaussian matrix
\begin{equation} \label{eqn:psd-weight-gauss}
\mtx{Z} \coloneqq \sum_{i=1}^n \gamma_i \mtx{H}_i
\quad\text{where \quad $\gamma_i \sim \normal_{\R}(0,1)$ iid.}
\end{equation}
Equivalently, $\mtx{Z} \sim \normal(\mtx{0}, \mathsf{V})$ with variance function
\begin{equation} \label{eqn:psd-weight-covar}
\mathsf{V}(\mtx{M}) \coloneqq \sum_{i=1}^n \abssq{\ip{\mtx{M}}{\mtx{H}_i}}
	= \sum_{i=1}^n (\Expect W_i^2 ) \cdot \abssq{\ip{\mtx{M}}{\mtx{A}_i}}
	\quad\text{for $\mtx{M} \in \Sym_d$.}
\end{equation}
With these definitions, we can state a comparison inequality.

\begin{theorem}[Gaussian comparison: Weighted psd sum] \label{thm:psd-weights}
Fix an arbitrary self-adjoint matrix $\mtx{\Delta} \in \Sym_d$.
Introduce the random $d$-dimensional matrices $\mtx{X}$ and $\mtx{Z}$
from~\eqref{eqn:psd-weight} and \eqref{eqn:psd-weight-gauss}.
Then
\begin{equation} \label{eqn:psd-weight-expect}
\Expect \lambda_{\min}(\mtx{X} - \Expect \mtx{X} + \mtx{\Delta})
	\geq \Expect \lambda_{\min}(\mtx{Z} + \mtx{\Delta}) - \sqrt{2 \smash{\sigma_*^2}(\mtx{Z}) \log d}.
\end{equation}
In addition, for all $t \geq 0$,
\begin{equation} \label{eqn:psd-weight-prob}
\Prob{ \lambda_{\min}(\mtx{X} - \Expect \mtx{X} + \mtx{\Delta}) \leq \Expect \lambda_{\min}(\mtx{Z} + \mtx{\Delta}) - t }
	\leq d \cdot \econst^{-t^2/(2 \sigma_*^2(\mtx{Z}))}.
\end{equation}
The weak variance $\sigma_*^2(\mtx{Z})$ is defined in~\eqref{eqn:weak-var}.
\end{theorem}

\Cref{thm:psd-weights} generalizes \cref{thm:sampling-intro} from the introduction
by allowing a shift $\mtx{\Delta}$ of the expectation.
This improvement comes for free and allows for a more transparent
and natural argument.

\begin{proof}[Proof of \cref{thm:psd-weights}]

\Cref{prop:psd-weight-mgf}, below, provides a comparison for the trace mgfs:
\[
\mgf_{\mtx{X}}(\theta)
	\coloneqq \Expect \trace \econst^{-\theta (\mtx{X} - \Expect \mtx{X} + \mtx{\Delta})}
	\leq \Expect \trace \econst^{- \theta (\mtx{Z} + \mtx{\Delta})}
\quad\text{for $\theta \geq 0$.}
\]
Bound the trace exponential on the right-hand side in terms of the minimum eigenvalue:
\[
\mgf_{\mtx{X}}(\theta)	\leq d \cdot \Expect \lambda_{\max}\big( \econst^{- \theta (\mtx{Z} + \mtx{\Delta})} \big)
	= d \cdot \Expect \econst^{-\theta \lambda_{\min}(\mtx{Z} + \mtx{\Delta})}.
\]
Indeed, the trace of a psd matrix %
does not exceed the dimension times the maximum eigenvalue.
The second relation is the spectral mapping theorem, combined with the fact that
the decreasing exponential function reverses the order of the eigenvalues.

To continue, add and subtract $\Expect \lambda_{\min}(\mtx{Z} + \mtx{\Delta})$ in
the exponent:
\begin{align*}
\mgf_{\mtx{X}}(\theta) &\leq d \cdot \econst^{-\theta \Expect \lambda_{\min}(\mtx{Z} + \mtx{\Delta})}
	\cdot \Expect \econst^{-\theta (\lambda_{\min}(\mtx{Z} + \mtx{\Delta}) - \Expect \lambda_{\min}(\mtx{Z} + \mtx{\Delta}))}.
\end{align*}
The Gaussian concentration inequality~\eqref{eqn:gauss-lip-mgf}
controls the fluctuations of the minimum eigenvalue $\lambda_{\min}(\mtx{Z} + \mtx{\Delta})$
around its mean:
\begin{equation} \label{eqn:psd-weight-mgf-pf}
\mgf_{\mtx{X}}(\theta) \leq d \cdot \econst^{-\theta \Expect \lambda_{\min}(\mtx{Z} + \mtx{\Delta})}
	\cdot \econst^{\theta^2 \sigma_*^2(\mtx{Z}) / 2}.
\end{equation}
The inequality~\eqref{eqn:psd-weight-mgf-pf} quickly leads to both the stated results.

To obtain the probability bound~\eqref{eqn:psd-weight-prob},
we employ the matrix Laplace transform method~\cite[Prop.~3.2.1]{Tro15:Introduction-Matrix}.
For fixed $t \geq 0$ and arbitrary $\theta > 0$,
\begin{align*}
\Prob{ \lambda_{\min}(\mtx{X} - \Expect \mtx{X} + \mtx{\Delta}) \leq \Expect \lambda_{\min}(\mtx{Z} + \mtx{\Delta}) - t }
	&\leq \econst^{\theta (\Expect \lambda_{\min}(\mtx{Z} + \mtx{\Delta}) - t)}
		\cdot \mgf_{\mtx{X}}(\theta) \\
	&\leq d \cdot \econst^{-\theta t} \cdot \econst^{\theta^2 \sigma_*^2(\mtx{Z}) / 2}.
\end{align*}
The second inequality is~\eqref{eqn:psd-weight-mgf-pf}.
Select $\theta = t / \sigma_*^2(\mtx{Z})$ to reach the probability inequality.

The result~\eqref{eqn:psd-weight-expect} for the expectation follows from a similar argument.
For arbitrary $\theta > 0$, the matrix Laplace transform
method~\cite[Prop.~3.2.2]{Tro15:Introduction-Matrix}
yields
\begin{align*}
\Expect \lambda_{\min}(\mtx{X} - \Expect \mtx{X} - \mtx{\Delta})
	&\geq -\frac{1}{\theta} \log \mgf_{\mtx{X}}(\theta) \\
	&\geq - \frac{1}{\theta} \left[ \log d - \theta \Expect \lambda_{\min}(\mtx{Z} + \mtx{\Delta}) + \theta^2 \sigma_*^2(\mtx{Z}) / 2 \right].
\end{align*}
The second inequality is~\eqref{eqn:psd-weight-mgf-pf}.
Choose $\theta = \sqrt{ (2 \log d) / \sigma_*^2(\mtx{Z}) }$ to reach the expectation bound.
\end{proof}

\begin{proof}[Proof of \cref{thm:sampling-intro} from \cref{thm:psd-weights}]
Choose the shift $\mtx{\Delta} = \Expect \mtx{X}$,
and define the Gaussian matrix $\mtx{Z}' \coloneqq \mtx{Z} + \Expect \mtx{X} \sim \normal(\Expect \mtx{X}, \mathsf{V})$,
where $\mathsf{V}$ is defined in~\eqref{eqn:psd-weight-covar}.
Change variables to state \cref{thm:sampling-intro} directly in terms of
the distribution of $\mtx{Z}'$.
\end{proof}

\subsection{Stahl's theorem: Complete monotonicity of the trace exponential}
\label{sec:stahl}

To prove \cref{thm:psd-weights}, we extend the considerations behind
the scalar comparison (\cref{thm:positive-sum}) to the matrix setting.
This strategy depends on a profound fact from matrix analysis,
called Stahl's theorem~\cite{Sta13:Proof-BMV}.

\begin{fact}[Stahl's theorem; formerly the BMV conjecture] \label{fact:bmv}
Fix \hilite{self-adjoint} matrices $\mtx{A}, \mtx{B} \in \Sym_d$, and assume
that \hilite{$\mtx{A}$ is psd}.  Then the trace exponential function
\[
f(w) \coloneqq \trace \exp( \mtx{B} - w \mtx{A} )
\quad\text{is completely monotone for $w \geq 0$.}
\]
In particular, \cref{lem:cov-cm} applies to the function $f$.
\end{fact}

While Stahl's theorem is recent, it has a remarkable history.
In 1975, Bessis--Moussa--Villani~\cite{BMV75:Monotonic-Converging}
conjectured \cref{fact:bmv} as part of their method for bounding
partition functions of quantum mechanical systems.  They established the
result in two special cases: (1) when $\mtx{A}, \mtx{B}$ commute;
or (2) when $\mtx{A}, \mtx{B}$ are $2 \times 2$ matrices.
Over the next decades, the BMV conjecture received significant
attention in mathematical physics, but it was not resolved.

In 2004, Lieb \& Seiringer~\cite{LS04:Equivalent-Forms}
proved that \cref{fact:bmv} admits an equivalent formulation:
For all \hilite{psd} $\mtx{A}, \mtx{B} \in \Sym_d$,
the coefficients of the polynomial
$w \mapsto \trace{} (w \mtx{A} + \mtx{B})^p$ are
nonnegative for all $p \in \N$.
This link with real algebraic geometry ignited a new stage
of research, based on sum-of-squares hierarchies and
semidefinite programming, that generated new evidence
supporting the BMV conjecture.

Finally, in 2012, Herbert Stahl~\cite{Sta13:Proof-BMV}
established \cref{fact:bmv} using classic methods from
complex analysis.  Bernstein's theorem~\eqref{eqn:bernstein}
states that a completely monotone function is the Laplace
transform of a finite, positive Borel measure.  Roughly speaking, Stahl
inverted the Laplace transform to obtain the representing
measure.  To prove that the representing measure is positive,
he exploited the theory of Riemann surfaces.
See~\cite{Ere15:Herbert-Stahls} for another account of
Stahl's proof.

Very recently, Otte Hein{\"a}vaara constructed a rather different
argument~\cite[Thm.~2]{Hei24:Tracial-Joint} that leads
to a remarkable generalization of \cref{fact:bmv}:

\begin{fact}[Hein{\"a}vaara's theorem] \label{fact:heina}
Fix \hilite{self-adjoint} matrices $\mtx{A}, \mtx{B} \in \Sym_d$,
and assume that \hilite{$\mtx{A}$ is psd}.  Consider a function
$h : \R \to \R$ that is completely monotone to order $K$.
Then the trace function
\[
f(w) \coloneqq \trace h(w \mtx{A} - \mtx{B})
\quad\text{is completely monotone to order $K$ for $w \in \R$.}
\]
Stahl's theorem (\cref{fact:bmv}) follows from the choice $h(w) = \econst^{-w}$.
\end{fact}

\noindent
Hein{\"a}vaara's proof of \cref{fact:heina}
appeals to a novel object, called a
\term{tracial joint spectral measure}, that captures the behavior
of trace functions of the form $(w,y) \mapsto \trace h(w \mtx{A} + y \mtx{B})$
for self-adjoint $\mtx{A}, \mtx{B}$.  His techniques
yield many deep new statements about trace functions.

\subsection{Additional tools}

The argument involves some standard tools from probability
and matrix analysis.  First, we record the Gaussian
integration by parts (IBP) rule~\cite[Lem.~1.1.1]{NP12:Normal-Approximations}.

\begin{fact}[Gaussian IBP] \label{fact:gauss-ibp}
Consider iid real standard normal variables $(\gamma_1, \dots, \gamma_n)$.
For each differentiable function $h : \R^n \to \R$,
\[
\Expect\big[ \gamma_i \cdot h(\gamma_1, \dots, \gamma_n) \big]
	= \Expect\big[ (\partial_i h)(\gamma_1, \dots, \gamma_n) \big].
\]
In this formula, $\partial_i h$ denotes the partial derivative of $h$
with respect to its $i$th argument.  The identity is valid whenever
the right-hand side is finite.
\end{fact}

Second, %
we recall a classic formula from matrix calculus~\cite[Example~X.4.2(v)]{Bha97:Matrix-Analysis}.
For self-adjoint matrices $\mtx{A}, \mtx{H} \in \Sym_d$, the derivative of the exponential
at $\mtx{A}$ in the direction $\mtx{H}$ satisfies
\begin{equation} \label{eqn:d-exp}
(\Diff\, \econst^{\mtx{A}}) [ \mtx{H} ] \coloneqq
	\lim\nolimits_{s \to 0}\ s^{-1} \big[ \econst^{\mtx{A} + s \mtx{H}} - \econst^{\mtx{A}} \big]
	= \int_0^1 \econst^{\tau \mtx{A}} \mtx{H} \econst^{(1-\tau) \mtx{A}} \idiff{\tau}.
\end{equation}
As a consequence, the derivative of the trace exponential simplifies to
\begin{equation} \label{eqn:d-trexp}
\big(\Diff \trace \econst^{\mtx{A}}\big)[\mtx{H}] = \trace\big[ \mtx{H} \econst^{\mtx{A}} \big].
\end{equation}
The statement~\eqref{eqn:d-trexp} follows from~\eqref{eqn:d-exp} when we take the trace
and cycle to combine the exponentials.

\subsection{Comparison for the trace mgf}

With Stahl's theorem (\cref{fact:bmv}) at hand, we can establish a bound
for the trace mgf associated with the minimum eigenvalue of a random psd matrix.

\begin{proposition}[Weighted psd sum: Comparison of trace mgfs] \label{prop:psd-weight-mgf}
Instate the hypotheses of \cref{thm:psd-weights}.
For $\theta \geq 0$,
\[
\mgf_{\mtx{X}}(\theta)
	\coloneqq \Expect \trace \econst^{-\theta(\mtx{X} - \Expect \mtx{X} + \mtx{\Delta})}
	\leq \Expect \trace \econst^{-\theta(\mtx{Z} + \mtx{\Delta})}
	\eqqcolon \mgf_{\mtx{Z}}(\theta).
\]
\end{proposition}

In the scalar setting (\cref{prop:scalar-mgf}),
the mgf of a Gaussian random variable emerges from a direct argument.
In the matrix setting, it is more expedient to \hilite{compare} the
trace mgf of the psd model with the trace mgf of the
Gaussian model.  This argument relies on interpolation,
much like classic Gaussian comparison inequalities~\cite{Kah86:Inegalite-Type}
or more recent work on universality for random matrices~\cite{BvH24:Universality-Sharp}.

\begin{proof}
To implement the comparison, we interpolate between the two random matrix models.
Define a path in the space of random matrices:
\begin{equation} \label{eqn:mtx-path}
\mtx{Y}_s \coloneqq \sqrt{s}\, (\mtx{X} - \Expect \mtx{X}) + \sqrt{1-s} \, \mtx{Z} + \mtx{\Delta}
\quad\text{for $s \in [0,1]$.}
\end{equation}
Introduce the trace mgf of the interpolants:
\[
u(s) \coloneqq \Expect \trace \econst^{-\theta \mtx{Y}_s}
\quad\text{for $s \in [0,1]$.}
\]
Note that $u(0) = \mgf_{\mtx{Z}}(\theta)$ and $u(1) = \mgf_{\mtx{X}}(\theta)$.
We \hilite{claim} that the derivative $u'(s) \leq 0$ for $s \in (0,1)$.
Once this point is settled, \cref{prop:scalar-mgf} follows inexorably.

By a scaling argument, we can assume that \hilite{$\theta = 1$}.
Now, for fixed $s \in (0,1)$, the derivative $u'(s)$ along the
interpolation path takes the form
\begin{equation} \label{eqn:u'}
u'(s) = \frac{-1}{2\sqrt{s}} \Expect \trace \big[ (\mtx{X} - \Expect \mtx{X}) \econst^{-\mtx{Y}_s} \big]
	- \frac{-1}{2\sqrt{1 - s}} \Expect \trace \big[ \mtx{Z} \econst^{-\mtx{Y}_s} \big]
	\eqqcolon \onecirc - \twocirc.
\end{equation}
We start with the second term $\twocirc$, involving the Gaussian random matrix, as it offers
a template for how to bound the first term $\onecirc$.

To handle the term $\twocirc$ from~\eqref{eqn:u'}, we invoke Gaussian integration by parts. %
In preparation, expand the random matrix $\mtx{Z} = \sum_{i=1}^n \gamma_i \mtx{H}_i$
as a sum: %
\begin{align*}
\twocirc &= \frac{-1}{2\sqrt{1-s}} \Expect \trace\big[ \mtx{Z} \econst^{-\mtx{Y}_s} \big]
	= \frac{-1}{2\sqrt{1-s}} \sum_{i=1}^n \Expect \big[ \gamma_i \trace\big[ \mtx{H}_i \econst^{-\mtx{Y}_s} \big] \big] \\
	&= \frac{1}{2} \sum_{i=1}^n \int_0^1 \Expect \trace\big[ \mtx{H}_i \econst^{-\tau \mtx{Y}_s} \mtx{H}_i \econst^{-(1-\tau) \mtx{Y}_s} \big] \idiff{\tau} \\
	&= \frac{1}{2} \sum_{i=1}^n \Expect[ W_i^2 ] \cdot \int_0^1 \Expect \trace\big[ \mtx{A}_i \econst^{-\tau \mtx{Y}_s} \mtx{A}_i \econst^{-(1-\tau) \mtx{Y}_s} \big] \idiff{\tau}. 
\end{align*}
The second line is the Gaussian IBP rule (\cref{fact:gauss-ibp}).  This calculation
relies on the formula~\eqref{eqn:d-exp} for the derivative of the matrix
exponential.  The partial derivative $\partial_{\gamma_i} \mtx{Y}_s = \sqrt{1 - s} \, \mtx{H}_i$,
owing to the expressions~\eqref{eqn:mtx-path} for $\mtx{Y}_s$ and~\eqref{eqn:psd-weight-gauss}
for $\mtx{Z}$.  For the last step, recall the definition~\eqref{eqn:psd-weight-system}
of the matrix coefficients $\mtx{H}_i$.

To obtain a bound for the term $\onecirc$ in~\eqref{eqn:u'}, we plan to exploit \cref{lem:cov-cm}.
Expand the centered random matrix $\mtx{X} - \Expect \mtx{X} = \sum_{i=1}^n (W_i - \Expect W_i) \mtx{A}_i$
as a sum: %
\begin{align*}
\onecirc &= \frac{-1}{2\sqrt{s}} \Expect \trace \big[ (\mtx{X} - \Expect \mtx{X}) \econst^{-\mtx{Y}_s} \big]
	= \frac{-1}{2\sqrt{s}} \sum_{i=1}^n \Expect \trace \big[ (W_i - \Expect W_i) \mtx{A}_i \econst^{-\mtx{Y}_s} \big] \\
	&\eqqcolon \frac{1}{2s} \sum_{i=1}^n \Expect \trace \big[ - \sqrt{s} (W_i - \Expect W_i) \mtx{A}_i \econst^{\mtx{B}_i - \sqrt{s} W_i \mtx{A}_i} \big].
\end{align*}
For each index $i$, we have defined the random matrix $\mtx{B}_i \coloneqq \sqrt{s} W_i \mtx{A}_i - \mtx{Y}_s$.
Observe that $\mtx{B}_i$ is statistically independent from the random variable $W_i$.

Conditional on $\mtx{B}_i$, introduce the (deterministic) function
\[
f_i(w) \coloneqq \trace \econst^{\mtx{B}_i - \sqrt{s} w \mtx{A}_i}
\quad\text{for $w \geq 0$ and $i = 1, \dots, n$.}
\]
Stahl's theorem (\cref{fact:bmv}) asserts that \hilite{each $f_i$ is completely monotone}
for fixed self-adjoint $\mtx{A}_i, \mtx{B}_i$ with $\mtx{A}_i$ psd.
In light of~\eqref{eqn:d-exp} and~\eqref{eqn:d-trexp},
the first two derivatives of $f_i$ satisfy
\begin{align*}
f_i'(w) &= - \sqrt{s}  \cdot \trace \big[ \mtx{A}_i \econst^{\mtx{B}_i - \sqrt{s} w \mtx{A}_i} \big]; \\
f_i''(w) &=  s\cdot \int_0^1 \trace \big[ \mtx{A}_i \econst^{\tau(\mtx{B}_i - \sqrt{s} w \mtx{A}_i)} \mtx{A}_i \econst^{(1-\tau)(\mtx{B}_i - \sqrt{s} w \mtx{A}_i)} \big] \idiff{\tau}.
\end{align*}
We may now express $\onecirc$ in terms of the functions $f_i$.  This revision
allows us to invoke \cref{lem:cov-cm}, conditional on $\mtx{B}_i$.  We arrive
at the inequality
\begin{align*}
\onecirc &= \frac{1}{2s} \sum_{i=1}^n \Expect \big[ (W_i - \Expect W_i) f_i'(W_i) \big]
	\leq \frac{1}{2s} \sum_{i=1}^n \Expect[ W_i^2 ] \cdot \Expect[ f_i''(W_i) ] \\
	&= \frac{1}{2} \sum_{i=1}^n \Expect[ W_i^2 ] \cdot \int_0^1 \Expect \trace \big[
	\mtx{A}_i \econst^{-\tau \mtx{Y}_s} \mtx{A}_i \econst^{-(1-\tau)\mtx{Y}_s} \big] \idiff{\tau} = \twocirc.
\end{align*}
The last line depends on the relations $\mtx{B}_i - \sqrt{s} W_i \mtx{A}_i = - \mtx{Y}_s$.
As promised, $u'(s) = \onecirc - \twocirc \leq 0$.
\end{proof}

\subsection{Extension: Polynomial moments}

The proof of~\cref{prop:psd-weight-mgf} can be adapted to obtain
a comparison theorem for one-sided polynomial moments.

\begin{proposition}[Randomly weighted sum: Polynomial moment bound] \label{prop:psd-weights-poly}
Instate the hypotheses of \cref{thm:psd-weights}.  For each $p \geq 4$,
\[
\Expect \trace{} (\mtx{X} - \Expect \mtx{X} + \mtx{\Delta})_-^p
	\leq \Expect \trace{} (\mtx{Z} + \mtx{\Delta})_-^p.
\]
As usual, the negative part $(a)_{-} \coloneqq \max\{-a, 0\}$ for $a \in \R$
binds before the power.
\end{proposition}

\begin{proof}[Proof sketch]
The structure of the argument is identical with the proof of~\cref{prop:psd-weight-mgf}.
For justification, when $p \geq 4$, note that $h : w \mapsto (w)_-^p$
is completely monotone to order four on the real line.
Heinev{\"a}ara's theorem (\cref{fact:heina}) ensures
that the associated trace function $f(w) \coloneqq \trace{} (w\mtx{A} - \mtx{B})_{-}^p$
is also completely monotone to order four.  Therefore,
we can activate \cref{lem:cov-cm}.

While it is not really necessary to compute the derivatives of the trace function $f$ explicitly,
one may employ the Dalecki{\u \i}--Kre{\u \i}n formula~\cite[Thm.~V.3.3]{Bha97:Matrix-Analysis}
to obtain the detailed expressions.
\end{proof}

\section{Gaussian comparison: Sum of iid random psd matrices}
\label{sec:iid}

This section develops a comparison theorem for the minimum
eigenvalue of a sum of \hilite{iid} random psd matrices.
This formulation includes covariance matrices and related models.
Surprisingly, this result is a consequence of the
comparison (\cref{thm:psd-weights}) for sums of
randomly weighted psd matrices.

Let $\mtx{W}$ be a random \hilite{psd} matrix with dimension $d$. %
Assume that $\mtx{W}$ is square-integrable: $\Expect \norm{\mtx{W}}^2 < + \infty$.
For a fixed natural number $k \in \N$, consider the random psd matrix $\mtx{Y}$
obtained by adding $k$ independent copies of $\mtx{W}$.  That is,
\begin{equation} \label{eqn:iid-sum}
\mtx{Y} \coloneqq \sum_{j=1}^k \mtx{W}_j
\quad\text{where $\mtx{W}_j \sim \mtx{W}$ iid.}
\end{equation}
We compare the random psd matrix $\mtx{Y}$ with an appropriate
Gaussian model.  Consider a \hilite{centered}
Gaussian matrix that follows the distribution
\begin{equation} \label{eqn:iid-gauss}
\mtx{Z} \sim \normal(\mtx{0}, \mathsf{V})
\quad\text{where}\quad
\mathsf{V} \coloneqq k \cdot \Mo[\mtx{W}].
\end{equation}
With these definitions, we can state
another comparison theorem.

\begin{theorem}[Gaussian comparison: Sum of iid psd matrices] \label{thm:iid-sum}
Fix an arbitrary self-adjoint matrix $\mtx{\Delta} \in \Sym_d$.
Introduce the random $d$-dimensional matrices $\mtx{Y}$ and $\mtx{Z}$
from~\eqref{eqn:iid-sum} and~\eqref{eqn:iid-gauss}.
Then
\[
\Expect \lambda_{\min}(\mtx{Y} - \Expect \mtx{Y} + \mtx{\Delta})
	\geq \Expect \lambda_{\min}(\mtx{Z} + \mtx{\Delta})
	- \sqrt{2 \smash{\sigma_*^2(\mtx{Z})} \log(2d)}.
\]
In addition, for all $t \geq 0$,
\[
\Prob{ \lambda_{\min}(\mtx{Y} - \Expect \mtx{Y} + \mtx{\Delta})
	\leq \lambda_{\min}(\mtx{Z} + \mtx{\Delta}) - t}
	\leq 2d \cdot \econst^{-t^2 / (2\sigma_*^2(\mtx{Z}))}.
\]
The weak variance $\sigma_*^2(\mtx{Z})$ is defined in~\eqref{eqn:weak-var}.
\end{theorem}

\begin{proof}[Proof of \cref{thm:iid-sum}]
The content of the argument is a comparison inequality for the trace mgfs:
\begin{equation} \label{eqn:iid-mgf-bound-main}
\mgf_{\mtx{Y}}(\theta) \coloneqq \Expect \trace \econst^{- \theta (\mtx{Y} - \Expect \mtx{Y} + \mtx{\Delta})}
	\leq 2 \Expect \trace \econst^{-\theta (\mtx{Z} + \mtx{\Delta})}
	\quad\text{for $\theta \geq 0$.}
\end{equation}
See~\eqref{eqn:iid-mgf-pf} below.
The rest of the argument follows the same route
as the proof of~\cref{thm:psd-weights}.
\end{proof}

\begin{proof}[Proof of \cref{thm:iid-intro} from \cref{thm:iid-sum}]
To derive the result in the introduction, we choose the shift $\mtx{\Delta} = \Expect \mtx{Y}$,
and we employ monotonicity properties of the Gaussian distribution
to relax the assumption on the variance function.

Here are the details.  Construct the Gaussian distribution
$\mtx{Z}' \sim \normal(\mtx{\Delta}, \mathsf{V}')$
where $\mtx{\Delta} = \Expect \mtx{Y} $ and the
variance function $\mathsf{V}' \geq \mathsf{V} = k \cdot \Mo[\mtx{W}]$.
Invoke \cref{thm:iid-sum} with the centered Gaussian matrix
$\mtx{Z} \sim \normal(\mtx{0}, \mathsf{V})$ to obtain the
expectation bound
\[
\Expect \lambda_{\min}(\mtx{Y})
	\geq \Expect \lambda_{\min}(\mtx{Z} + \mtx{\Delta})	- \sqrt{2\smash{\sigma_*^2(\mtx{Z})} \log(2d)}. 
\]
Gaussian monotonicity (\cref{prop:gauss-monotone}) for the concave function
$\mtx{A} \mapsto \lambda_{\min}(\mtx{A} + \mtx{\Delta})$
and monotonicity of the weak variance~\eqref{eqn:wvar-monotone}
imply that
\[
\Expect \lambda_{\min}(\mtx{Z} + \mtx{\Delta})
	\geq \Expect \lambda_{\min}(\mtx{Z}')
	\quad\text{and}\quad
	\sigma_*^2(\mtx{Z}) \leq \sigma_*^2(\mtx{Z}').
\]
Combine the last two displays and adjust the notation to reach
the expectation bound~\eqref{eqn:intro-thm-iid-expect} in \cref{thm:iid-intro}.
The proof of the tail bound~\eqref{eqn:intro-thm-iid-tail} is similar.
\end{proof}

\subsection{Proof of the trace mgf bound: Overview}

The trace mgf bound~\eqref{eqn:iid-mgf-bound-main}
for an iid sum is a (nontrival) corollary of the trace mgf
bound for a randomly weighted sum (\cref{prop:psd-weight-mgf}).
Our strategy depends on an {empirical approximation}
of the distribution of the iid sum $\mtx{Y}$.  We draw and fix a
large sample from the distribution of the summand $\mtx{W}$,
and we form a proxy $\smash{\widehat{\mtx{Y}}}$ for the random matrix $\mtx{Y}$
by summing $k$ of the sample points. %
The empirical approximation $\smash{\widehat{\mtx{Y}}}$ is a randomly weighted sum
of fixed psd matrices with \hilite{dependent} coefficients.
Via the classic Poissonization technique, we can decouple the coefficients
to arrive at a randomly weighted sum with \hilite{independent} coefficients.
\Cref{prop:psd-weight-mgf} allows us to compare the trace mgf of the
independent model with a Gaussian distribution.
Related arguments have appeared in the probability literature;
for example, see~\cite[Sec.~11.11]{Tal21:Upper-Lower-2ed}.

\subsection{Step 1: Empirical approximation}

To lighten notation, assume that the shift $\mtx{\Delta} \in \Sym_d$
equals the zero matrix.  The proof for a general shift is no different.

For a large parameter $n \in \N$, draw a sample
$\coll{A}_n \coloneqq (\mtx{A}_1, \dots, \mtx{A}_n)$
where the matrices $\mtx{A}_i$ are iid copies of $\mtx{W}$.
The sampled matrices may not be distinct.  Until
the last steps of the proof, \hilite{we regard the
sample $\coll{A}_n$ as fixed}.

Consider a random matrix $\smash{\widehat{\mtx{W}}}$ that takes
a uniformly random value from the fixed sample $\coll{A}_n$:
\[
\widehat{\mtx{W}} = \mtx{A}_I
\quad\text{where}\quad
I \sim \uniform\{1, \dots, n\}.
\]
We construct a proxy $\smash{\widehat{\mtx{Y}}_n}$ for the iid sum $\mtx{Y}$
by forming an iid sum of copies of $\smash{\widehat{\mtx{W}}}$.
\[
\widehat{\mtx{Y}}_n \coloneqq \sum_{j=1}^k \widehat{\mtx{W}}_j
\quad\text{where $\widehat{\mtx{W}}_j \sim \widehat{\mtx{W}}$.}
\]
As a heuristic, when the number $n$ of sample points is large,
the distribution of the proxy $\smash{\widehat{\mtx{Y}}_n}$
is close to the distribution of the original sum $\mtx{Y}$.
\cref{lem:empirical-weak} justifies this claim.

\subsection{Step 2: Multinomial model}

To analyze the proxy $\smash{\widehat{\mtx{Y}}_n}$, we work with an alternative representation.
The $k$ independent summands in the proxy take the form
\[
\widehat{\mtx{W}}_j = \mtx{A}_{I_j}
\quad\text{where $I_j \sim \uniform\{1, \dots, n\}$ iid for $j = 1, \dots, k$.}
\]
For each index $i = 1, \dots, n$, define a scalar random variable $\delta_i$
that counts the number of the $I_j$ that select the index $i$.
That is,
\[
\delta_i \coloneqq \#\big\{ j \in \{1, \dots, k\} : I_j = i \big\}.
\]
As a consequence, $\vct{\delta} \coloneqq (\delta_1, \dots, \delta_n) \sim \multinomial(k, n)$
follows the multinomial distribution of $k$ balls placed independently and uniformly
at random in $n$ bins.  Recall that
\[
\delta_i \sim \binomial(1/n, k)
\quad\text{and}\quad
\sum_{i=1}^n \delta_i = k.
\]
Because of the coupling between the distributions,
\begin{equation} \label{eqn:proxy-multi}
\widehat{\mtx{Y}}_n = \sum_{j=1}^k \widehat{\mtx{W}}_j
	= \sum_{i=1}^n \delta_i \mtx{A}_i.
\end{equation}
Let us emphasize that the coefficient vector $\vct{\delta}$
is independent from the sample $\coll{A}_n$.

\subsection{Step 3: Poissonization}

The coefficients $\delta_i$ in the representation~\eqref{eqn:proxy-multi}
are not independent, but we can pass to a model that has independent coefficients:
\begin{equation} \label{eqn:proxy-poisson}
\mtx{X}_n \coloneqq \sum_{i=1}^n Q_i \mtx{A}_i
\quad\text{where $Q_i \sim \poisson(k/n)$ iid.}
\end{equation}
The Poisson variables $Q_i$ are independent from each other and from the
multinomial variables $\delta_i$.  Since $\Expect Q_i = \Expect \delta_i$,
the conditional expectations of the random matrices $\mtx{X}_n$
and $\smash{\widehat{\mtx{Y}}_n}$ coincide.
\begin{equation} \label{eqn:proxy-expect}
\Expect \big[ \mtx{X}_n \condbar \coll{A}_n \big]
	= \Expect \big[ \widehat{\mtx{Y}}_n \condbar \coll{A}_n \big]
	= \Expect \big[ k\widehat{ \mtx{W} } \condbar \coll{A}_n \big].
\end{equation}
Standard arguments quickly lead to a comparison between the two random models
$\smash{\widehat{\mtx{Y}}_n}$ and $\mtx{X}_n$.

\begin{lemma}[Poissonization] \label{lem:poisson}
Consider the random matrices $\smash{\widehat{\mtx{Y}}_n}$ and $\mtx{X}_n$
defined in~\eqref{eqn:proxy-multi} and~\eqref{eqn:proxy-poisson}.
For all $\theta \geq 0$,
\[
\Expect \big[ \trace \econst^{-\theta (\widehat{\mtx{Y}}_n - \Expect[ \widehat{\mtx{Y}}_n \condbar \coll{A}_n])} \lcondbar \coll{A}_n \big]
	\leq 2 \Expect \big[ \trace \econst^{-\theta ({\mtx{X}}_n - \Expect[ {\mtx{X}}_n \condbar \coll{A}_n ])} \lcondbar \coll{A}_n \big].
\]
\end{lemma}

\begin{proof}
With the sample $\coll{A}_n$ fixed,
consider the deterministic, positive function
\[
f(s_1, \dots, s_n) \coloneqq \trace \exp \left( - \theta \left( \sum_{i=1}^n s_i \mtx{A}_i
	- \Expect \big[ k \widehat{\mtx{W}} \condbar \coll{A}_n \big] \right) \right)
	\quad\text{for $(s_1, \dots, s_n) \in \R_+^n$.}
\]
According to~\eqref{eqn:proxy-expect}, the conditional expectation in the definition of $f$
coincides with the conditional expectation of both $\smash{\widehat{\mtx{Y}}_n}$ and of $\mtx{X}_n$.
Thus, we obtain the trace exponential with $\smash{\widehat{\mtx{Y}}_n}$ by evaluating
$f(\delta_1, \dots, \delta_n)$, and we obtain the trace exponential
with $\mtx{X}_n$ by evaluating $f(Q_1, \dots, Q_n)$.

The key insight is that the function $f$ is monotone decreasing
with respect to each argument $s_i$ because each matrix $\mtx{A}_i$ is psd.
This is a well-established and easy fact~\cite[Thm.~2.10]{Car10:Trace-Inequalities},
and it is also contained in Stahl's theorem (\cref{fact:bmv}).

For completeness, we include the familiar argument~\cite[Sec.~5.4]{MU17:Probability-Computing-2ed}
that leads to the conclusion of the lemma.  Recall that $(\delta_1, \dots, \delta_n) \sim \multinomial(k, n)$
and that $Q_i \sim \poisson(k/n)$ iid for $i = 1, \dots, n$.
Define the sum $Q \coloneqq \sum_{i=1}^n Q_i$
of the Poisson variables, and note that $Q \sim \poisson(k)$.
{Conditional on the sample $\coll{A}_n$}, calculate that
\begin{align*}
\Expect[ f( Q_1, \dots, Q_n ) \condbar \coll{A}_n ]
	&\geq \sum_{\ell = 0}^k \Expect[ f(Q_1, \dots, Q_n) \condbar \coll{A}_n, Q = \ell ] \cdot \Prob{Q = \ell} \\
	&\geq \Expect[ f(Q_1, \dots, Q_n) \condbar \coll{A}_n, Q = k ] \cdot \sum_{\ell = 0}^k \Prob{ Q = \ell } \\
	&\geq \frac{1}{2} \Expect[ f(\delta_1, \dots, \delta_n) \condbar \coll{A}_n ].
\end{align*}
The first inequality follows from the law of total expectation and the positivity of $f$.
The second inequality depends on the monotonicity of $f$.
The third inequality holds because $k$ is the median of the
$\poisson(k)$ distribution~\cite[Thm.~2]{Cho94:Medians-Gamma}.
Last, conditioning the Poisson variables on the event $\{Q = k\}$
produces the multinomial distribution~\cite[Thm.~5.6]{MU17:Probability-Computing-2ed}.
\end{proof}

\begin{remark}[Poissonization] \label{rem:poisson}
The analog of \Cref{lem:poisson} holds if we replace the trace exponential
with any other trace function that is both positive and monotone. %
In particular, it applies to the one-sided power function
$\mtx{M} \mapsto \trace{} (\mtx{M})_{-}^p$ for $p > 0$.
\end{remark}

\subsection{Step 4: Comparison with the Gaussian model}

We may now invoke the existing trace mgf bound (\cref{prop:psd-weight-mgf})
to compare the independent model $\mtx{X}_n$ with a suitable Gaussian matrix.

Conditional on the sample $\coll{A}_n$, define the variance function
\[
\mathsf{V}_n(\mtx{M}) \coloneqq \sum_{i=1}^n (\Expect  Q_i^2 ) \cdot \abssq{\ip{\mtx{M}}{\mtx{A}_i}}
	= \left(1 + \frac{k}{n}\right) \frac{k}{n} \sum_{i=1}^n \abssq{\ip{\mtx{M}}{\mtx{A}_i}}
	\quad\text{for $\mtx{M} \in \Sym_d$.}
\]
Indeed, since $Q_i \sim \poisson(k/n)$, its second moment is $(1+k/n)(k/n)$.
Construct the centered (conditionally) Gaussian random matrix
\begin{equation} \label{eqn:proxy-gauss}
\mtx{Z}_n \sim \normal(\mtx{0}, \mathsf{V}_n).
\end{equation}
By the law of large numbers, when the number $n$ of samples is large, $\mathsf{V}_n \approx \mathsf{V}$,
where $\mathsf{V}$ is defined in~\eqref{eqn:iid-gauss}.
As a consequence, the distribution of $\mtx{Z}_n$ is close to the distribution of
the original Gaussian model $\mtx{Z}$ with covariance $\mathsf{V}$.
\Cref{lem:gauss-weak}, below, fully justifies this claim.

To compare the trace mgfs of the proxy $\smash{\widehat{\mtx{Y}}_n}$ and the Gaussian matrix $\mtx{Z}_n$,
first apply the Poissonization result (\cref{lem:poisson}).
Then  invoke the trace mgf bound (\cref{prop:psd-weight-mgf}),
conditional on $\coll{A}_n$.  We arrive at the comparison
\begin{align} \label{eqn:iid-prelimit}
\Expect \big[ \trace \econst^{- \theta (\widehat{\mtx{Y}}_n - \Expect[ \widehat{\mtx{Y}}_n \condbar \coll{A}_n ])} \lcondbar \coll{A}_n \big]
	&\leq 2 \Expect \big[ \trace \econst^{- \theta ({\mtx{X}}_n - \Expect[ {\mtx{X}}_n \condbar \coll{A}_n ])} \lcondbar \coll{A}_n \big] \notag \\
	&\leq 2 \Expect \big[ \trace \econst^{- \theta \mtx{Z}_n } \lcondbar \coll{A}_n \big].
\end{align}
It remains to relate the random matrices $\smash{\widehat{\mtx{Y}}_n}$ and $\mtx{Z}_n$
to the target models $\mtx{Y}$ and $\mtx{Z}$.

\subsection{Step 5: Limits}

At this stage, we can unfreeze	 the random sample $\coll{A}_n$
and take limits.  This process will produce the bound
\begin{equation} \label{eqn:iid-mgf-pf}
\Expect \trace \econst^{- \theta(\mtx{Y} - \Expect \mtx{Y})}
	\leq 2 \Expect \trace \econst^{-\theta \mtx{Z}}.
\end{equation}
The random matrices $\mtx{Y}$ and $\mtx{Z}$ are defined in~\eqref{eqn:iid-sum}
and~\eqref{eqn:iid-gauss}.  The remaining steps leading to this result are technical.

First, \cref{lem:truncation} shows that it is enough to
prove~\eqref{eqn:iid-mgf-pf} under the additional
assumption that the random summand $\mtx{W}$ is bounded.
In this case, we may calculate that
\begin{align*}
\Expect \trace \econst^{-\theta (\mtx{Y} - \Expect \mtx{Y})}
	&= \lim\nolimits_{n \to \infty} \Expect \big[ \Expect \big[ \trace \econst^{- \theta (\widehat{\mtx{Y}}_n - \Expect[ \widehat{\mtx{Y}}_n \condbar \coll{A}_n ])} \lcondbar \coll{A}_n \big] \big] \\
	&\leq 2\lim\nolimits_{n \to \infty} \Expect \big[ \Expect \big[ \trace \econst^{- \theta \mtx{Z}_n} \lcondbar \coll{A}_n \big] \big]
	= 2\Expect \trace \econst^{-\theta \mtx{Z}}.
\end{align*}
The first limit follows from \cref{lem:empirical-weak}.
The second relation is the mgf bound~\eqref{eqn:iid-prelimit}.
The second limit follows from \cref{lem:gauss-weak}.
The details of these computations occupy the upcoming subsections.

\subsection{Technical Step 6: Truncation}

To continue, we restrict our attention to the setting where the
random summand $\mtx{W}$ is bounded in norm.

\begin{lemma}[Truncation] \label{lem:truncation}
Suppose that~\eqref{eqn:iid-mgf-pf} holds when the norm $\norm{\mtx{W}}$ %
is uniformly bounded. %
Then~\eqref{eqn:iid-mgf-pf} remains valid when $\Expect \norm{\mtx{W}}^2 < + \infty$.
\end{lemma}

\begin{proof}
Suppose that $\Expect \norm{\mtx{W}}^2 < + \infty$.
Let $R > 0$ be a truncation parameter.
We will consider the random matrix models arising from
the truncated summand: $\mtx{W} \indicator\{ \norm{\mtx{W}} \leq R \}$.

For iid summands $\mtx{W}_j \sim \mtx{W}$, define the coupled random matrix models
\[
\mtx{Y} \coloneqq \sum_{j=1}^k \mtx{W}_j
\quad\text{and}\quad
\mtx{Y}_{\wedge R} \coloneqq \sum_{j=1}^k \mtx{W}_j \indicator\{ \norm{\mtx{W}_j} \leq R \}.
\]
Since the trace exponential is monotone with respect to the psd order~\cite[Thm.~2.10]{Car10:Trace-Inequalities},
\[
\trace \econst^{-\theta (\mtx{Y}_{\wedge R} - \Expect \mtx{Y}_{\wedge R})}
	\leq \trace \econst^{\theta \Expect \mtx{Y} } < + \infty
	\quad\text{pointwise and for all $R > 0$.}
\]
We have used the semidefinite relations $\mtx{0} \psdle \smash{\mtx{Y}_{\wedge R}}$ and $\Expect \smash{\mtx{Y}_{\wedge R}} \psdle \Expect \mtx{Y}$, where $\psdle$ is the Loewner order.
Bounded convergence yields
\[
\lim\nolimits_{R \to \infty} \Expect \trace\econst^{-\theta (\mtx{Y}_{\wedge R} - \Expect \mtx{Y}_{\wedge R})}
	= \Expect \trace \econst^{-\theta (\mtx{Y} - \Expect \mtx{Y})}.
\]
Indeed, $\smash{\mtx{Y}_{\wedge R}} \to \mtx{Y}$ pointwise, and so $\Expect \smash{\mtx{Y}_{\wedge R}} \to \Expect \mtx{Y}$.
The limit of the expectation is confirmed by applying monotone convergence to the quadratic
forms induced by the random psd matrices.

As for the Gaussian model, define the variance functions for an argument $\mtx{M} \in \Sym_d$:
\[
\mathsf{V}(\mtx{M}) \coloneqq k \cdot \Expect\big[ \abssq{\ip{\mtx{M}}{\mtx{W}}} \big] %
\quad\text{and}\quad
\mathsf{V}_{\wedge R}(\mtx{M}) \coloneqq k \cdot \Expect\big[ \abssq{\ip{\mtx{M}}{\mtx{W}}} \cdot \indicator\{ \norm{\mtx{W}} \leq R \} \big].
\]
It is easy to see that $\smash{\mathsf{V}_{\wedge R}}(\mtx{M}) \leq \mathsf{V}(\mtx{M})$.
Therefore, the Gaussian random matrices $\smash{\mtx{Z}_{\wedge R}} \sim \normal(\mtx{0}, \smash{\mathsf{V}_{\wedge R}})$
and $\mtx{Z} \sim \normal(\mtx{0}, \mathsf{V})$ satisfy %
\[
\Expect \trace \econst^{-\theta \mtx{Z}_{\wedge R}} \leq \Expect \trace \econst^{-\theta \mtx{Z}}
\quad\text{for all $R > 0$.}
\]
This statement follows from monotonicity (\cref{prop:gauss-monotone}) for the expectation
of a convex function of a Gaussian.  Convexity of the trace exponential is an easy classical
fact~\cite[Thm.~2.10]{Car10:Trace-Inequalities}, which is also contained in Stahl's
theorem (\cref{fact:bmv}).

We conclude that
\begin{align*}
\Expect \trace \econst^{-\theta (\mtx{Y} - \Expect \mtx{Y})}
	&= \limsup\nolimits_{R \to \infty} \Expect \trace \econst^{-\theta (\mtx{Y}_{\wedge R} - \Expect \mtx{Y}_{\wedge R})} \\
	&\leq 2\limsup\nolimits_{R \to \infty} \Expect \trace \econst^{- \theta \mtx{Z}_{\wedge R}}
	\leq 2 \Expect \trace \econst^{-\theta \mtx{Z}}.
\end{align*}
The inequality in the last display depends on the hypothesis of the lemma,
namely that the relation~\eqref{eqn:iid-mgf-pf} holds when $\norm{\mtx{W}}$
is uniformly bounded.
\end{proof}

\subsection{Technical Step 7: Convergence of the empirical model}

Next, we must verify that the empirical approximations $\smash{\widehat{\mtx{Y}}_n}$
converge weakly to the original random matrix model $\mtx{Y}$.

\begin{lemma}[Empirical approximation: Weak convergence] \label{lem:empirical-weak}
Assume that the random summand has two finite moments: $\Expect \norm{\mtx{W}}^2 < + \infty$.  Define random
matrices $\mtx{Y}$ and $\smash{\widehat{\mtx{Y}}_n}$ as in~\eqref{eqn:iid-sum} and~\eqref{eqn:proxy-multi}.

\begin{enumerate}
\item	For each bounded, Lipschitz function $h : \Sym_d \to \R$,
\begin{equation} \label{eqn:empirical-weak}
\Expect h\big( \widehat{\mtx{Y}}_n - \Expect\big[ \widehat{\mtx{Y}}_n \condbar \coll{A}_n \big] \big)
	\to \Expect h(\mtx{Y} - \Expect \mtx{Y})
	\quad\text{as $n \to \infty$.}
\end{equation}

\item	If $\norm{\mtx{W}}$ is uniformly bounded, the limit~\eqref{eqn:empirical-weak} also
holds for $h(\mtx{M}) \coloneqq \trace \econst^{-\theta \mtx{M}}$ with $\theta \in \R$.
\end{enumerate}

\noindent
Lipschitz functions are defined with respect to the Frobenius norm on self-adjoint matrices.
The expectations average over everything, including the random sample $\coll{A}_n$.
\end{lemma}

\begin{proof}
To begin, let us explore some properties of the empirical approximation $\smash{\widehat{\mtx{Y}}_n}$.
The representation~\eqref{eqn:proxy-multi} shows that
\[
\widehat{\mtx{Y}}_n = \sum_{i=1}^n \delta_i^{(n)} \mtx{A}_i
\quad\text{where $\vct{\delta}^{(n)} \sim \multinomial(k, n)$.}
\]
The multinomial coefficients $\smash{\vct{\delta}^{(n)}} \coloneqq \big(\delta^{(n)}_1, \dots, \delta^{(n)}_n \big)$
are \hilite{independent} from the sample $\coll{A}_n$.  Define the
event $\set{D}_n$ where the multinomial selects $k$ distinct summands:
\[
\set{D}_n \coloneqq \big\{ \# \supp\big(\vct{\delta}^{(n)}\big) = k \big\}.
\]
For large sample size $n$, it is likely that $\set{D}_n$ occurs.
By the birthday paradox argument~\cite[Sec.~5.1]{MU17:Probability-Computing-2ed},
\begin{equation} \label{eqn:birthday}
\Probe(\set{D}_n) = \prod_{j=1}^k \left(1 - \frac{j-1}{n}\right)
	\geq 1 - \sum_{j=1}^k \frac{j-1}{n}
	> 1 - \frac{k^2}{n}.
\end{equation}
Conditional on the event $\set{D}_n$ occurring, the distribution of the empirical approximation
$\smash{\widehat{\mtx{Y}}_n}$ is the same as the distribution $\mtx{Y}$.  More precisely,
for each Borel set $\set{B} \subseteq \Sym_d$,
\[
\Prob{ \widehat{\mtx{Y}}_n \in \set{B} \lcondbar \set{D}_n }
= \Prob{ \sum_{i \in \supp(\vct{\delta}^{(n)})} \mtx{A}_i \in \set{B} \lcondbar \set{D}_n }
= \Prob{ \sum_{j=1}^k \mtx{W}_j \in \set{B} } = \Prob{ \mtx{Y} \in \set{B} }.
\]
Indeed, each sample $\mtx{A}_i$ is an independent draw from the distribution
$\mtx{W}$, as are the random matrices $\mtx{W}_1, \dots, \mtx{W}_k$.  This argument
formalizes the intuition that the empirical approximation is a good proxy for
the original random matrix.

Next, we turn to the conditional expectation of the empirical approximation.
Since each coefficient $\delta_i \sim \binomial(1/n, k)$,
\[
\Expect\big[ \widehat{\mtx{Y}}_n \condbar \coll{A}_n \big]
	= \frac{k}{n} \sum_{i=1}^n \mtx{A}_i.
\]
Each sample $\mtx{A}_i$ is an independent copy of $\mtx{W}$, so its
expectation satisfies $\Expect \mtx{A}_i = \Expect \mtx{W} = k^{-1} \Expect \mtx{Y}$.
By orthogonality,
\begin{equation} \label{eqn:proxy-mean}
\Expect \lnorm{ \Expect \big[ \widehat{\mtx{Y}}_n \condbar \coll{A}_n \big] - \Expect \mtx{Y} }_{\mathrm{F}}^2
	= \Expect \lnorm{ \frac{k}{n} \sum_{i=1}^n (\mtx{A}_i - \Expect \mtx{A}_i) }_{\mathrm{F}}^2
	= \frac{k^2}{n} \Expect \norm{ \mtx{W} - \Expect \mtx{W} }_{\mathrm{F}}^2
	\to 0.
\end{equation}
The limit occurs as $n \to \infty$, while $k$ is a fixed number.  We have used
the fact that $\mtx{W}$ has two moments.

Now, suppose that $h$ has bounded Lipschitz norm $L$.  In other words, both the uniform norm
of $h$ and the Lipschitz constant of $h$ are at most $L$.  Define the sequence of moments
\[
E_n \coloneqq \Expect h\big( \widehat{\mtx{Y}}_n - \Expect \big[ \widehat{\mtx{Y}}_n \condbar \coll{A}_n \big] \big).
\]
Add and subtract the matrix $\Expect \mtx{Y}$, and invoke the Lipschitz property:
\[
E_n = \Expect h\big( \widehat{\mtx{Y}}_n - \Expect \mtx{Y} \big)
	\pm L \cdot \Expect \lnorm{ \Expect\big[\widehat{\mtx{Y}}_n \condbar \coll{A}_n \big] - \Expect \mtx{Y} }_{\mathrm{F}}.
\]
The notation $x = a \pm b$ is shorthand for the pair of inequalities $a - b \leq x \leq a + b$.
From~\eqref{eqn:proxy-mean}, we see that the second term on the right-hand side of the last display tends to zero.
As for the first term, we compute the expectation by conditioning on the event $\set{D}_n$:
\begin{align*}
\Expect h\big(\widehat{\mtx{Y}}_n - \Expect \mtx{Y} \big)
	&= \Expect\big[ h\big( \widehat{\mtx{Y}}_n - \Expect \mtx{Y} \big) \lcondbar \set{D}_n \big] \cdot \Probe(\set{D}_n)
	+ \Expect\big[ h\big( \widehat{\mtx{Y}}_n - \Expect \mtx{Y} \big) \lcondbar \set{D}_n^\comp \big] \cdot \Probe(\set{D}_n^\comp) \\
	&= \Expect\big[ h(\mtx{Y} - \Expect \mtx{Y}) \big] \pm \frac{2k^2 L}{n}
	\to \Expect h(\mtx{Y} - \Expect \mtx{Y}).
\end{align*}
Indeed, the conditional distribution $\smash{\widehat{\mtx{Y}}_n} \condbar \set{D}_n \sim \mtx{Y}$.
The inequalities depend on two applications of the probability estimate~\eqref{eqn:birthday}
and the fact that the magnitude of $h$ is uniformly bounded by $L$.  Altogether,
\[
E_n = \Expect h \big( \widehat{\mtx{Y}}_n - \Expect\big[ \widehat{\mtx{Y}}_n \condbar \coll{A}_n \big] \big)
	\to \Expect h(\mtx{Y} - \Expect \mtx{Y})
	\quad\text{as $n \to \infty$.}
\]
We have established the weak convergence claim.

Last, assume that the random summand satisfies $\norm{\mtx{W}} \leq R$
for some $R > 0$.
In this case, the random matrices $\smash{\widehat{\mtx{Y}}_n}$ and $\mtx{Y}$
are all bounded in norm by $kR$.  The trace exponential function
$h(\mtx{M}) = \trace \econst^{-\theta \mtx{M}}$ is bounded and
Lipschitz on the common support of these random matrices.
Therefore, the weak convergence result~\eqref{eqn:empirical-weak}
ensures that
\[
\Expect \trace \econst^{-\theta (\widehat{\mtx{Y}}_n - \Expect [\widehat{\mtx{Y}}_n \condbar \coll{A}_n])}
	\to \Expect \trace \econst^{-\theta (\mtx{Y} - \Expect \mtx{Y})}
	\quad\text{as $n \to \infty$.}
\]
This is the second conclusion.
\end{proof}

\subsection{Technical Step 8: Convergence of the Gaussian model}

Finally, we must argue that the sequence \smash{$\mtx{Z}_n$} of Gaussian models
converges weakly to the target Gaussian distribution $\mtx{Z}$.
The proof relies on characteristic functions.

\begin{lemma}[Gaussian model: Weak convergence] \label{lem:gauss-weak}
Assume that the random summand has two finite moments: $\Expect \norm{\mtx{W}}^2 < + \infty$.
Define Gaussian matrices $\mtx{Z}$ and $\mtx{Z}_n$ as in~\eqref{eqn:iid-gauss}
and~\eqref{eqn:proxy-gauss}.

\begin{enumerate}
\item	For each bounded Lipschitz function $h : \Sym_d \to \R$,
\begin{equation} \label{eqn:gauss-weak}
\Expect h(\mtx{Z}_n) \to \Expect h(\mtx{Z})
\quad\text{as $n \to \infty$.}
\end{equation}

\item	If $\norm{\mtx{W}}$ is uniformly bounded, the limit~\eqref{eqn:gauss-weak} also holds
for $h(\mtx{M}) \coloneqq \trace \econst^{-\theta \mtx{M}}$ with $\theta \in \R$.
\end{enumerate}

\noindent
The expectation averages over everything, including the random sample $\coll{A}_n$.
\end{lemma}

\begin{proof}
Consider a self-adjoint Gaussian matrix $\mtx{G} \sim \normal(\mtx{0}, \set{C})$.
For a self-adjoint %
argument $\mtx{M} \in \Sym_d$, the characteristic function $\chi_{\mtx{G}}$
of the Gaussian $\mtx{G}$ takes the form
\[
\chi_{\mtx{G}}(\mtx{M}) \coloneqq \Expect \econst^{\iunit \ip{\mtx{M}}{\mtx{G}}}
	= \econst^{ - \mathsf{C}(\mtx{M})/2 }.
\]
This just reinterprets the usual formula for the
Gaussian characteristic function~\cite[Sec.~9.5]{Dud02:Real-Analysis}. %

Now, the centered Gaussian matrices $\mtx{Z}_n$ and $\mtx{Z}$ have covariance
functions
\[
\mathsf{V}_n(\mtx{M}) \coloneqq \left(1 + \frac{k}{n} \right)\frac{k}{n} \sum_{i=1}^n \abssq{\ip{\mtx{M}}{\mtx{A}_i}}
\quad\text{and}\quad
\mathsf{V}(\mtx{M}) \coloneqq k \cdot \Expect \abssq{\ip{\mtx{M}}{\mtx{W}} }.
\]
Since $\mathsf{V}_n$ is a random operator that depends on the sample $\coll{A}_n$,
the random matrix $\mtx{Z}_n$ is actually a Gaussian mixture.  Each random sample
$\mtx{A}_i$ is an independent copy of $\mtx{W}$.  Therefore,
\begin{equation} \label{eqn:cov-lln}
\mathsf{V}_n(\mtx{M}) = \left(1 + \frac{k}{n}\right) \frac{k}{n} \sum_{i=1}^n \abssq{\ip{\mtx{M}}{\mtx{A}_i}}
	\to k \cdot \Expect \abssq{\ip{\mtx{M}}{\mtx{W}} }
	\quad\text{almost surely.}
\end{equation}
The formula~\eqref{eqn:cov-lln} is just the strong law of
large numbers~\cite[Thm.~8.3.5]{Dud02:Real-Analysis}. %
The statement is justified by the fact that $\Expect \norm{ \mtx{W} }^2$ is finite.

Recall that a sequence of random matrices converges weakly if and only if
the characteristic functions converge pointwise to a limit that is continuous
at the origin~\cite[Thm.~9.8.2]{Dud02:Real-Analysis}.
Therefore, to prove the weak convergence statement~\eqref{eqn:gauss-weak},
it suffices to verify that
\[
\chi_{\mtx{Z}_n}(\mtx{M}) \to \chi_{\mtx{Z}}(\mtx{M})
	\quad\text{for each $\mtx{M} \in \Sym_d$.}
\]
Equivalently, we can obtain weak convergence from the limit
\[
\Expect \econst^{ - \mathsf{V}_n(\mtx{M}) / 2} 
	\to \econst^{ - \mathsf{V}(\mtx{M}) / 2}
	\quad\text{for each $\mtx{M} \in \Sym_d$.}
\]
But this statement follows instantly from bounded convergence and the
almost sure limit~\eqref{eqn:cov-lln}.  Indeed, variance functions are
positive, so the exponentials are uniformly bounded by one.

Last, assume that the random summand satisfies $\norm{\mtx{W}} \leq R$ for some $R > 0$.
We must upgrade the weak convergence~\eqref{eqn:gauss-weak} to convergence
for the trace mgf function $h(\mtx{M}) \coloneqq \trace \econst^{-\theta \mtx{M}}$
with $\theta \in\R$.
This step requires asymptotic uniform integrability~\cite[Thm.~2.20]{vdV98:Asymptotic-Statistics}.
We \hilite{claim} that
\begin{equation} \label{eqn:gauss-aui}
\lim\nolimits_{B \to \infty} \limsup\nolimits_{n \to \infty}
	\Expect\big[ \trace \econst^{-\theta \mtx{Z}_n} \indicator\{ \norm{\mtx{Z}_n} \geq B \} \big]
	\to 0.
\end{equation}
Granted~\eqref{eqn:gauss-aui}, we obtain the limit
$\Expect h(\mtx{Z}_n) \to \Expect h(\mtx{Z})$, and the proof is complete.

Let us establish the claim.  Since $\norm{\mtx{W}} \leq R$,
the covariance operators $\mathsf{V}_n$ are bounded. %
More precisely, for $n \geq k$, since $\norm{\mtx{A}_i} \leq R$ as well,
\[
\mathsf{V}_n(\mtx{M})
	= \left(1 + \frac{k}{n}\right) \frac{k}{n} \sum_{i=1}^n \abssq{\ip{\mtx{M}}{\mtx{A}_i}}
	\leq 2k R^2 d \fnormsq{\mtx{M}} \eqqcolon \mathsf{C}(\mtx{M}). %
\]
Define the Gaussian random matrix $\mtx{G} \sim \normal(\mtx{0}, \mathsf{C})$.
By monotonicity (\cref{prop:gauss-monotone}), applied conditionally on $\coll{A}_n$, 
the comparison $\Expect f(\mtx{Z}_n) \leq \Expect f(\mtx{G})$ holds
for each convex function $f : \Sym_d \to \R$.

To confirm asymptotic uniform integrability~\eqref{eqn:gauss-aui}, make a simple bound on the trace exponential.
Then apply the inequalities of Cauchy--Schwarz and Markov to the expectation in~\eqref{eqn:gauss-aui} to obtain
\begin{align*}
	\Expect\big[ \trace \econst^{-\theta \mtx{Z}_n } \indicator\{ \norm{\mtx{Z}_n} \geq B \} \big]
	&\leq d \cdot \Expect\big[ \econst^{\abs{\theta} \, \norm{ \mtx{Z}_n} } \indicator\{ \norm{\mtx{Z}_n} \geq B \} \big] \\
	&\leq d \cdot \left( \Expect \econst^{2 \abs{\theta} \, \norm{\mtx{Z}_n}} \right)^{1/2} (B^{-1} \Expect \norm{\mtx{Z}_n})^{1/2} \\
	&\leq d \cdot \left( \Expect \econst^{2 \abs{\theta} \, \norm{\mtx{G}}} \right)^{1/2} (B^{-1} \Expect \norm{\mtx{G}})^{1/2}.
\end{align*}
The last inequality follows from the comparison of $\mtx{Z}_n$ with $\mtx{G}$.
Since $\mtx{G}$ is a fixed Gaussian matrix, the two expectations with respect to $\mtx{G}$ are finite,
and the asymptotic uniform integrability condition~\eqref{eqn:gauss-aui} is valid.  
\end{proof}

\subsection{Extension: Polynomial moments}

The proof of the trace mgf comparison~\eqref{eqn:iid-mgf-pf}
can be adapted to obtain a comparison theorem for polynomial moments.

\begin{proposition}[iid psd sum: Polynomial moment bound] \label{prop:iid-poly}
Instate the hypotheses of \cref{thm:iid-sum}.  For each $p \geq 4$,
\[
\Expect \trace{} (\mtx{Y} - \Expect \mtx{Y} + \mtx{\Delta})_-^p
	\leq 2 \Expect \trace{} (\mtx{Z} + \mtx{\Delta})_-^p.
\]
\end{proposition}

\begin{proof}[Proof sketch] %
Note that the trace function $f(w) \coloneqq \trace{} (w\mtx{A} - \mtx{B})_{-}^p$
is completely monotone of order four when $p \geq 4$.
Therefore, we can activate the Poissonization result (\cref{rem:poisson})
and the polynomial moment bound for weighted sums (\cref{prop:psd-weights-poly}).
These are the main changes, as compared with the proof of~\eqref{eqn:iid-mgf-pf}.
There are also some minor technical differences in
the proofs of \cref{lem:truncation,lem:empirical-weak,lem:gauss-weak}.
\end{proof}

\appendix

\section{Matrix concentration for psd sums}
\label{app:epz}

For completeness, we include a short proof of the matrix concentration
results for the lower tail of a psd sum.  This result was communicated
to the author by Andreas Maurer in 2011, but there does not seem to be
a citation available.

\begin{fact}[Matrix concentration: Exponential Paley--Zygmund] \label{fact:matrix-epz}
Consider an \hilite{independent} family $(\mtx{W}_1, \dots, \mtx{W}_n)$ of random \hilite{psd} matrices
with common dimension $d$ and with two finite moments.
Define the matrix sum and the sum of second moments:
\[
\mtx{Y} \coloneqq \sum_{i=1}^n \mtx{W}_i
\quad\text{and}\quad
L_2 \coloneqq \lnorm{ \sum_{i=1}^n \Expect \mtx{W}_i^2 }.
\]
Then
\begin{equation} \label{eqn:epz-expect}
\Expect \lambda_{\min}(\mtx{Y}) \geq \lambda_{\min}(\Expect \mtx{Y}) - \sqrt{2 L_2 \log d}.
\end{equation}
In addition, for $t \geq 0$,
\begin{equation} \label{eqn:epz-tail}
\Prob{ \lambda_{\min}(\mtx{Y}) \leq \lambda_{\min}(\Expect \mtx{Y}) - t }
	\leq d \cdot \econst^{-t^2 / (2L_2)}.
\end{equation}
\end{fact}

To identify the connection between \cref{fact:matrix-epz} and the Paley--Zygmund
inequality~\cite[Exer.~2.4]{BLM13:Concentration-Inequalities},
select $t = \eps \cdot \lambda_{\min}(\Expect \mtx{Y})$ in the tail bound~\eqref{eqn:epz-tail}.
We obtain a ratio of the squared norm of the first moment to the norm of the sum of second moments.
The proof depends on a trace mgf bound.

\begin{lemma}[Exponential Paley--Zygmund: log-mgf bound] \label{lem:epz-cgf}
Let $\mtx{W}$ be a random psd matrix with two finite moments.  For $\theta \geq 0$,
we have the semidefinite relation
\[
\log \Expect \econst^{-\theta \mtx{W}}
	\psdle \theta (\Expect \mtx{W}) + \tfrac{1}{2} \theta^2 (\Expect \mtx{W}^2).
\]
\end{lemma}

\begin{proof}
Recall the numerical inequality $\econst^{-a} \leq 1 - a + a^2 / 2$, valid for $a \geq 0$.
By the transfer rule~\cite[Prop.~2.1.4]{Tro15:Introduction-Matrix},
the inequality extends to matrices:
\[
\Expect \econst^{-\theta \mtx{W}}
	\psdle \Expect[ \Id - \theta \mtx{W} + \tfrac{1}{2} \theta^2 \mtx{W}^2 ]
	= \Id - \theta (\Expect \mtx{W}) + \tfrac{1}{2} \theta^2 (\Expect \mtx{W}^2)
	\quad\text{for $\theta \geq 0$.}
\]
Since the logarithm is matrix monotone~\cite[Prop.~8.4.4]{Tro15:Introduction-Matrix},
we can extract the logarithm.  Apply the numerical inequality $\log(1 + a) \leq a$
for $a > - 1$ using the transfer rule.
\end{proof}

\begin{proof}[Proof of \cref{fact:matrix-epz}]
The result follows quickly from standard matrix concentration arguments.
For example, to derive the probability inequality, we apply the matrix Laplace transform
method~\cite[Prop.~3.2.1]{Tro15:Introduction-Matrix}.  For each $\theta > 0$,
\begin{align*}
\Prob{ \lambda_{\min}(\mtx{Y} - \Expect \mtx{Y}) \leq -t }
	&\leq \econst^{-\theta t} \Expect \trace \exp( -\theta \mtx{Y} + \theta \Expect \mtx{Y} ) \\
	&\leq \econst^{-\theta t} \trace \exp\left( \sum_{i=1}^n \big( \log \Expect \econst^{-\theta \mtx{W}_i} + \theta  \Expect \mtx{W}_i \big) \right) \\
	&\leq \econst^{-\theta t} \trace \exp\left( \tfrac{1}{2} \theta^2 \sum_{i=1}^n \Expect \mtx{W}_i^2 \right)
	\leq d \cdot \econst^{-\theta t + \theta^2 L_2 / 2}.
\end{align*}
The second inequality requires the subadditivity of matrix log-mgfs~\cite[Lem.~3.1]{Tro15:Introduction-Matrix}
and the monotonicity of the trace exponential.
The third inequality is \cref{lem:epz-cgf}.  The last inequality depends on the
spectral mapping theorem and the definition of $L_2$.  Select $\theta = t/ (2L_2)$
to complete the bound.  The stated result follows from an application
of Weyl's inequality~\cite[Cor.~III.2.2]{Bha97:Matrix-Analysis}.
\end{proof}

\section*{Acknowledgements}

The author thanks Ethan Epperly, Jorge Garza-Vargas, Otte Hein{\"a}vaara, De Huang,
Raphael Meyer, Ramon van Handel, Roman Vershynin, and Rob Webber for valuable discussions and feedback.
This research was supported in part by NSF FRG Award 1952777, ONR Award N-00014-24-1-2223,
and the Caltech Carver Mead New Adventures Fund.

\printbibliography

\end{document}